\documentclass[a4paper, 12pt]{amsart}

\usepackage{amsmath,amsthm,amssymb,mathtools,amscd}
\usepackage{latexsym,graphicx,color,indentfirst,tikz}
\usepackage[bottom=1.5in]{geometry}
\usepackage[linesnumbered, ruled, norelsize]{algorithm2e}
\usepackage{fancyhdr}

\makeatletter \@addtoreset{equation}{section} \makeatother
\newtheorem{theorem}{Theorem}[section]
\newtheorem{lemma}[theorem]{Lemma}
\newtheorem{assumption}[theorem]{Assumption}

\newtheorem{proposition}[theorem]{Proposition}
\newtheorem{definition}[theorem]{Definition}
\newtheorem{remark}[theorem]{Remark}

\newcommand{\ajj}{{\alpha_{j+1}}}

\newcommand{\astar}{\alpha^*} 
\newcommand{\brd}{d}
\newcommand{\bra}{{d_\alpha}}
\newcommand{\br}[1]{{d_{#1}}}
\newcommand{\Bb}[2]{{\mathcal{B}_{#1} (#2)}}
\newcommand{\Bbra}{{\mathcal{B}_{d_\alpha} (\xad)}}

\newcommand{\cra}{r_\alpha}
\newcommand{\crr}[1]{{r_{#1}}}
\newcommand{\js}{{j^*}}

\newcommand{\ksj}{{k^*(j)}}
\newcommand{\ksjs}{{k^*}} 
\newcommand{\nrho}{s}
\newcommand{\qa}{\bar q}
\newcommand{\Rg}{\mathop{\mathrm{Rg}}}
\newcommand{\x}[2]{{x_{#1, #2}}}
\newcommand{\xj}[1]{{x_{j,#1}}}
\newcommand{\xjk}{{x_{j,k}}}
\newcommand{\xjks}{{x_{j, k^*}}}
\newcommand{\xjsks}{{x_{\js, \ksjs}}} 
\newcommand{\xad}{{x_\alpha^\delta}}
\newcommand{\xajd}{{x_{\alpha_j}^\delta}}
\newcommand{\xajjd}{{x_{\ajj}^\delta}}
\newcommand{\xbd}{x_{\bar \alpha}^\delta}

\def\endproof{\hspace*{\fill} \qed \vspace{5mm}}

\def\eps{\varepsilon}
\def\geq{\geqslant}

\def\IN{\mathbb{N}}
\def\IR{\mathbb{R}}
\def\l{\left}

\def\leq{\leqslant}
\def\r{\right}
\def\st{~ : ~}
\def\xd{{x^\dag}}
\def\xs{{\bar x_0}}
\def\yd{y^\delta}

\newcommand{\start}{1}

\title[Global minimization of Tikhonov functionals]{A global minimization algorithm for Tikhonov functionals with sparsity constraints}

\author{Wei Wang}
\address[W.~Wang]{College of Mathematics, Physics and Information Engineering, Jiaxing University, Zhejiang 314001, People's Republic of China}
\address{Department of Mathematics, Harbin Institute of Technology, Harbin, Heilongjiang, 150001,People's Republic of China}
\email{\tt weiwangmath@gmail.com}

\author{Stephan W. Anzengruber}
\address[S.W.~Anzengruber]{Department of Mathematics, Technische Universit\"{a}t Chemnitz,
  09107 Chemnitz,  Germany}
\email[Corresponding author]{\tt stephan.anzengruber@mathematik.tu-chemnitz.de}

\author{Ronny Ramlau}
\address[R.~Ramlau]{Johannes Kepler University, Industrial Mathematics Institute, Altenbergerstra{\ss}e 69, 4040 Linz, Austria}
\address{Johann Radon Institute for Computational and Appied Mathematics, Altenbergerstra{\ss}e 69, 4040 Linz, Austria}
\email{\tt ronny.ramlau@jku.at}

\author{Bo Han}
\address[B.~Han]{Department of Mathematics, Harbin Institute of Technology, Harbin, Heilongjiang, 150001,People's Republic of China}
\email{\tt bohan@hit.edu.cn}

\keywords{Inverse problems, Tikhonov regularization, global optimization, gradient descent method, TIGRA, parameter choice rules}

\subjclass[2010]{65J20, 47J06, 47A52, 49J40}

\date{}

\begin{document}

\begin{abstract}
  In this paper we present a globally convergent algorithm for the computation of a minimizer of the Tikhonov functional with sparsity promoting penalty term for nonlinear forward operators in Banach space. The \emph{dual TIGRA} method uses a gradient descent iteration in the dual space at decreasing values of the regularization parameter $\alpha_j$, where the approximation obtained with $\alpha_j$ serves as the starting value for the dual iteration with parameter $\ajj$. With the discrepancy principle as a global stopping rule the method further yields an automatic parameter choice.
  We prove convergence of the algorithm under suitable step-size selection and stopping rules and illustrate our theoretic results with numerical experiments for the nonlinear autoconvolution problem.
\end{abstract}

\maketitle

\section{Introduction}

  Tikhonov regularization has become a well-established and widely used method for the stable solution of ill-posed problems.
  Its regularizing properties as well as convergence and rates of convergence of the approximate solutions have been studied in great detail. In particular, the theory for the classical formulation, where the penalty term is the square of the norm in a Hilbert space, is well developed (see, e.g.~\cite{ehn96}). Over the last decade, there has been a growing interest in more general convex penalty terms in Banach spaces as well. Using a total variation type penalty \cite{av94, cl97, m01, rof92}, for example, allows for the reconstruction of sharp edges in images, whereas $\ell^p$-norms of basis coefficients with $p<2$ \cite{ddd04, rt06, r08} are known to promote sparsity in the solution -- a desireable effect in many applications.  We refer the interested reader to the monographs \cite{sgghl09, skhk12} and to the incomplete list \cite{ahm13, ar10, ar11, bo04, ghs08, hkps07, hm12, rs06} for further information in this direction.

  Open questions remain, however, concerning the computational aspects of finding a minimizer of the Tikhonov functional. Especially, if the forward operator is non-linear many optimization routines will only converge locally and it may not be possible to obtain a global minimizer. Several attempts have been made to tackle this problem in the classical Hilbert space setting. In \cite{r02b, r03} the TIGRA (TIkhonov-GRAdient) algorithm was introduced which applies a gradient descent method to the Tikhonov functional at a decreasing sequence of regularization parameters $\alpha_j = \qa^j \alpha_0$. This algorithm was further studied in \cite{k10-1, k10-2, k11}. Another approach to obtain a global minimizer are multi-level methods \cite{b06, b08, hqs12}, which minimize at different discretization levels using a gradient descent method for the least-squares functional.

  In this paper, we will present a globally convergent minimization routine to obtain sparse reconstructions.  To promote sparsity in the solution with respect to a given (Schauder) basis, we represent Banach space elements in terms of their coefficents and consider a nonlinear operator equation
    \[ F(x) = y, \qquad y \in \Rg(F), \]
  between $X = \ell^p$ and a Hilbert space $Y$. Suppose that we are given only noisy data $\yd$ satisfying $\| y - \yd \|_Y \leq \delta$, then as regularized solutions $\xad$ we choose the minimizers of the generalized Tikhonov functional with $\ell^p$-penalty,
  \begin{equation}\label{Phi}
    \Phi_{\alpha}(x):=\frac{1}{2}\|F(x)-\yd\|^{2} + \frac{\alpha}{p} \|x\|_{\ell^p}^p, \qquad 1 < p \leq 2,
  \end{equation}
  where
    \[ \|x\|_{\ell^p}^p = \sum_{i \in \IN} |x_i|^p. \]

  The algorithm we propose to obtain a minimizer is a generalization of the TIGRA algorithm to Banach spaces, which uses a gradient descent iteration in the dual space and will thus be called \emph{dual TIGRA} or simply \emph{d-TIGRA}. To the best of our knowledge, this is the first time a globally convergent algorithm is presented for the Tikhonov functional \eqref{Phi} with non-linear forward operator. Gradient descent methods in the dual space with linear $F$ have already been considered in \cite{bkmss08}, for example. 
  The global convergence of d-TIGRA results from a directional convexity of the Tikhonov functional $\Phi_\alpha (x)$ in a ball around a global minimizer $\xad$ and the size of this {\it region of convexity} grows unboundedly as $\alpha \to \infty$. Conceptually, the d-TIGRA algorithm goes as follows.

  \medskip

  \begin{itemize}
    \item Fix the starting value $\xs \in X$ and find $\alpha_0$ large enough such that $\xs$ belongs to the region of convexity of $\Phi_{\alpha_0} (x)$.
    \item The dual gradient descent iterates converge to a minimizer of $\Phi_{\alpha_0} (x)$ provided that the step-sizes and stopping rule are suitably chosen.
    \item Once the iteration stops, reduce the value of $\alpha$ by a factor $\qa < 1$ such that the last iterate again belongs to the region of convexity for $\alpha_{j+1} = \qa \alpha_j$.
  \end{itemize}

  \medskip

   \noindent In this fashion, we continue with the dual gradient descent method as an inner iteration and to reduce $\alpha_j$ by a factor $\qa$ in the outer iteration until the discrepancy principle $\| F(\xjsks) - \yd \| \leq \tau \delta$ is satisfied, where $\xjsks$ denotes the last iterate with regularization parameter $\alpha_\js$.
  This algorithm generates a sequence $\{ x_{j, k} \}_{j \leq \js, k \leq \ksj}$ that converges towards a global minimizer of $\Phi_{\alpha_\js} (x)$ and thus also provides an automatic choice of the regularization parameter, which turns out to give near optimal estimates with respect to the Bregman distance and performs favorably in numerical experiments.

 The paper is organized as follows: In Section \ref{sec:prelim} we introduce basic concepts and notations, formulate our standing assumptions and give general regularization results.
 Then the main theoretical results are presented in Section \ref{sec:duallwb}, where we obtain the region of convexity of the Tikhonov functional, we ensure that any starting value $\xs$ belongs to this region of convexity for sufficiently large $\alpha$, we give a convergence analysis for the dual gradient descent iteration, and we show that the final iterate $\xjks$ at level $j$ can be used as a starting point for the dual gradient descent method at level $j+1$ with $\ajj = \qa \alpha_j$. Based on these findings, we summarize our assumptions on the parameters and prove the global convergence result for the dual TIGRA method in Section \ref{sec:globalmin}, which we illustrate in Section \ref{sec:numerics} with a numerical example. Several technical proofs are collected in Section \ref{sec:proofs} and, finally, we have included an overview over several constants that appear in the analysis at the end, for the convenience of the reader.

\section{Preliminaries} \label{sec:prelim}

  Let $(Y, \|.\|_Y)$ be a real Hilbert space and $F : \ell^p (\IN) \to Y$ be a non-linear operator, $1<p\leq 2$. We assume that $F$ is sequentially closed in the weak topologies of $\ell^p$ and $Y$.  Whenever there is no risk of confusion, we will denote the norms in $\ell^p$, $\ell^q = (\ell^p)^*$ (where $1/p+1/q=1$) and $Y$ simply by $\|.\|$. Similarly, the duality product $\langle ., . \rangle_{\ell^q \times \ell^p}$ as well as the scalar product in $Y$ are denoted by $\langle ., . \rangle$.

  The Bregman distance is a powerful tool for regularization theory in Banach spaces and it will also play a key role in our analysis.

\begin{definition} \label{dfn:bregman}
  Let $X$ be a Banach space, $f: X \to \IR$ be a convex functional and denote by $\partial f (x)$ the subdifferential of $f$ at $x \in X$. The Bregman distance $D_f^\xi (z,x)$ of two elements $x,z\in X$ with respect to $\xi \in \partial f (x)$ is defined by
    $$D_f^\xi (z,x) := f(z) - f(x) - \langle \xi, z - x \rangle_{X^* \times X}.$$
\end{definition}

Clearly, $D_f^\xi (x,x)=0$ and the convexity of $f$ further implies that $D_f^\xi (z, x) \geq 0$ for all $\xi \in \partial f (x)$. For strictly convex functionals, we have that $D_f^\xi (x,z)=0$ if and only if $x=z$.
If the functional $f$ is Fr\'echet differentiable at $x \in X$, then $\partial f (x) = \{ \nabla f(x) \}$ and we will write $D_f (z,x)$ omitting the dependence of the Bregman distance on $\xi = \nabla f(x)$.

All of the above is, indeed, the case for the $\ell^p$ functionals
    $$ f_p(x) := \frac{1}{p}\|x\|_{\ell^p}^p=\frac{1}{p}\sum_{i \in \IN} |x_{i}|^{p}, \qquad 1 < p \leq 2,$$
which are of main interest to us here. The Bregman distance then reads
    $$D_{f_p}(z,x):=\frac{1}{p} \|z\|^p - \frac{1}{p} \|x\|^p - \left\langle \nabla f_p(x), z - x \right\rangle, \quad x,z\in \ell^p,$$
where
    \[ \nabla f_p(x) = \{ |x_i|^{p-1}{\rm sgn}(x_i) \}_{i \in \IN} \in \ell^q, \]
and the latter is closely linked to the duality mapping in $\ell^p$ with gauge function $t \mapsto t^p$.

\begin{definition}
  The set-valued duality mapping $J_p$ with $p>1$ of a Banach space $X$ into $X^*$ is defined by
  \begin{align*}
    J_p(x) := \{ x^*\in X^* \st \langle x^*,x\rangle_{X^* \times X} = \|x\|_X \cdot \|x^*\|_{X^*}, \|x^*\|_{X^*} = \|x\|_X^{p-1} \}.
  \end{align*}
\end{definition}

  From here on, we consider $X = \ell^p (\IN)$ with fixed $1<p\leq 2$, where the duality mappings are single valued,
    $$J_p (x) = \left(\nabla\frac{1}{p}\| \cdot \|_{\ell^p}^p\right)(x) = \nabla f_p (x),$$
  and we slightly abuse notation identifying the set $J_p(x)$ with its unique element. Hence,
  \begin{equation} \label{eq:duality}
     J_q(J_p(x)) = x \qquad \mbox{and} \qquad \| J_p (x) \|_{X^*} = \|x\|_X^{p-1}
  \end{equation}
  hold, and the Bregman distance $D_{f_p}(z, x)$ can be equivalently expressed as
  \begin{equation} \label{eq:Bregd}
  \begin{aligned}
    D_{f_p}(z, x) 
      & = \frac{1}{q} \|x\|^{p} + \frac{1}{p} \|z\|^{p} -\langle J_{p}(x),z\rangle,\\
      & = \frac{1}{q} \|J_p(x)\|^{q} - \frac{1}{q} \|J_p (z)\|^{q} -\langle J_{p}(x)-J_p(z),z\rangle\\
      & = D_{f_q} (J_p(x), J_p(z))
  \end{aligned}
  \end{equation}
for all $x, z \in \ell^p$.

  In addition, the following important inequalities hold in $\ell^p$ which is $2$-convex and $p$-smooth. They are based on results in \cite{xr91, bi00} and were also used in \cite{l08, rr10}, for example.

\begin{lemma}\label{lemma1}
  Let $1<p\leq 2$, then there exists a constant $\tilde{c}_p > 0$ such that
  \begin{equation} \label{Dfsp}
      D_{f_p}(z,x)\leq \tilde{c}_{p}\|x-z\|^{p},
  \end{equation}
  holds for all $x,z \in \ell^p$. For $c_1, c_2>0$ and $x, z \in \ell^p$ satisfying $\|x\|\leq c_{1}$ and $\|x-z\|\leq c_{2}$ it holds that
    \begin{equation} \label{Dfl}
      D_{f_p}(z,x)\geq c_{p}\|x-z\|^{2}, 
    \end{equation}
  with $c_{p}=\frac{(p-1)}{2}(c_{1}+c_{2})^{p-2}$.
\end{lemma}

\begin{proof}
  We first note, that \eqref{Dfsp} holds true in any $p$-smooth Banach space.

  The proof of \eqref{Dfl} is based on the estimate from \cite[Lemma 1.4.6]{bi00}
  \begin{equation}\label{ADf1}
    \frac{1}{p} ( \|x-z\| + \|z\| )^p - \frac{1}{p} \|z\|^p \leq \|z-x\| \|x\|^{p-1} - \langle J_p (x),z-x\rangle,
  \end{equation}
  as well as on the mean value theorem for $\phi(t)=(t+\|x\|)^{p}$ in the form, 
    \[ \phi(t) = \|x\|^p + p\|x\|^{p-1} t + \int_{0}^{t} p(p-1) (\tau+\|x\|)^{p-2} (t-\tau) d\tau. \]
  Indeed, for $\|x\|\leq c_1$ and $t = \|x-z\| \leq c_2$, we obtain
  \begin{align*}
    D_{f_{p}}(z,x) & \geq \frac{1}{p} \phi (\|x-z\|) - \frac{1}{p} \|x\|^p - \|x\|^{p-1} \|x-z\|
      \geq \frac{p-1}{2} (c_1 + c_2)^{p-2} \|x-z\|^{2}.
  \end{align*}
  The assertion follows with $c_{p}=\frac{(p-1)}{2}(c_{1}+c_{2})^{p-2}$.
\end{proof}

  On the other hand, in $X^* = \ell^q$ with $q \geq 2$ which is $q$-convex and $2$-smooth, the following inequalities hold. The proof is again based on \cite[Lemma 1.4.6]{bi00} and goes as the proof of Lemma \ref{lemma1} with the inequality signs reversed.

\begin{lemma} \label{lemma2}
  For $q \geq 2$, there exists a  constant $c_q>0$ such that
    \begin{equation} 
      D_{f_q}(z,x)\geq c_{q} \|x-z\|^{q},
    \end{equation}
  holds for all $x,z\in \ell^q$. For $c_1, c_2>0$ and $x, z \in \ell^p$ satisfying $\|x\|\leq c_{1}$ and $\|x-z\|\leq c_{2}$ it holds that
    \begin{equation} \label{Dfsq}
      D_{f_q}(z,x)\leq \tilde{c}_{q} \|x-z\|^{2},
    \end{equation}
  with $\tilde{c}_{q}=\frac{(q-1)}{2}(c_{1}+c_{2})^{q-2}$.
\end{lemma}

Throughout this paper, we denote by $\xad$ a minimizer of the Tikhonov functional $\Phi_\alpha (x)$ in \eqref{Phi} with noisy data $\yd$ satisfying $\| y - \yd \| \leq \delta$ and regularization parameter $\alpha > 0$. Here, the given data $\yd$ as well as the noise level $\delta$ are assumed to be fixed. Existence of minimizers as well as stability results for the regularized problems are well-known and we refer to \cite{sgghl09, skhk12} for details. By $\xd$ we denote a $f_p$-minimizing solution of the original problem $F(x) = y$, in the sense that $F(\xd)=y$ and
    $$f_p(\xd)=\min_{ x\in D(F)} \{ f_p(x) \st F(x)=y\}.$$

Let us now summarize our standing assumptions on the forward operator $F$ and a source condition on the $f_p$-minimizing solution $\xd$.

\begin{assumption} \label{assp_main} The operator $F(x)$ is Fr\'echet differentiable on $\ell^p$ and the following holds:
  \begin{enumerate}
    \item \label{ass:lip}  $F'$ is Lipschitz continuous, i.e.$~$there exists $L > 0$ such that
      \begin{equation*}
	\|F'(x)-F'(z)\|_{\mathcal{L}(\ell^p, Y)} \leq L\|x-z\|
      \end{equation*}
      holds for all $x,z\in \ell^p$. 
    \item \label{ass:nl} There exists $c>0$ such that
      \begin{equation*}
	  \|F(x)-F(z)-F'(z)(x-z)\|\leq c D_{f_p}(x,z)
      \end{equation*}
      holds for all $x,z\in \ell^p$. 
  \item \label{ass:src} There exist $\nrho > 0$, $\varrho < \frac{1}{\nrho c}$, a $f_p$-minimizing solution $\xd \in X$ and $\omega\in Y$ such that $\|\omega\|\leq \varrho$ and
      \begin{equation*}
	  J_p (\xd) = F'(\xd)^{*} \omega.
      \end{equation*}
  \end{enumerate}
\end{assumption}

\begin{remark} \label{rem:nl}
  It is well known that Lipschitz continuity of the Fr\'echet derivative yields
    \begin{equation} \label{eq:lipnl}
      \|F(x)-F(z)-F'(z)(x-z)\| \leq \frac{L}{2} \| x- z \|^2
    \end{equation}
  and thus, on bounded domains, \ref{ass:lip} and \ref{ass:nl} in Assumption \ref{assp_main} are also related via Lemma~\ref{lemma1}.
\end{remark}

The {\it nonlinearity condition} \ref{ass:nl} in combination with the {\it source condition} \ref{ass:src} was first introduced by Resmerita and Scherzer in \cite{rs06} to obtain convergence rates for Tikhonov regularization with sparsity constraints; we also refer to \cite{sgghl09, hh09}. Here we derive slightly different estimates, which are better suited for our purposes, but result in the same rates of convergence.
Similar results for an a-posteriori parameter choice rule, namely Morozov's discrepancy principle, have been established in \cite{ar10, ar11, hm12}.

\begin{theorem}\label{rate0}
 Let $\yd \in Y$ with $\| \yd - y \| \leq \delta$ and $\xad$ be a minimizer of $\Phi_\alpha (x)$ given in \eqref{Phi} with $1<p \leq 2$. If Assumption \ref{assp_main} \ref{ass:nl}, \ref{ass:src} hold true for a $f_p$-minimizing solution $\xd \in X$, then we have
  \begin{equation}\label{L:EE}
  \begin{aligned}
    \|F(\xad)-\yd\| & \leq 2 \alpha \|\omega\|+\delta,\\
    D_{f_p}(\xad,\xd) & \leq \frac{(\delta+\alpha\|\omega\|)^{2}}{2(1-c\|\omega\|)\alpha}.
  \end{aligned}
  \end{equation}
\end{theorem}

\begin{proof}
 Since $\xad$ is a minimizer of $\Phi_{\alpha}(x)$, it follows that
  \begin{align*}
    \frac{1}{2}\|F(\xad)-\yd \|^{2} + \alpha f_p(\xad) - \alpha f_p(\xd) & \leq \frac{1}{2}\|F(\xd)-\yd\|^{2}.
  \end{align*}
  The source condition in Assumption \ref{assp_main} \ref{ass:src} thus yields
  \begin{multline*}
    \frac{1}{2} \|F(\xad)- \yd\|^{2} + \alpha D_{f_p}(\xad,\xd) \leq \frac{1}{2}\|F(\xd)-\yd\|^{2} - \alpha \langle J_p (\xd),\xad-\xd\rangle\\
      \begin{aligned}
      & \leq \frac{1}{2}\delta^{2}-\alpha\langle\omega,F'(\xd)(\xad-\xd)\rangle\\
      & = \frac{1}{2}\delta^{2} + \alpha \left\langle \omega,y-\yd+\yd-F(\xad)+F(\xad)-y-F'(\xd)(\xad-\xd)\right\rangle\\
      & \leq \frac{1}{2}\delta^{2}+\alpha\delta\|\omega\| +\alpha\|\omega\|\|F(\xad)-\yd\|+\alpha\|\omega\|c D_{f_p}(\xad,\xd).
      \end{aligned}
  \end{multline*}
  Therefore
  \begin{equation*}
    \|F(\xad)-\yd\|^{2}-2\alpha\|\omega\|\|F(\xad)-\yd\|+2(1-c\|\omega\|)\alpha D_{f_p}(\xad,\xd)\leq\delta^{2}+2\alpha\delta\|\omega\|\\
  \end{equation*}
  holds and finally
    $$\left(\|F(\xad)-\yd\|-\alpha\|\omega\|\right)^{2}+2(1-c\|\omega\|)\alpha D_{f_p}(\xad,\xd)\leq(\delta+\alpha\|\omega\|)^{2}$$
  implies both error estimates in \eqref{L:EE}.
\end{proof}

  One easily verifies that the right hand side in the estimate for $D_{f_p}(\xad,\xd)$ in \eqref{L:EE} is minimized if the regularization parameter is chosen as $\alpha = \delta / \| \omega \|$. This a-priori parameter choice rule would thus give optimal convergence rates estimates in terms of the Bregman distance. Since the source element $\omega$ itself will usually not be at hand, we might use the estimate $\| \omega \| \leq \varrho$ from Assumption~\ref{assp_main}~\ref{ass:src} and choose $\alpha = \delta / \varrho$ instead. This is indeed feasible, if the scaling parameter $\nrho = 3$ and the d-TIGRA algorithm will approximate $\alpha = \delta / \varrho$ from above (cf. Proposition~\ref{pro:outk}). In what follows, we will consider values $\alpha \geq \astar$, where
    \begin{equation} \label{eq:astar}
      \astar := \frac{\delta}{(\nrho - 2) \varrho}
    \end{equation}
  and for such $\alpha$ \eqref{L:EE} yields
    \begin{equation} \label{eq:dpalpha}
      \|F(\xad)-\yd\| \leq 2 \alpha \|\omega\| + \delta \leq \nrho \varrho \alpha 
    \end{equation}
  which will turn out to be of great use. Thus, the scaling parameter $\nrho > 2$ in Assumption~\ref{assp_main}~\ref{ass:src} can be regarded as a trade off between the smallness condition  $\varrho < \frac{1}{\nrho c}$ and the choice of the regularization parameter $\astar$.

  We conclude this section with two important consequences of Assumption \ref{assp_main}. It is worth noting, that the bounds obtained below are not necessarily optimal. Similar yet more elaborate estimates were already used in \cite{r03}. For our purposes, however, it is the existence of uniform bounds in $\alpha$ which is primarily important.

\begin{lemma} \label{lem:normbds}
  If Assumption \ref{assp_main} holds, then
    \begin{align}
      \| \xad \| & \leq A := \l( \frac{p}{2 \astar} \r)^{1/p} \| F(0) - \yd \|^{2/p} \label{eq:normxad}\\
      \|F'(\xad)\| & \leq K := L A + \|F'(0)\|, \label{eq:K}
    \end{align}
  for all $\alpha \geq \astar$.
\end{lemma}

\begin{proof}
  From the minimizing property of $\xad$ we obtain
    \begin{align*}
      \alpha f_p (\xad) & \leq \Phi_\alpha (\xad) \leq \Phi_\alpha (0) \leq \frac{1}{2} \| F(0) - \yd \|^2.
    \end{align*}
  Thus, for $\alpha \geq \astar$,
    \[ \| \xad \|^p \leq \frac{p}{2 \alpha} \| F(0) - \yd \|^2 \leq \frac{p}{2 \astar} \| F(0) - \yd \|^2. \]
  Now, \eqref{eq:normxad} readily follows by taking the $p$-th root.

  For the estimate on $\|F'(\xad)\|$ we observe that by Assumption \ref{assp_main} \ref{ass:lip}
    \begin{align*}
      \|F'(\xad)\| & \leq \|F'(\xad) - F'(0) \| + \|F'(0)\| \leq L \| \xad \| + \|F'(0)\|.
    \end{align*}
  Thus \eqref{eq:K} follows using \eqref{eq:normxad}.
\end{proof}

  Under our standing assumptions, the distance of minimizers $\xad$ and $\xbd$ corresponding to nearby values $\alpha$ and $\bar \alpha$ turns out to be of the order of $|\alpha - \bar \alpha|$. This result has essential implications regarding the outer update step where we decrease the regularization parameter $\alpha$ by a factor $\qa < 1$ (cf. Lemma~\ref{lem:qa} and Pro\-position~\ref{pro:ink}). The rather technical proof is postponed to Section \ref{ssec:upd}.

\begin{proposition} \label{pro:upd}
  Let Assumption \ref{assp_main} hold and let $\bar \alpha \geq \astar$. Then both $\xad$ and $F(\xad)$ are continuous from the right with respect to $\bar \alpha$, in the sense that
  \begin{equation*}
    \lim_{\alpha \to +\bar \alpha} \| \xad - \xbd \| = 0, \quad  \lim_{\alpha \to +\bar \alpha} \| F(\xad) - F(\xbd) \|= 0.
  \end{equation*}
  In particular, for $\bar \alpha \leq \alpha \leq \frac{\bar \alpha}{\qa_0}$ with
    \begin{equation} \label{eq:qa0}
      \qa_0 := \frac{2 \nrho c \varrho}{1 + \nrho c \varrho} < 1,
    \end{equation}
  we have
  \begin{equation} \label{eq:xadcont}
    \| \xad - \xbd \| \leq \sigma (\alpha - \bar \alpha),
  \end{equation}
  where $\sigma = \tfrac{2 \nrho \varrho K} {c_A (1- \nrho c \varrho) \astar}$ and $c_A : = \frac{p-1}{2} (3 A)^{p-2}$ with $A$ from Lemma~\ref{lem:normbds}.
\end{proposition}

\begin{remark}
  Note, that Proposition~\ref{pro:upd} implies uniqueness of the minimizer $\xad$ for $\alpha \geq \astar$. Indeed, if we let $\bar \alpha = \alpha$ with arbitrarily chosen minimizers $\xad$, $\xbd$, then \eqref{eq:xadcont} yields $\| \xbd - \xad\| = 0$.
\end{remark}

\smallskip

\section{The iterated dual gradient descent method}  \label{sec:duallwb}

  In this section we  analyze a dual gradient descent method for the computation of a minimizer of the Tikhonov functional with sparsity constraints using the discrepancy principle as a global stopping rule. Several technical proofs are collected in Section \ref{sec:proofs} for improved readability. The algorithm can be summarized as follows:

\begin{algorithm}
  \caption{The dual TIGRA algorithm}
  \SetAlgoNoLine
  \SetKwInOut{Input}{input}
  \SetKwInOut{Output}{output}
  \SetKwInOut{Init}{init}
  \SetKw{MyEnd}{end}

  \Input{$\qa < 1$, $\alpha_{0}$  and $\xs$.}

  \Init{$j=0$ and $\x{0}{0} = \xs$.}

  \BlankLine

  \While{$\| F(\xjks) - \yd\|\leq \tau \delta$}{
      \If{$j \geq 1$}{
	$\ajj = \qa \alpha_j$

	$\x{j+1}{0} = \xjk$
      }

      $j = j+1$, $k = 0$

      \While{$\|\nabla\Phi_{\alpha_j}(\xjk)\|<C_j$}{
	$J_{p}(\xj{k+1}) = J_{p}(\xjk) - \beta_{j,k} \nabla\Phi_{\alpha_j}(\xjk)$

	$\xj{k+1} =J_{q}(J_{p}(\xj{k+1}))$

	$k = k+1$
      }
  }
\end{algorithm}

  There are several requirements we need to verify in order to prove that the above algorithm is well-defined and globally convergent. First of all, for each $j$ the dual gradient descent algorithm in lines {\bf {\scriptsize 7-11}} ~ starting from $\xj{0}$ should converge to a minimizer $\xajd$ of $\Phi_{\alpha_j} (x)$ and, secondly, $\|\nabla\Phi_{\alpha_j}(\xjk)\|$ should tend to zero as $k \to \infty$, so that each inner iteration terminates after a finite number of steps. Additionally we need to ensure that the discrepancy principle can be used as a stopping rule for the outer iteration.

\subsection{A convexity property of the Tikhonov functional} \label{ssec:convexity}

  In the following, we would like to investigate the \emph{local directional convexity} of the Tikhonov functional $\Phi_\alpha (x)$ near a minimizer $\xad$, i.e.~the convexity of the functions
    \begin{equation} \label{eq:phi}
      \varphi_{\alpha,h}(t) : = \Phi_{\alpha}(\xad+th), \qquad t\in \IR
    \end{equation}
  for $h\in X$ with $\|h\|=1$. To this end, recall the following characterization of convexity on intervals.

\begin{proposition}\label{pro:convex}
  A continuously differentiable function $\varphi : \IR \to \IR$ is convex on an interval $I$ if and only if
  \begin{equation*}
    \varphi(t_2)-\varphi(t_1)-\varphi'(t_1)(t_2-t_1)\geq 0
  \end{equation*}
  for all $t_{1},t_{2} \in I$ and $\varphi$ is strictly convex if equality only holds for $t_1 = t_2$.
\end{proposition}

  Thus, $\varphi_{\alpha,h}(t)$ from \eqref{eq:phi} is convex on $[-r,r]$, if
  \begin{align*}
    0 & \leq \varphi_{\alpha,h}(t_{2}) -\varphi_{\alpha,h}(t_{1})-(t_{2}-t_{1})\varphi'_{\alpha,h}(t_{1})  = R_{\Phi_\alpha}(\xad + t_2 h, \xad + t_1 h) 
  \end{align*}
  holds for all $t_1,t_2 \in [-r,r]$, where
  \begin{equation} \label{eq:RPhi}
    R_{\Phi_\alpha} (z, x) : = \Phi_\alpha (z) - \Phi_\alpha (x) - \langle \nabla \Phi_\alpha (x), z-x \rangle
  \end{equation}
  for $x, z \in X$.  The following Theorem asserts the existence of such $r$ as a function of the regularization parameter $\alpha$. The proof can be found in Section \ref{ssec:convex}.

\begin{theorem}\label{thm:convex}
  Let Assumption \ref{assp_main} hold and let $\gamma = \frac{1 - \nrho c \varrho}{2}$. Then, for all $\alpha \geq \astar$ and $h \in X$ with $\|h\|=1$,
    \[ R_{\Phi_\alpha} (\xad + t_2 h, \xad + t_1 h) \geq \gamma \alpha D_{f_p}(\xad + t_2 h,\xad + t_1 h) \]
  holds for $|t_1|,|t_2| \leq \cra$, where $R_{\Phi_\alpha}$ is as defined in \eqref{eq:RPhi} and
    \begin{equation} \label{eq:ralpha}
      \cra := \frac{2}{1+\sqrt{2}} \min\left\{ \frac{\gamma \alpha}{c \nrho K}, \sqrt{\frac{\gamma  \alpha}{10 c L}} \right\}.
    \end{equation}
  In particual, $\varphi_{\alpha,h}(t)$ in \eqref{eq:phi} is thus convex on $[-\cra, \cra]$.
\end{theorem}

Of course, Theorem~\ref{thm:convex} does not ensure that $\Phi_\alpha(x)$ is itself convex in any neighbourhood of $\xad$. We will say that $\Phi_\alpha (x)$ is \emph{directionally convex} on $B_{\cra} (\xad)$, tacitly assuming that only lines through $\xad$ are considered. Here $B_r(x)$ denotes the ball of radius $r$ centered at $x\in X$, i.e.
  \[ B_{\cra} (\xad) = \{ x \in X \st \| x- \xad \| \leq \cra \}. \]
Note that $\cra$ does not depend on the direction $h$ and that $\cra \to \infty$ as $\alpha \to \infty$.

As a first consequence of Theorem \ref{thm:convex}, we obtain an estimate on the directional derivative of $\Phi_\alpha(x)$ in direction $\xad$.

\begin{proposition} \label{pro:vphi}
  Let Assumption \ref{assp_main} hold and let $x \in B_{\crr{\alpha}}(\xad)$ with $\cra$ from Theorem~\ref{thm:convex} for some $\alpha \geq \astar$, then
    \begin{equation}\label{Phi:Df}
      \langle \nabla \Phi_{\alpha}(x), x - \xad \rangle \geq \gamma \alpha D_{f_p}(\xad, x)
    \end{equation}
  holds with $\gamma = \frac{1 - \nrho c \varrho}{2}$.
\end{proposition}

\begin{proof}
    Using the minimizing property of $\xad$ and Theorem~\ref{thm:convex}, we obtain
    \begin{align*}
      -\langle\nabla\Phi_{\alpha}(x), x - \xad \rangle & = R_{\Phi_{\alpha}} (\xad, x) -\Phi_{\alpha}(\xad)+\Phi_{\alpha}(x)\\
	& \geq R_{\Phi_{\alpha}}(\xad,x) \geq \gamma\alpha  D_{f_p}(\xad,x),
    \end{align*}
  and the proof is complete.
\end{proof}
 
 \begin{remark} \label{rem:critpts}
    Note that, due to the strict convexity of $f_p(x)$, Proposition \ref{pro:vphi} in particular implies that there are no critical points other than $\xad$ inside $B_{\cra} (\xad)$. The size $\cra$ of this region increases as $\alpha \to \infty$ and due to Proposition \ref{pro:upd} there exists $\alpha$ such that $\xad \neq 0$, but $0 \in B_{\cra} (\xad)$. Consequently,  $F'(0)^*$ cannot be the zero mapping as otherwise $\nabla \Phi_\alpha (0) = 0$ would hold and thus $x=0$ would be a critical point of $\Phi_\alpha (x)$.
 \end{remark}

\subsection{Initial values} \label{ssec:updstop}

  Our convergence analysis below relies primarily on the Bregman distance and for the strictly convex functionals $f_p (x)$ it is well known that
	\[ D_{f_p} (\xajd, \xjk) \to 0 \]
  also yields $\| \xajd - \xjk \| \to 0$. But at the same time it will be important to ensure that the iterates $\xjk$ remain inside the region of directional convexity $B_{\crr {\alpha_j} } (\xajd)$ with radius $\crr {\alpha_j} $ from Theorem~\ref{thm:convex} throughout the iteration (compare the proof of Theorem~\ref{thm:convg}, for example). The following general result asserts that indeed both can be achieved simultaneously. The proof is given in Section \ref{ssec:Dfnorm}.

  \begin{lemma} \label{lem:Dfnorm}
    To every $\alpha \geq \astar$ there exists $\bra > 0$ such that
      \[ \Bbra \subset B_{\cra} (\xad), \]
    where $\Bb{d}{z}$ denotes the ball of radius $d$ around $z \in X$ with respect to the Bregman distance, i.e.
      \begin{equation} \label{eq:Bbra}
        \Bbra : = \{ x \in X \st D_{f_p} (\xad, x) \leq \bra \}.
      \end{equation}
    Moreover, the numbers $\bra$ may be chosen such that they depend on $\alpha$ continuously, they are strictly monotonically increasing and that $\bra \to \infty$ as $\alpha \to \infty$.
  \end{lemma}

  In view of Theorem~\ref{thm:convg} below, we refer to $\Bbra$ as the \emph{region of convergence} of the dual gradient descent method.
  For the initiation of the algorithm it will be important that the starting point $\xs$ belongs to $\Bb {d_{\alpha_0}} {x_{\alpha_0}^\delta}$
  for the regularization parameter $\alpha_0$. It is our next objective to show that the latter is indeed the case for sufficiently large $\alpha_0$.

\begin{proposition} \label{pro:start}
  Let Assumption \ref{assp_main} hold, then to each starting value $\xs \in X$ there exists 
  $\alpha_0 > \astar$ large enough such that $\xs \in \Bb {d_{\alpha_0}} {x_{\alpha_0}^\delta}$ as defined in \eqref{eq:Bbra}.
\end{proposition}

\begin{proof}
  Observe that by Lemma~\ref{lemma1} and Lemma~\ref{lem:normbds},
    \[ D_{f_p} (\xad, \xs) \leq \tilde c_p \| \xad - \xs \| \leq \tilde c_p (A + \| \xs \|) \]
  remains bounded, whereas $\bra \to \infty$ as $\alpha \to \infty$ according to Lemma~\ref{lem:Dfnorm}. Thus, there exists $\alpha_0$ such that $D_{f_p} (x_{\alpha_0}^{\delta}, \xs) \leq \br {\alpha_0}$ and the proof is complete.
\end{proof}

\subsection{Convergence of the dual gradient descent method} \label{ssec:convg}

  Let us now analyze the dual gradient descent iteration for fixed regularization parameter $\alpha_j \geq \astar$, which reads as follows:

  \begin{equation}\label{eq:lwb}
  	\begin{aligned}
    J_p ( \xj{k+1} ) & = J_p (\xjk) - \beta_{j,k} \nabla \Phi_{\alpha_j} (\xjk)\\
      \xj{k+1} & = J_q (J_p ( \xj{k+1} )).
      \end{aligned}
  \end{equation}

  In particular, we will show that the sequence of iterates $\{ \xjk \}_{k \in \IN}$ approaches the minimizer $\xajd$ in the Bregman distance. Then, for the strictly convex functionals $f_p (x)$ with $1<p\leq 2$, this even implies convergence in norm. 

  For an appropriate starting value $\xj 0$ and suitably chosen step-sizes $\beta_{j,k} $ the Bregman distance between the iterates $\xjk$ and the exact minimizer $\xajd$ does not increase throughout the algorithm. We prove this result in Section \ref{ssec:convg_breg}.

\begin{theorem} \label{thm:convg_breg}
  Let Assumption \ref{assp_main} hold and suppose that $D_{f_p} (\xajd,\xjk) \leq \bra$ with $\bra$ as in Lemma~\ref{lem:Dfnorm} and that $\nabla \Phi_{\alpha_j} (\xjk) \neq 0$.
  Then $ \xj{k+1} $ from \eqref{eq:lwb} is at least as close to $\xajd$ as $\xjk$ with respect to the Bregman distance, i.e.
    $$D_{f_p}(\xajd, \xj{k+1} )\leq D_{f_p}(\xajd,\xjk),$$
  if the step-size $\beta_{j,k} $ is chosen such that
    \begin{equation} \label{eq:beta1}
	\beta_{j,k}  \leq \frac{\gamma \; \bar c_{j,k} \; \alpha_j}{\|\nabla\Phi_{\alpha_j}(\xjk)\|^2},
    \end{equation}
  holds, where 
  \begin{equation} \label{eq:cbar}
    \bar c_{j,k}  :=  \min \l\{ 1, \frac{8 \bar c_A (\Phi_{\alpha_j} (\xjk) - \phi_{j,k})}{4 K^2 + 4 L \nrho \varrho \alpha_j + 4 L K \crr{\alpha_j} + L^2 \crr{\alpha_j}^2 + 8 \alpha_j \tilde{c}_p \bar c_A^{\frac{2-p}{2}}}  \r\}
  \end{equation}
  with $\bar c_A:= \frac{p-1}{2} (2 \cra + A)^{p-2}$ and
    \[ \phi_{j,k} : =\min\{\Phi_{\alpha_j}(J_q ( J_p (\xjk) + t \nabla\Phi_{\alpha_j} (\xjk)) ):t \in \IR^{+}\}. \]
\end{theorem}

  Using the decay with respect to the Bregman distance, we also obtain monotonicity in the Tikhonov functional values $\Phi_{\alpha_j} (\xjk)$ and that $\| \nabla \Phi_{\alpha_j} (\xjk) \| \to 0$ as $k \to \infty$. It is worth noting that the latter also ensures that the inner stopping rule $\| \nabla \Phi_{\alpha_j} (\xjk) \| \leq C_{\alpha_j}$ is well-defined for any value $C_{\alpha_j} > 0$. For the proof, which goes along the lines of Theorem 2.9 in \cite{r02b}, we refer to Section \ref{ssec:stop}, where the constants $M_{\alpha_j}$ are explicitely given as well.

\begin{proposition} \label{pro:stop}
  Let Assumption \ref{assp_main} hold and suppose that $D_{f_p} (\xajd,x_0) \leq \bra$ with $\bra$ as in Lemma~\ref{lem:Dfnorm}.
  Then there exists a constant $M_{\alpha_j} > 0$ such that the sequence $\{ \xjk \}_{k \in \IN_0}$ generated by \eqref{eq:lwb} with step-sizes
  \begin{equation} \label{eq:betak}
    \beta_{j,k}  := \min \left\{\frac{\gamma \; \bar c_{j,k} \; \alpha_j}{\|\nabla\Phi_{\alpha_j}(\xjk)\|^2}, \frac{1}{2 M_{\alpha_j}} \right\}
  \end{equation}
  satisfies
  \begin{align*}
    \Phi_{\alpha_j} ( \xj{k+1} ) \leq \Phi_{\alpha_j} (\xjk) \qquad \mbox{and} \qquad  R_{\Phi_{\alpha_j}} ( \xj{k+1} , \xjk) \leq M_{\alpha_j} D_{f_p} ( \xj{k+1} , \xjk)
  \end{align*}
  for all $k \geq 0$, as well as
    \[ \| \nabla \Phi_{\alpha_j} (\xjk) \| \to 0 \qquad \mbox{as } k \to \infty. \]
\end{proposition}

  Now we can give the convergence result for the dual gradient descent method with fixed $\alpha_j \geq \astar$.

\begin{theorem}\label{thm:convg}
  Suppose that Assumption \ref{assp_main} holds and that $\xj{0} \in \Bb{\br{\alpha_j}}{\xajd}$ for $\alpha_j \geq \astar$. Let $\{ \xjk \}_{k \in \IN_0}$ be the sequence generated by the dual iteration \eqref{eq:lwb} with step-sizes $\beta_{j,k}$ from \eqref{eq:betak}, then $\xjk$ converges to the global minimizer $\xajd$ of $\Phi_{\alpha_j}$, i.e.
    \[ \xjk \to \xajd \qquad \mbox{as } k \to \infty. \]
\end{theorem}

\begin{proof}
  We show $D_{f_p} (\xajd, \xjk) \to 0$, then $\xjk \to \xajd$ follows due to the strict convexity of the functionals $f_p (x)$. From Proposition~\ref{pro:vphi}, we obtain
	\begin{align*}
	   D_{f_p} (\xajd, \xjk) & \leq \frac{1}{\gamma \alpha_j} \langle \nabla\Phi_{\alpha_j}(\xjk), \xjk-\xajd \rangle \leq \frac{1}{\gamma \alpha_j} \| \nabla\Phi_{\alpha_j}(\xjk) \| \cdot \| \xjk-\xajd \|.
	\end{align*}
  Thus, the result follows as, on the one hand, $\| \xjk-\xajd \| \leq \crr{\alpha_j}$ remains bounded for all $k$ throughout the iteration due to $\xj{0} \in \Bb{\br{\alpha_j}}{\xajd}$ and Theorem~\ref{thm:convg_breg}, and, on the other hand, $\| \nabla\Phi_{\alpha_j}(\xjk) \| \to 0$ as $k \to \infty$ according to Proposition \ref{pro:stop}.
\end{proof}

\subsection{Updating the regularization parameter}

  According to Theorem~\ref{thm:convg}, the dual gradient descent method with regularization parameter $\alpha_j$ converges to the minimizer $\xajd$ of the Tikhonov functional $\Phi_{\alpha_j} (x)$, if the starting value $\x j 0$ belongs to the region of convergence, $\x j 0 \in \Bb {\br{\alpha_j}} {\xajd}$. Based on Proposition~\ref{pro:upd}, we will now show that is the case for the choice $\x {j+1} 0 = \xj {k^*(j)}$ if the update $\ajj = \qa \alpha_j$ is not too large and if $C_j$ in the stopping rule $\| \nabla \Phi_{\alpha_j} (\xjk) \| \leq C_j$ is chosen suitably. To this end, we first specify the requirements for the update factor $\qa$.

\begin{lemma} \label{lem:qa}
  Let Assumption \ref{assp_main} hold and let $\alpha_0$ be chosen as in Proposition~\ref{pro:start}. If $\qa \in (\qa_0, 1)$ with $\qa_0$ from \eqref{eq:qa0} is chosen such that
    \begin{equation} \label{eq:qa}
      \rho \, \sigma \, \alpha_0 \, (1-\qa) \leq \min \big\{ \br{\astar} / 3, 1 \big\}
    \end{equation}
  holds, where $\sigma$ is as in Proposition \ref{pro:upd} and $\rho : = \max \big\{ \tilde c_p^{1/p}, ~ 2 A^{p-1} + \crr{\alpha_0}^{p-1} \big\}$
  with $\tilde c_p$ and $\bra$ as in Lemma~\ref{lemma1} and Lemma~\ref{lem:Dfnorm}, respectively, then for all $\alpha \in (\astar, \alpha_0]$
    \begin{align}
       \label{eq:upd} D_{f_p} (\xbd, \xad) & \leq \frac{\br{\bar \alpha}}{3} \qquad \mbox{and} \qquad   \| \xbd - \xad \| \leq \frac{\br{\bar \alpha}}{3 \rho}
    \end{align}
  holds, where $\bar \alpha := \max \{ \qa \alpha, \astar \}$.
\end{lemma}

\begin{proof}
  Note that the existence of $\qa$ as in \eqref{eq:qa} is guaranteed as only the left hand side in \eqref{eq:qa} tends to zero as $\qa \to 1$. Thus, by virtue of Proposition~\ref{pro:upd}, we obtain
    \begin{align*}
      \| \xad - \xbd \| & \leq \sigma \alpha_0 (1-\qa) \leq \frac{\br{\astar}}{3\rho} \leq \frac{\bra}{3\rho}, 
    \end{align*}
  and, due to Lemma~\ref{lemma1}, also
    \begin{align*}
      D_{f_p} (\xad, \xbd) & \leq \tilde c_p \|\xbd-\xad\|^p \leq (\rho \sigma \alpha_0 (1-\qa) )^p \leq \min \big\{ \br{\astar} /3, 1 \big\}^p \leq \frac{\bra}{3},
    \end{align*}
  where we have used $\qa \alpha \leq \bar \alpha \leq \alpha \leq \min \{ \frac{\bar \alpha}{\qa_0}, \alpha_0 \}$ as well as the monotonicity of $\bra$.
\end{proof}

\begin{proposition} \label{pro:ink}
  Let $\qa$ be chosen as in Lemma \ref{lem:qa}. Suppose that $\xj{0} \in \Bb {\br {\alpha_j}}{\xajd}$ and let $\{ \xjk \}_{k \geq 0}$ be the sequence generated by \eqref{eq:lwb} with step-sizes $\beta_{j, k}$ from \eqref{eq:betak}. If $\xjks$ is defined to be the first iterate which satisfies
 \begin{equation}\label{TOL:0}
 \begin{aligned}
    &\|\nabla \Phi_{\alpha_j} (\xjks)\| \leq C_{\alpha_j},
 \end{aligned}
 \end{equation}
  with
      \begin{equation} \label{eq:Cabd}
        C_{\alpha_j} \leq \frac{\gamma \alpha_j \br{\ajj}}{3 \crr{\alpha_j}} ,
      \end{equation}
 then $\xjks \in \Bb {\br {\ajj}}{\xajjd}$, where $\ajj := \max \{ \qa \alpha_j, \astar \}$ and $\qa$ is as in Lemma \ref{lem:qa}.
\end{proposition}

{\it Proof.}
  Note that
  \begin{multline*}
    D_{f_p} (\xajjd, \xjks) = D_{f_p} (\xajjd, \xajd) + D_{f_p} (\xajd, \xjks)
	    + \langle J_p (\xajd) - J_p (\xjks),\xajjd - \xajd \rangle.
  \end{multline*}
  Due to \eqref{eq:duality} and Lemma \ref{lem:normbds}, we obtain
    \[ \| J_p (\xajd) - J_p (\xjks) \| \leq \| \xajd \|^{p-1} + \| \xjks \|^{p-1} \leq \rho \]
  with $\rho$ from Lemma \ref{lem:qa}. In addition, using Proposition \ref{pro:vphi}
  \begin{align*}
     D_{f_p}(\xajd,  \xjks) & \leq \frac{1}{\gamma \alpha_j} \langle \nabla \Phi_{\alpha_j}(\xjks), \xjks-\xajd\rangle\\
	& \leq \frac{1}{\gamma \alpha_j} \|\nabla \Phi_{\alpha_j}(\xjks)\|  \| \xjks - \xajd\| \leq \frac{C_j \crr{\alpha_j}}{\gamma \alpha_j} \leq \frac{\br{\ajj}}{3}
  \end{align*}
  holds as $\xjks \in \Bb {\br {\alpha_j}}{\xajd}$ due to Theorem \ref{thm:convg_breg}. Combining these estimates with \eqref{eq:upd} in Lemma \ref{lem:qa}, we obtain
    \begin{align*}
      D_{f_p} (\xajjd, \xjks) & \leq \frac{\br{\ajj}}{3} +  \frac{\br{\ajj}}{3} + \rho \frac{\br{\ajj}}{3 \rho} = \br{\ajj}. \qquad \endproof
    \end{align*}

\medskip

\section{The global minimization strategy}  \label{sec:globalmin}

  As we have seen, for fixed $\alpha_j \geq \astar$ the dual gradient descent iteration \eqref{eq:lwb} yields a sequence of iterates $\{ \xjk \}_{k \in \IN_0}$ which converges to $\xajd$ and will terminate with an approximation $\xjks$. But we have yet to ensure that the proposed algorithm terminates after a finite number $\js$ of outer iteration steps with a regularization parameter $\alpha_\js \geq \astar$.
 Before we give this main convergence result for the d-TIGRA method, we collect here all the assumptions on the parameters, which determine the behavior of the algorithm. For the values of the various constants appearing in the expressions below, we refer the reader to the list of constants which we have included at the end of this paper.

\begin{assumption} \label{ass:tigra} $~$\\
  \indent Suppose that
  \begin{enumerate}
    \item \label{ass:init}
	the initial parameter $\alpha_0$ is sufficiently large such that $\xs \in \Bb {d_{\alpha_0}} {x_{\alpha_0}^\delta}$, i.e.
	  \[ D_{f_p} (x_{\alpha_0}^{\delta}, \xs) \leq \br {\alpha_0}, \]
    \item \label{ass:qa}
	the factor $\qa \in (\qa_0, 1)$ with $\qa_0$ from \eqref{eq:qa0} satisfies
	  \[ \rho \, \sigma \, \alpha_0 \, (1-\qa) \leq \min \big\{ \br{\astar} / 3, 1 \big\}, \]
    \item \label{ass:stepsize}
	the step-sizes $\beta_{j,k}$ are chosen according to
	  \[ \beta_{j,k}  = \min \left\{\frac{\gamma \; \bar{c}_{j, k} \; \alpha_j}{\|\nabla\Phi_{\alpha_j}(\xjk)\|^2}, \frac{1}{2 M_{\alpha_j}} \right\}, \]
    \item \label{ass:control}
	the constants $C_{\alpha_j}$ such that
	  \[ C_{\alpha_j} \leq \gamma \alpha_j \min \bigg\{ \frac{\br{\ajj}}{3 \crr{\alpha_j}}, \frac{\delta \, \bar c_A}{K + \crr{\alpha_j} (c \, \bar c_A + L)} \bigg\},  \]
    \item \label{ass:tau}
	  and $\tau$ in the discrepancy principle as
	  \[ \tau = 2 + \frac{2}{\qa (s-2)}. \]
  \end{enumerate}
\end{assumption}

 For the following Proposition, compare also Proposition 6.3 and 6.4 in \cite{r03}.

\begin{proposition} \label{pro:outk}
  Let Assumptions \ref{assp_main} and \ref{ass:tigra} hold, 
  then the iterates $\xjk$ generated by the dual TIGRA algorithm remain inside the region of convergence, i.e. $\xjk \in \Bb {\br{\alpha_j}} {\xajd}$ defined in \eqref{eq:Bbra}, and the algorithm terminates after a finite number $\js$ of outer iteration steps with regularization parameter $\alpha_\js \geq \astar$.
\end{proposition}

\begin{proof}
  We begin by showing that for $\alpha_0$, $\qa$, $\beta_{j,k}$, and $C_{\alpha_j}$ as in Assumption~\ref{ass:tigra}, we have $\xjk \in \Bb {\br{\alpha_j}} {\xajd}$: This holds true for the starting value $\x 0 0 = \xs$ by definition of $\alpha_0$, and as then the Bregman distance decreases with step-sizes $\beta_{j,k}$ (cf. Theorem~\ref{thm:convg_breg}) also for $\x 0 k$, $k \geq 0$. Finally, for the outer iteration we have $\x {j+1} 0 = \xjks \in \Bb {\br {\ajj}}{\xajjd}$ according to Proposition \ref{pro:ink}, where $\xjks$ denotes the final iterate with index $j$, and using an induction argument the assertion follows for all $j \in \IN_0$ and $0 \leq k \leq \ksj$.  According to Lem\-ma~\ref{lem:Dfnorm} it thus holds that
    \[ D_{f_p} (\xajd, \xjk) \leq \br {\alpha_j} \qquad \mbox{and} \qquad \| \xajd - \xjk \| \leq \crr{\alpha_j}. \]

  To prove that the dual TIGRA algorithm stops with $\alpha_\js \geq \astar$, we proceed by contradiction. Note that the sequence $\alpha_j = \qa^j \alpha_0$ is monotonically decreasing and converges to zero since $\qa < 1$. Let now $j$ denote the unique index such that $\alpha_j \geq \astar$ and $\ajj = \qa \alpha_j < \astar$ and suppose that $\js \geq j+1$.
  Due to Lemma \ref{lemma1} with $c_p = \bar c_A = \frac{p-1}{2} (A + 2 \crr{\alpha_j})^{p-2}$, the inner stopping rule $\| \nabla \Phi_{\alpha_j} (\xjks) \| \leq C_{\alpha_j}$ and Proposition~\ref{pro:vphi} we obtain
  \begin{align*}
    \| \xajd - \xjks \|^2 & \leq \frac{1}{\bar c_A} D_{f_p} (\xajd, \xjks) \leq \frac{1}{\bar c_A \gamma \alpha_j} \langle \nabla \Phi_{\alpha_j} (\xjks), \xjks - \xajd \rangle\\
	& \leq  \frac{C_{\alpha_j}}{\bar c_A \gamma \alpha_j} \| \xajd - \xjks \|.
  \end{align*}
  Using a boot-strap argument on the one hand, and that $\| \xajd - \xjks \| \leq \crr{\alpha_j}$ on the other hand, it follows that
    \[ \| \xajd - \xjks \| \leq \frac{C_{\alpha_j}}{\bar c_A \gamma \alpha_j} \qquad \mbox{and} \qquad D_{f_p} (\xajd, \xjks) \leq \frac{C_{\alpha_j} \crr{\alpha_j}}{\gamma \alpha_j}. \]
  Thus, Assumption~\ref{assp_main}, Lemma \ref{lem:normbds} and Assumption~\ref{ass:tigra}~\ref{ass:control} yield
  \begin{align*}
      \| F( & \xjks) - F(\xajd) \|\\
	  & \leq \| F(\xjks) - F(\xajd) - F'(\xjks) (\xajd - \xjks) \| + \| F'(\xjks) \| \| \xajd - \xjks \|\\
	  & \leq c D_{f_p} (\xajd, \xjks) + ( L \crr{\alpha_j} + K) \| \xajd - \xjks \|\\
	  & \leq \frac{C_{\alpha_j}}{\bar c_A \gamma \alpha_j} \big( \crr{\alpha_j} (c \, \bar c_A + L) + K \big) \leq \delta.
  \end{align*}
  In combination with $\|F(\xajd)-\yd\| \leq 2 \varrho \alpha_j + \delta$ from \eqref{eq:dpalpha} and $\ajj < \astar = \frac{\delta}{(s-2) \varrho}$ we obtain
  \begin{equation*}
  \begin{aligned}
	\|F(\xjks)-\yd\| & \leq \|F(\xjks) - F(\xajd)\| + \|F(\xajd)-\yd\| \leq 2 \varrho \frac{\ajj}{\qa} + 2 \delta\\
		& \leq \frac{2 \delta}{\qa (s-2)} + 2  \delta \leq \tau \delta.
  \end{aligned}
  \end{equation*}
  Thus the discrepancy principle with $\tau$ from Assumption~\ref{ass:tigra}~\ref{ass:tau} is satisfied for $\xjks$, which contradicts our assumption $\js \geq j+1$. Hence, the iteration terminates with a regularization parameter $\alpha_\js \geq \astar$.
\end{proof}

\section{Numerical examples}  \label{sec:numerics}

  In this section we compare numerical results for the d-TIGRA method to a dual version of the modified Landweber method \cite{s98}. Both methods are applied to the nonlinear ill-posed auto-convolution problem (compare also \cite{ar10}), where the forward operator is given by
  \begin{equation}\label{E:auto}
      G(f)(s) = \int_0^{s} f(s-t) f(t) dt, \qquad s \in [0,1], f \in L^2[0,1].
  \end{equation}
  It has been shown in \cite{gh94} that the auto-convolution operator is continuous as
  			\[ \| G(f_1) - G(f_2) \|_{C[0,1]} \leq (\| f_1 \|_{L^2[0,1]} + \| f_2 \|_{L^2[0,1]} ) \| f_1 - f_2 \|_{L^2[0,1]}, \]
  with Fr\'echet derivative $G'(f)$ given by
  			\[ \big[ G'(f) h \big] (s) = 2 \, [f * h] (s) = 2 \int_0^{s} f(s-t) h(t) dt, \qquad s \in [0,1], h \in L^2[0,1] \]
  and weakly sequentially closed on the domain
  	 		\[ D(G) = D^+ := \{ f \in L^2[0,1] \st f(t) \geq 0 \mbox{ a.e. } t \in [0,1] \}. \]
  As mentioned in \cite{ar10}, $G'(.)$ is Lipschitz continuous with constant $L=2$ and adjoint
  			\[  G'(f)^* h = 2 \, \big( f * \tilde h \big)^{\widetilde{}}, \]
  where $\tilde h (t) = (h)^{\widetilde{}} (t) = h(1-t)$.

  For the reconstruction of solutions which are sparse with respect to a certain Riesz basis or frame $\{ u_i \}_{i \in \IN}$ of $L^2[0,1]$, we represent functions $f \in L^2[0,1]$ in terms of their coefficients $x$ such that $f = Tx$, where the bounded, linear operator $T : \ell^2 \to L^2 [0,1]$ is given by
      \[ Tx := \sum_{i \in \IN} x_i u_i \]
  with adjoint
        \[ T^*f = \{\langle f, u_i \rangle\}_{i \in \IN} . \]
  In accordance with our theoretic results, we restrict the domain of $T$ to sequences in $\ell^p (\IN)$ with $1 < p \leq 2$ and consider the problem of reconstructing a solution $\xd$ of $F(x) = y$, where
  		\[ F := G \circ T : \quad \ell^p \to Y = L^2[0,1] \]
  from noisy data $\yd$ with $\| y - \yd \|_Y \leq \delta$. Clearly, $T: \ell^p \to L^2[0,1]$ as well as $T^*: L^2[0,1] \to \ell^q$ remain bounded as
  		\[ \| T x \|_{L^2[0,1]} \leq \| T \|_{\ell^2 \to L^2[0,1]} \| x \|_{\ell^2} \leq \| T \|_{\ell^2 \to L^2[0,1]} \| x \|_{\ell^p} \]
   holds for all $x \in \ell^p$ and an analog expression is obtained for $T^*$.
   As above, the Tikhonov-functional with sparsity constraints reads
      $$\Phi_\alpha (x):=\frac{1}{2}\|F(x)-\yd\|^{2}+ \frac{\alpha}{p}\sum |x_i|^p$$
   and we have
      $$\nabla \Phi_\alpha (x) = T^* G'(T x)^* (G(T x) - \yd) + \alpha J_p (x),$$
  where $J_p (x)_i = {\rm sign}(x_i)|x_i|^{p-1}$. 

  In our numerical experiments, we compute approximations of an exact solution $f^\dag$ which is sparse with respect to the Haar wavelet basis and is described by three nonzero coefficients, namely $x^\dag(\{2, 4, 7\}) = \{3, -1, 0.5\}$. Due to the quadratic nature of the autoconvolution problem, whenever $f^\dag$ is a solution of $G(f) = y$, then so is $-f^\dag$.
  The functions are sampled at 512 points equispaced in $[0,1]$ and we use wavelet coefficients up to index level $J=9$.

\subsection{Results for the dual TIGRA method} \label{ssec:numdtigra}

  The Fr\'echet derivative $F'(x)$ of the autoconvolution problem is Lipschitz continuous and thus Assumption~\ref{assp_main}~\ref{ass:init} is satisfied. However, nonlinearity conditions as required in Assumption~\ref{assp_main}~\ref{ass:nl} are not known, and recent results from \cite{bh14} indicate that source conditions as in Assumption~\ref{assp_main}~\ref{ass:src} certainly cannot hold for the autoconvolution operator. This is supported by the fact that $F'(0)^* w = 0$ for all $w \in Y$, which contradicts our assumptions according to Remark \ref{rem:critpts}.

\def\sol{(1, 3.00)(2, 3.00)(3, 3.00)(4, 3.00)(5, 3.00)(6, 3.00)(7, 3.00)(8, 3.00)(9, 3.00)(10, 3.00)(11, 3.00)(12, 3.00)(13, 3.00)(14, 3.00)(15, 3.00)(16, 3.00)(17, 3.00)(18, 3.00)(19, 3.00)(20, 3.00)(21, 3.00)(22, 3.00)(23, 3.00)(24, 3.00)(25, 3.00)(26, 3.00)(27, 3.00)(28, 3.00)(29, 3.00)(30, 3.00)(31, 3.00)(32, 3.00)(33, 3.00)(34, 3.00)(35, 3.00)(36, 3.00)(37, 3.00)(38, 3.00)(39, 3.00)(40, 3.00)(41, 3.00)(42, 3.00)(43, 3.00)(44, 3.00)(45, 3.00)(46, 3.00)(47, 3.00)(48, 3.00)(49, 3.00)(50, 3.00)(51, 3.00)(52, 3.00)(53, 3.00)(54, 3.00)(55, 3.00)(56, 3.00)(57, 3.00)(58, 3.00)(59, 3.00)(60, 3.00)(61, 3.00)(62, 3.00)(63, 3.00)(64, 3.00)(65, 3.00)(66, 3.00)(67, 3.00)(68, 3.00)(69, 3.00)(70, 3.00)(71, 3.00)(72, 3.00)(73, 3.00)(74, 3.00)(75, 3.00)(76, 3.00)(77, 3.00)(78, 3.00)(79, 3.00)(80, 3.00)(81, 3.00)(82, 3.00)(83, 3.00)(84, 3.00)(85, 3.00)(86, 3.00)(87, 3.00)(88, 3.00)(89, 3.00)(90, 3.00)(91, 3.00)(92, 3.00)(93, 3.00)(94, 3.00)(95, 3.00)(96, 3.00)(97, 3.00)(98, 3.00)(99, 3.00)(100, 3.00)(101, 3.00)(102, 3.00)(103, 3.00)(104, 3.00)(105, 3.00)(106, 3.00)(107, 3.00)(108, 3.00)(109, 3.00)(110, 3.00)(111, 3.00)(112, 3.00)(113, 3.00)(114, 3.00)(115, 3.00)(116, 3.00)(117, 3.00)(118, 3.00)(119, 3.00)(120, 3.00)(121, 3.00)(122, 3.00)(123, 3.00)(124, 3.00)(125, 3.00)(126, 3.00)(127, 3.00)(128, 3.00)(129, 3.00)(130, 3.00)(131, 3.00)(132, 3.00)(133, 3.00)(134, 3.00)(135, 3.00)(136, 3.00)(137, 3.00)(138, 3.00)(139, 3.00)(140, 3.00)(141, 3.00)(142, 3.00)(143, 3.00)(144, 3.00)(145, 3.00)(146, 3.00)(147, 3.00)(148, 3.00)(149, 3.00)(150, 3.00)(151, 3.00)(152, 3.00)(153, 3.00)(154, 3.00)(155, 3.00)(156, 3.00)(157, 3.00)(158, 3.00)(159, 3.00)(160, 3.00)(161, 3.00)(162, 3.00)(163, 3.00)(164, 3.00)(165, 3.00)(166, 3.00)(167, 3.00)(168, 3.00)(169, 3.00)(170, 3.00)(171, 3.00)(172, 3.00)(173, 3.00)(174, 3.00)(175, 3.00)(176, 3.00)(177, 3.00)(178, 3.00)(179, 3.00)(180, 3.00)(181, 3.00)(182, 3.00)(183, 3.00)(184, 3.00)(185, 3.00)(186, 3.00)(187, 3.00)(188, 3.00)(189, 3.00)(190, 3.00)(191, 3.00)(192, 3.00)(193, 3.00)(194, 3.00)(195, 3.00)(196, 3.00)(197, 3.00)(198, 3.00)(199, 3.00)(200, 3.00)(201, 3.00)(202, 3.00)(203, 3.00)(204, 3.00)(205, 3.00)(206, 3.00)(207, 3.00)(208, 3.00)(209, 3.00)(210, 3.00)(211, 3.00)(212, 3.00)(213, 3.00)(214, 3.00)(215, 3.00)(216, 3.00)(217, 3.00)(218, 3.00)(219, 3.00)(220, 3.00)(221, 3.00)(222, 3.00)(223, 3.00)(224, 3.00)(225, 3.00)(226, 3.00)(227, 3.00)(228, 3.00)(229, 3.00)(230, 3.00)(231, 3.00)(232, 3.00)(233, 3.00)(234, 3.00)(235, 3.00)(236, 3.00)(237, 3.00)(238, 3.00)(239, 3.00)(240, 3.00)(241, 3.00)(242, 3.00)(243, 3.00)(244, 3.00)(245, 3.00)(246, 3.00)(247, 3.00)(248, 3.00)(249, 3.00)(250, 3.00)(251, 3.00)(252, 3.00)(253, 3.00)(254, 3.00)(255, 3.00)(256, 3.00)(257, -3.41)(258, -3.41)(259, -3.41)(260, -3.41)(261, -3.41)(262, -3.41)(263, -3.41)(264, -3.41)(265, -3.41)(266, -3.41)(267, -3.41)(268, -3.41)(269, -3.41)(270, -3.41)(271, -3.41)(272, -3.41)(273, -3.41)(274, -3.41)(275, -3.41)(276, -3.41)(277, -3.41)(278, -3.41)(279, -3.41)(280, -3.41)(281, -3.41)(282, -3.41)(283, -3.41)(284, -3.41)(285, -3.41)(286, -3.41)(287, -3.41)(288, -3.41)(289, -3.41)(290, -3.41)(291, -3.41)(292, -3.41)(293, -3.41)(294, -3.41)(295, -3.41)(296, -3.41)(297, -3.41)(298, -3.41)(299, -3.41)(300, -3.41)(301, -3.41)(302, -3.41)(303, -3.41)(304, -3.41)(305, -3.41)(306, -3.41)(307, -3.41)(308, -3.41)(309, -3.41)(310, -3.41)(311, -3.41)(312, -3.41)(313, -3.41)(314, -3.41)(315, -3.41)(316, -3.41)(317, -3.41)(318, -3.41)(319, -3.41)(320, -3.41)(321, -5.41)(322, -5.41)(323, -5.41)(324, -5.41)(325, -5.41)(326, -5.41)(327, -5.41)(328, -5.41)(329, -5.41)(330, -5.41)(331, -5.41)(332, -5.41)(333, -5.41)(334, -5.41)(335, -5.41)(336, -5.41)(337, -5.41)(338, -5.41)(339, -5.41)(340, -5.41)(341, -5.41)(342, -5.41)(343, -5.41)(344, -5.41)(345, -5.41)(346, -5.41)(347, -5.41)(348, -5.41)(349, -5.41)(350, -5.41)(351, -5.41)(352, -5.41)(353, -5.41)(354, -5.41)(355, -5.41)(356, -5.41)(357, -5.41)(358, -5.41)(359, -5.41)(360, -5.41)(361, -5.41)(362, -5.41)(363, -5.41)(364, -5.41)(365, -5.41)(366, -5.41)(367, -5.41)(368, -5.41)(369, -5.41)(370, -5.41)(371, -5.41)(372, -5.41)(373, -5.41)(374, -5.41)(375, -5.41)(376, -5.41)(377, -5.41)(378, -5.41)(379, -5.41)(380, -5.41)(381, -5.41)(382, -5.41)(383, -5.41)(384, -5.41)(385, -1.58)(386, -1.58)(387, -1.58)(388, -1.58)(389, -1.58)(390, -1.58)(391, -1.58)(392, -1.58)(393, -1.58)(394, -1.58)(395, -1.58)(396, -1.58)(397, -1.58)(398, -1.58)(399, -1.58)(400, -1.58)(401, -1.58)(402, -1.58)(403, -1.58)(404, -1.58)(405, -1.58)(406, -1.58)(407, -1.58)(408, -1.58)(409, -1.58)(410, -1.58)(411, -1.58)(412, -1.58)(413, -1.58)(414, -1.58)(415, -1.58)(416, -1.58)(417, -1.58)(418, -1.58)(419, -1.58)(420, -1.58)(421, -1.58)(422, -1.58)(423, -1.58)(424, -1.58)(425, -1.58)(426, -1.58)(427, -1.58)(428, -1.58)(429, -1.58)(430, -1.58)(431, -1.58)(432, -1.58)(433, -1.58)(434, -1.58)(435, -1.58)(436, -1.58)(437, -1.58)(438, -1.58)(439, -1.58)(440, -1.58)(441, -1.58)(442, -1.58)(443, -1.58)(444, -1.58)(445, -1.58)(446, -1.58)(447, -1.58)(448, -1.58)(449, -1.58)(450, -1.58)(451, -1.58)(452, -1.58)(453, -1.58)(454, -1.58)(455, -1.58)(456, -1.58)(457, -1.58)(458, -1.58)(459, -1.58)(460, -1.58)(461, -1.58)(462, -1.58)(463, -1.58)(464, -1.58)(465, -1.58)(466, -1.58)(467, -1.58)(468, -1.58)(469, -1.58)(470, -1.58)(471, -1.58)(472, -1.58)(473, -1.58)(474, -1.58)(475, -1.58)(476, -1.58)(477, -1.58)(478, -1.58)(479, -1.58)(480, -1.58)(481, -1.58)(482, -1.58)(483, -1.58)(484, -1.58)(485, -1.58)(486, -1.58)(487, -1.58)(488, -1.58)(489, -1.58)(490, -1.58)(491, -1.58)(492, -1.58)(493, -1.58)(494, -1.58)(495, -1.58)(496, -1.58)(497, -1.58)(498, -1.58)(499, -1.58)(500, -1.58)(501, -1.58)(502, -1.58)(503, -1.58)(504, -1.58)(505, -1.58)(506, -1.58)(507, -1.58)(508, -1.58)(509, -1.58)(510, -1.58)(511, -1.58)(512, -1.58)}

\def\recaaa{(1, 2.66)(2, 2.57)(3, 2.47)(4, 2.56)(5, 2.43)(6, 2.63)(7, 2.73)(8, 2.46)(9, 2.69)(10, 2.60)(11, 2.59)(12, 2.76)(13, 2.73)(14, 2.64)(15, 2.60)(16, 2.75)(17, 2.77)(18, 2.67)(19, 2.68)(20, 2.62)(21, 2.71)(22, 2.71)(23, 2.63)(24, 2.85)(25, 2.84)(26, 2.62)(27, 2.80)(28, 2.51)(29, 2.56)(30, 2.80)(31, 2.91)(32, 2.62)(33, 2.76)(34, 2.75)(35, 2.79)(36, 2.91)(37, 2.84)(38, 2.95)(39, 2.87)(40, 2.86)(41, 2.94)(42, 2.81)(43, 2.68)(44, 2.68)(45, 2.77)(46, 2.84)(47, 2.82)(48, 2.86)(49, 2.76)(50, 2.88)(51, 2.79)(52, 2.80)(53, 2.77)(54, 2.71)(55, 3.02)(56, 2.76)(57, 2.96)(58, 2.89)(59, 2.85)(60, 2.97)(61, 2.86)(62, 2.68)(63, 2.90)(64, 2.95)(65, 2.36)(66, 2.29)(67, 2.48)(68, 2.37)(69, 2.40)(70, 2.32)(71, 2.22)(72, 2.19)(73, 2.32)(74, 2.40)(75, 2.26)(76, 2.25)(77, 2.31)(78, 2.31)(79, 2.34)(80, 2.41)(81, 2.29)(82, 2.39)(83, 2.31)(84, 2.22)(85, 2.23)(86, 2.36)(87, 2.24)(88, 2.23)(89, 2.23)(90, 2.35)(91, 2.30)(92, 2.44)(93, 2.27)(94, 2.33)(95, 2.39)(96, 2.50)(97, 3.13)(98, 2.98)(99, 3.06)(100, 2.98)(101, 2.91)(102, 3.11)(103, 3.05)(104, 2.97)(105, 3.12)(106, 3.16)(107, 3.11)(108, 3.31)(109, 2.92)(110, 2.99)(111, 3.08)(112, 3.29)(113, 3.89)(114, 3.82)(115, 3.95)(116, 4.10)(117, 3.98)(118, 4.03)(119, 3.98)(120, 3.95)(121, 4.46)(122, 4.31)(123, 4.51)(124, 4.63)(125, 4.37)(126, 4.56)(127, 4.52)(128, 4.59)(129, 3.76)(130, 3.93)(131, 3.87)(132, 3.92)(133, 3.82)(134, 3.94)(135, 3.88)(136, 3.76)(137, 3.74)(138, 3.85)(139, 3.97)(140, 3.87)(141, 3.89)(142, 3.91)(143, 3.84)(144, 3.91)(145, 3.96)(146, 3.90)(147, 3.79)(148, 3.85)(149, 3.88)(150, 3.72)(151, 3.93)(152, 4.01)(153, 3.85)(154, 3.84)(155, 3.96)(156, 3.77)(157, 3.91)(158, 3.78)(159, 3.89)(160, 3.84)(161, 3.85)(162, 3.86)(163, 3.93)(164, 3.97)(165, 3.88)(166, 3.69)(167, 3.91)(168, 3.71)(169, 3.90)(170, 3.77)(171, 3.91)(172, 3.76)(173, 3.65)(174, 3.99)(175, 4.02)(176, 3.88)(177, 4.05)(178, 3.99)(179, 3.76)(180, 3.85)(181, 4.04)(182, 3.80)(183, 3.86)(184, 3.88)(185, 3.80)(186, 3.79)(187, 4.00)(188, 3.71)(189, 3.96)(190, 3.79)(191, 3.95)(192, 3.88)(193, 2.29)(194, 2.36)(195, 2.35)(196, 2.47)(197, 2.23)(198, 2.43)(199, 2.35)(200, 2.34)(201, 2.34)(202, 2.43)(203, 2.28)(204, 2.34)(205, 2.36)(206, 2.41)(207, 2.43)(208, 2.29)(209, 2.31)(210, 2.38)(211, 2.27)(212, 2.28)(213, 2.30)(214, 2.20)(215, 2.34)(216, 2.23)(217, 2.42)(218, 2.35)(219, 2.30)(220, 2.39)(221, 2.26)(222, 2.20)(223, 2.46)(224, 2.18)(225, 1.99)(226, 2.09)(227, 2.29)(228, 2.15)(229, 1.88)(230, 2.26)(231, 2.12)(232, 2.05)(233, 2.15)(234, 2.10)(235, 2.06)(236, 2.06)(237, 2.03)(238, 1.94)(239, 2.04)(240, 2.09)(241, 2.15)(242, 2.05)(243, 2.19)(244, 1.96)(245, 2.16)(246, 2.03)(247, 2.05)(248, 2.07)(249, 2.20)(250, 1.98)(251, 2.06)(252, 2.02)(253, 2.10)(254, 2.21)(255, 2.12)(256, 1.97)(257, -3.92)(258, -4.07)(259, -4.06)(260, -3.98)(261, -3.97)(262, -3.97)(263, -4.01)(264, -4.02)(265, -4.08)(266, -4.01)(267, -3.93)(268, -3.86)(269, -3.87)(270, -4.02)(271, -3.95)(272, -3.92)(273, -3.90)(274, -4.01)(275, -4.05)(276, -4.20)(277, -4.15)(278, -3.95)(279, -3.88)(280, -4.06)(281, -3.99)(282, -3.99)(283, -4.02)(284, -4.16)(285, -3.98)(286, -4.11)(287, -4.07)(288, -4.17)(289, -4.75)(290, -4.55)(291, -4.63)(292, -4.52)(293, -4.58)(294, -4.50)(295, -4.49)(296, -4.63)(297, -4.68)(298, -4.54)(299, -4.62)(300, -4.60)(301, -4.60)(302, -4.66)(303, -4.62)(304, -4.63)(305, -5.05)(306, -4.96)(307, -5.04)(308, -4.96)(309, -4.88)(310, -4.86)(311, -4.99)(312, -4.96)(313, -4.95)(314, -4.90)(315, -5.14)(316, -4.92)(317, -4.93)(318, -5.19)(319, -5.00)(320, -4.91)(321, -4.46)(322, -4.34)(323, -4.51)(324, -4.50)(325, -4.42)(326, -4.46)(327, -4.58)(328, -4.47)(329, -4.33)(330, -4.35)(331, -4.40)(332, -4.29)(333, -4.52)(334, -4.51)(335, -4.15)(336, -4.39)(337, -4.54)(338, -4.42)(339, -4.57)(340, -4.41)(341, -4.41)(342, -4.48)(343, -4.39)(344, -4.40)(345, -4.40)(346, -4.40)(347, -4.42)(348, -4.41)(349, -4.28)(350, -4.40)(351, -4.36)(352, -4.41)(353, -4.46)(354, -4.23)(355, -4.44)(356, -4.51)(357, -4.43)(358, -4.46)(359, -4.47)(360, -4.65)(361, -4.35)(362, -4.41)(363, -4.25)(364, -4.53)(365, -4.43)(366, -4.54)(367, -4.47)(368, -4.39)(369, -4.52)(370, -4.22)(371, -4.28)(372, -4.43)(373, -4.42)(374, -4.41)(375, -4.30)(376, -4.40)(377, -4.35)(378, -4.42)(379, -4.47)(380, -4.46)(381, -4.43)(382, -4.50)(383, -4.38)(384, -4.38)(385, -0.63)(386, -0.65)(387, -0.48)(388, -0.56)(389, -0.54)(390, -0.43)(391, -0.49)(392, -0.30)(393, -0.51)(394, -0.17)(395, -0.22)(396, -0.34)(397, -0.36)(398, -0.42)(399, -0.28)(400, -0.38)(401, -0.59)(402, -0.49)(403, -0.40)(404, -0.36)(405, -0.28)(406, -0.33)(407, -0.41)(408, -0.42)(409, -0.53)(410, -0.52)(411, -0.47)(412, -0.59)(413, -0.34)(414, -0.36)(415, -0.36)(416, -0.33)(417, -0.29)(418, -0.23)(419, -0.39)(420, -0.36)(421, -0.27)(422, -0.31)(423, -0.43)(424, -0.31)(425, -0.20)(426, -0.48)(427, -0.34)(428, -0.30)(429, -0.37)(430, -0.37)(431, -0.36)(432, -0.36)(433, -0.25)(434, -0.40)(435, -0.23)(436, -0.38)(437, -0.49)(438, -0.36)(439, -0.33)(440, -0.18)(441, -0.30)(442, -0.37)(443, -0.39)(444, -0.31)(445, -0.40)(446, -0.23)(447, -0.35)(448, -0.47)(449, -0.32)(450, -0.38)(451, -0.30)(452, -0.31)(453, -0.47)(454, -0.40)(455, -0.33)(456, -0.31)(457, -0.26)(458, -0.45)(459, -0.26)(460, -0.42)(461, -0.27)(462, -0.24)(463, -0.14)(464, -0.35)(465, -0.27)(466, -0.39)(467, -0.22)(468, -0.32)(469, -0.51)(470, -0.27)(471, -0.39)(472, -0.44)(473, -0.35)(474, -0.32)(475, -0.33)(476, -0.29)(477, -0.26)(478, -0.34)(479, -0.45)(480, -0.36)(481, -0.33)(482, -0.34)(483, -0.49)(484, -0.48)(485, -0.49)(486, -0.47)(487, -0.26)(488, -0.40)(489, -0.39)(490, -0.51)(491, -0.55)(492, -0.35)(493, -0.38)(494, -0.61)(495, -0.34)(496, -0.42)(497, -0.33)(498, -0.27)(499, -0.46)(500, -0.53)(501, -0.56)(502, -0.51)(503, -0.39)(504, -0.43)(505, -0.39)(506, -0.47)(507, -0.31)(508, -0.47)(509, -0.39)(510, -0.50)(511, -0.32)(512, -0.46)}

\def\recaba{(1, 2.58)(2, 2.42)(3, 2.28)(4, 2.40)(5, 2.22)(6, 2.53)(7, 2.68)(8, 2.26)(9, 2.57)(10, 2.44)(11, 2.43)(12, 2.68)(13, 2.64)(14, 2.49)(15, 2.44)(16, 2.67)(17, 2.68)(18, 2.53)(19, 2.54)(20, 2.44)(21, 2.59)(22, 2.59)(23, 2.46)(24, 2.80)(25, 2.78)(26, 2.44)(27, 2.72)(28, 2.28)(29, 2.36)(30, 2.73)(31, 2.89)(32, 2.44)(33, 2.59)(34, 2.58)(35, 2.64)(36, 2.81)(37, 2.72)(38, 2.89)(39, 2.76)(40, 2.74)(41, 2.88)(42, 2.68)(43, 2.47)(44, 2.47)(45, 2.61)(46, 2.73)(47, 2.69)(48, 2.77)(49, 2.63)(50, 2.81)(51, 2.68)(52, 2.70)(53, 2.64)(54, 2.56)(55, 3.04)(56, 2.64)(57, 2.97)(58, 2.86)(59, 2.80)(60, 2.99)(61, 2.82)(62, 2.55)(63, 2.90)(64, 2.97)(65, 3.01)(66, 2.90)(67, 3.20)(68, 3.04)(69, 3.08)(70, 2.97)(71, 2.82)(72, 2.77)(73, 2.96)(74, 3.09)(75, 2.88)(76, 2.86)(77, 2.95)(78, 2.95)(79, 3.01)(80, 3.12)(81, 2.91)(82, 3.07)(83, 2.96)(84, 2.82)(85, 2.83)(86, 3.03)(87, 2.85)(88, 2.84)(89, 2.84)(90, 3.02)(91, 2.95)(92, 3.17)(93, 2.92)(94, 3.01)(95, 3.10)(96, 3.28)(97, 3.33)(98, 3.10)(99, 3.23)(100, 3.11)(101, 2.99)(102, 3.31)(103, 3.21)(104, 3.09)(105, 3.29)(106, 3.35)(107, 3.27)(108, 3.58)(109, 2.97)(110, 3.07)(111, 3.19)(112, 3.50)(113, 3.34)(114, 3.24)(115, 3.42)(116, 3.65)(117, 3.47)(118, 3.55)(119, 3.47)(120, 3.44)(121, 3.75)(122, 3.52)(123, 3.80)(124, 3.97)(125, 3.60)(126, 3.88)(127, 3.81)(128, 3.92)(129, 3.75)(130, 4.01)(131, 3.91)(132, 3.99)(133, 3.83)(134, 4.00)(135, 3.90)(136, 3.72)(137, 3.68)(138, 3.87)(139, 4.05)(140, 3.90)(141, 3.93)(142, 3.95)(143, 3.84)(144, 3.96)(145, 4.00)(146, 3.92)(147, 3.75)(148, 3.83)(149, 3.87)(150, 3.62)(151, 3.95)(152, 4.08)(153, 3.81)(154, 3.78)(155, 3.97)(156, 3.68)(157, 3.90)(158, 3.70)(159, 3.86)(160, 3.79)(161, 3.71)(162, 3.72)(163, 3.83)(164, 3.89)(165, 3.75)(166, 3.45)(167, 3.78)(168, 3.49)(169, 3.78)(170, 3.58)(171, 3.80)(172, 3.55)(173, 3.38)(174, 3.90)(175, 3.94)(176, 3.72)(177, 4.00)(178, 3.91)(179, 3.54)(180, 3.67)(181, 3.97)(182, 3.60)(183, 3.69)(184, 3.72)(185, 3.58)(186, 3.57)(187, 3.88)(188, 3.43)(189, 3.81)(190, 3.54)(191, 3.78)(192, 3.68)(193, 2.40)(194, 2.52)(195, 2.49)(196, 2.67)(197, 2.29)(198, 2.60)(199, 2.47)(200, 2.45)(201, 2.46)(202, 2.60)(203, 2.37)(204, 2.45)(205, 2.48)(206, 2.55)(207, 2.58)(208, 2.35)(209, 2.27)(210, 2.38)(211, 2.20)(212, 2.21)(213, 2.24)(214, 2.08)(215, 2.29)(216, 2.12)(217, 2.41)(218, 2.30)(219, 2.23)(220, 2.35)(221, 2.13)(222, 2.04)(223, 2.43)(224, 1.98)(225, 1.31)(226, 1.46)(227, 1.75)(228, 1.52)(229, 1.09)(230, 1.67)(231, 1.44)(232, 1.33)(233, 1.44)(234, 1.35)(235, 1.29)(236, 1.28)(237, 1.22)(238, 1.07)(239, 1.21)(240, 1.28)(241, 1.28)(242, 1.13)(243, 1.33)(244, 0.96)(245, 1.26)(246, 1.06)(247, 1.08)(248, 1.09)(249, 1.29)(250, 0.94)(251, 1.04)(252, 0.98)(253, 1.10)(254, 1.25)(255, 1.10)(256, 0.84)(257, -3.39)(258, -3.62)(259, -3.62)(260, -3.50)(261, -3.48)(262, -3.48)(263, -3.55)(264, -3.57)(265, -3.68)(266, -3.57)(267, -3.46)(268, -3.34)(269, -3.37)(270, -3.60)(271, -3.49)(272, -3.46)(273, -3.47)(274, -3.65)(275, -3.73)(276, -3.96)(277, -3.89)(278, -3.59)(279, -3.48)(280, -3.76)(281, -3.67)(282, -3.68)(283, -3.72)(284, -3.95)(285, -3.67)(286, -3.87)(287, -3.82)(288, -3.97)(289, -4.25)(290, -3.96)(291, -4.08)(292, -3.91)(293, -4.00)(294, -3.87)(295, -3.85)(296, -4.07)(297, -4.15)(298, -3.93)(299, -4.05)(300, -4.02)(301, -4.03)(302, -4.13)(303, -4.06)(304, -4.08)(305, -4.36)(306, -4.21)(307, -4.34)(308, -4.22)(309, -4.07)(310, -4.03)(311, -4.22)(312, -4.19)(313, -4.19)(314, -4.11)(315, -4.47)(316, -4.12)(317, -4.16)(318, -4.55)(319, -4.25)(320, -4.10)(321, -4.13)(322, -3.95)(323, -4.20)(324, -4.19)(325, -4.06)(326, -4.12)(327, -4.30)(328, -4.13)(329, -3.91)(330, -3.93)(331, -4.01)(332, -3.84)(333, -4.18)(334, -4.17)(335, -3.60)(336, -3.97)(337, -4.23)(338, -4.03)(339, -4.27)(340, -4.01)(341, -4.01)(342, -4.11)(343, -3.97)(344, -3.97)(345, -3.96)(346, -3.96)(347, -3.99)(348, -3.98)(349, -3.77)(350, -3.95)(351, -3.88)(352, -3.95)(353, -3.92)(354, -3.57)(355, -3.89)(356, -3.98)(357, -3.85)(358, -3.90)(359, -3.90)(360, -4.18)(361, -3.70)(362, -3.78)(363, -3.55)(364, -3.96)(365, -3.80)(366, -3.96)(367, -3.86)(368, -3.72)(369, -3.87)(370, -3.40)(371, -3.48)(372, -3.71)(373, -3.70)(374, -3.67)(375, -3.49)(376, -3.64)(377, -3.55)(378, -3.66)(379, -3.73)(380, -3.71)(381, -3.66)(382, -3.76)(383, -3.56)(384, -3.56)(385, -1.73)(386, -1.76)(387, -1.49)(388, -1.60)(389, -1.58)(390, -1.41)(391, -1.49)(392, -1.20)(393, -1.55)(394, -1.02)(395, -1.10)(396, -1.28)(397, -1.30)(398, -1.40)(399, -1.18)(400, -1.34)(401, -1.64)(402, -1.48)(403, -1.34)(404, -1.29)(405, -1.16)(406, -1.23)(407, -1.37)(408, -1.38)(409, -1.53)(410, -1.51)(411, -1.45)(412, -1.62)(413, -1.25)(414, -1.30)(415, -1.29)(416, -1.25)(417, -1.14)(418, -1.03)(419, -1.28)(420, -1.24)(421, -1.09)(422, -1.16)(423, -1.34)(424, -1.16)(425, -0.99)(426, -1.42)(427, -1.20)(428, -1.14)(429, -1.26)(430, -1.25)(431, -1.24)(432, -1.23)(433, -1.07)(434, -1.29)(435, -1.03)(436, -1.26)(437, -1.43)(438, -1.23)(439, -1.19)(440, -0.96)(441, -1.14)(442, -1.25)(443, -1.29)(444, -1.17)(445, -1.29)(446, -1.02)(447, -1.21)(448, -1.41)(449, -1.17)(450, -1.26)(451, -1.13)(452, -1.15)(453, -1.41)(454, -1.29)(455, -1.18)(456, -1.15)(457, -1.07)(458, -1.37)(459, -1.07)(460, -1.32)(461, -1.09)(462, -1.04)(463, -0.89)(464, -1.22)(465, -1.09)(466, -1.28)(467, -1.02)(468, -1.16)(469, -1.46)(470, -1.09)(471, -1.28)(472, -1.37)(473, -1.22)(474, -1.17)(475, -1.18)(476, -1.12)(477, -1.07)(478, -1.19)(479, -1.37)(480, -1.23)(481, -1.10)(482, -1.12)(483, -1.35)(484, -1.32)(485, -1.35)(486, -1.33)(487, -0.99)(488, -1.21)(489, -1.18)(490, -1.37)(491, -1.44)(492, -1.13)(493, -1.17)(494, -1.53)(495, -1.10)(496, -1.23)(497, -1.13)(498, -1.03)(499, -1.32)(500, -1.42)(501, -1.47)(502, -1.37)(503, -1.18)(504, -1.25)(505, -1.18)(506, -1.30)(507, -1.05)(508, -1.29)(509, -1.17)(510, -1.35)(511, -1.07)(512, -1.27)}

\def\recaca{(1, 2.71)(2, 3.24)(3, 3.74)(4, 3.52)(5, 3.99)(6, 3.02)(7, 2.57)(8, 3.95)(9, 3.04)(10, 3.46)(11, 3.62)(12, 2.91)(13, 2.98)(14, 3.49)(15, 3.80)(16, 3.11)(17, 3.04)(18, 3.59)(19, 3.56)(20, 3.97)(21, 3.58)(22, 3.54)(23, 3.98)(24, 2.95)(25, 2.96)(26, 4.04)(27, 3.17)(28, 4.59)(29, 4.36)(30, 3.23)(31, 2.72)(32, 4.19)(33, 3.92)(34, 3.91)(35, 3.73)(36, 3.24)(37, 3.50)(38, 3.05)(39, 3.38)(40, 3.50)(41, 3.02)(42, 3.68)(43, 4.44)(44, 4.32)(45, 3.99)(46, 3.58)(47, 3.66)(48, 3.47)(49, 3.88)(50, 3.37)(51, 3.66)(52, 3.53)(53, 3.70)(54, 3.94)(55, 2.40)(56, 3.75)(57, 2.67)(58, 3.06)(59, 3.21)(60, 2.59)(61, 3.17)(62, 4.02)(63, 2.91)(64, 2.66)(65, 3.10)(66, 3.55)(67, 2.55)(68, 3.07)(69, 3.00)(70, 3.30)(71, 3.85)(72, 4.03)(73, 3.31)(74, 2.78)(75, 3.51)(76, 3.57)(77, 3.35)(78, 3.37)(79, 3.18)(80, 2.82)(81, 3.50)(82, 2.98)(83, 3.40)(84, 3.81)(85, 3.77)(86, 3.21)(87, 3.69)(88, 3.79)(89, 3.77)(90, 3.22)(91, 3.49)(92, 2.72)(93, 3.53)(94, 3.31)(95, 3.04)(96, 2.36)(97, 2.72)(98, 3.42)(99, 2.96)(100, 3.41)(101, 3.77)(102, 2.82)(103, 3.25)(104, 3.64)(105, 3.01)(106, 2.82)(107, 3.13)(108, 2.10)(109, 4.05)(110, 3.72)(111, 3.35)(112, 2.39)(113, 3.30)(114, 3.54)(115, 2.98)(116, 2.29)(117, 2.78)(118, 2.58)(119, 2.86)(120, 2.72)(121, 2.12)(122, 2.83)(123, 1.96)(124, 1.40)(125, 2.59)(126, 1.78)(127, 1.96)(128, 1.62)(129, 3.19)(130, 2.31)(131, 2.57)(132, 2.41)(133, 2.84)(134, 2.37)(135, 2.62)(136, 3.21)(137, 3.32)(138, 2.68)(139, 2.06)(140, 2.48)(141, 2.36)(142, 2.29)(143, 2.56)(144, 2.17)(145, 2.16)(146, 2.38)(147, 2.90)(148, 2.63)(149, 2.62)(150, 3.44)(151, 2.34)(152, 1.87)(153, 2.65)(154, 2.83)(155, 2.20)(156, 3.07)(157, 2.41)(158, 2.99)(159, 2.51)(160, 2.67)(161, 2.78)(162, 2.72)(163, 2.32)(164, 2.15)(165, 2.52)(166, 3.51)(167, 2.41)(168, 3.38)(169, 2.44)(170, 2.96)(171, 2.30)(172, 3.16)(173, 3.65)(174, 1.95)(175, 1.86)(176, 2.55)(177, 1.64)(178, 1.98)(179, 3.01)(180, 2.59)(181, 1.67)(182, 2.76)(183, 2.50)(184, 2.32)(185, 2.71)(186, 2.74)(187, 1.75)(188, 3.15)(189, 1.91)(190, 2.76)(191, 2.00)(192, 2.26)(193, 2.84)(194, 2.35)(195, 2.38)(196, 1.88)(197, 3.04)(198, 2.03)(199, 2.48)(200, 2.53)(201, 2.60)(202, 2.07)(203, 2.71)(204, 2.51)(205, 2.35)(206, 2.05)(207, 2.06)(208, 2.78)(209, 2.56)(210, 2.24)(211, 2.71)(212, 2.67)(213, 2.53)(214, 2.99)(215, 2.28)(216, 2.75)(217, 1.89)(218, 2.18)(219, 2.26)(220, 1.80)(221, 2.38)(222, 2.59)(223, 1.33)(224, 2.63)(225, 2.68)(226, 2.21)(227, 1.16)(228, 1.84)(229, 3.10)(230, 1.18)(231, 1.91)(232, 2.16)(233, 1.69)(234, 1.89)(235, 1.89)(236, 1.90)(237, 1.93)(238, 2.33)(239, 1.80)(240, 1.55)(241, 1.18)(242, 1.51)(243, 0.78)(244, 1.81)(245, 0.68)(246, 1.28)(247, 1.02)(248, 1.00)(249, 0.27)(250, 1.25)(251, 0.83)(252, 0.85)(253, 0.35)(254, -0.28)(255, 0.02)(256, 0.74)(257, -0.24)(258, 0.41)(259, 0.29)(260, -0.23)(261, -0.42)(262, -0.51)(263, -0.35)(264, -0.34)(265, -0.16)(266, -0.55)(267, -1.04)(268, -1.48)(269, -1.47)(270, -0.82)(271, -1.31)(272, -1.49)(273, -1.79)(274, -1.33)(275, -1.18)(276, -0.44)(277, -0.74)(278, -1.80)(279, -2.17)(280, -1.39)(281, -1.87)(282, -1.80)(283, -1.75)(284, -1.18)(285, -2.15)(286, -1.62)(287, -1.81)(288, -1.45)(289, -1.16)(290, -2.11)(291, -1.85)(292, -2.42)(293, -2.20)(294, -2.70)(295, -2.77)(296, -2.14)(297, -2.00)(298, -2.81)(299, -2.35)(300, -2.51)(301, -2.49)(302, -2.24)(303, -2.47)(304, -2.47)(305, -1.86)(306, -2.39)(307, -2.08)(308, -2.45)(309, -3.00)(310, -3.29)(311, -2.69)(312, -2.80)(313, -2.85)(314, -3.11)(315, -1.99)(316, -3.21)(317, -3.08)(318, -1.96)(319, -2.89)(320, -3.54)(321, -2.98)(322, -3.66)(323, -2.87)(324, -2.92)(325, -3.42)(326, -3.23)(327, -2.73)(328, -3.28)(329, -3.90)(330, -3.84)(331, -3.59)(332, -4.14)(333, -3.11)(334, -3.18)(335, -4.94)(336, -3.80)(337, -3.20)(338, -3.81)(339, -3.14)(340, -4.01)(341, -3.94)(342, -3.65)(343, -4.08)(344, -4.05)(345, -4.04)(346, -4.05)(347, -3.97)(348, -3.93)(349, -4.65)(350, -4.05)(351, -4.32)(352, -4.11)(353, -3.92)(354, -5.01)(355, -3.97)(356, -3.59)(357, -4.03)(358, -3.84)(359, -3.83)(360, -2.89)(361, -4.32)(362, -4.09)(363, -4.89)(364, -3.54)(365, -3.98)(366, -3.47)(367, -3.76)(368, -4.17)(369, -3.47)(370, -4.93)(371, -4.64)(372, -3.88)(373, -3.99)(374, -4.02)(375, -4.54)(376, -4.02)(377, -4.14)(378, -3.70)(379, -3.45)(380, -3.46)(381, -3.55)(382, -3.25)(383, -3.80)(384, -3.80)(385, -3.27)(386, -3.13)(387, -3.90)(388, -3.60)(389, -3.52)(390, -4.04)(391, -3.79)(392, -4.60)(393, -3.36)(394, -4.94)(395, -4.70)(396, -4.05)(397, -3.96)(398, -3.65)(399, -4.33)(400, -3.76)(401, -3.01)(402, -3.48)(403, -3.84)(404, -3.94)(405, -4.27)(406, -4.10)(407, -3.59)(408, -3.54)(409, -3.06)(410, -3.07)(411, -3.22)(412, -2.75)(413, -3.77)(414, -3.60)(415, -3.60)(416, -3.67)(417, -3.75)(418, -4.13)(419, -3.30)(420, -3.38)(421, -3.88)(422, -3.66)(423, -3.07)(424, -3.70)(425, -4.16)(426, -2.75)(427, -3.44)(428, -3.62)(429, -3.17)(430, -3.16)(431, -3.17)(432, -3.20)(433, -3.69)(434, -2.98)(435, -3.84)(436, -3.03)(437, -2.52)(438, -3.14)(439, -3.24)(440, -3.97)(441, -3.36)(442, -2.95)(443, -2.83)(444, -3.15)(445, -2.80)(446, -3.62)(447, -2.97)(448, -2.37)(449, -3.15)(450, -2.89)(451, -3.35)(452, -3.24)(453, -2.40)(454, -2.85)(455, -3.16)(456, -3.28)(457, -3.52)(458, -2.52)(459, -3.51)(460, -2.70)(461, -3.38)(462, -3.51)(463, -3.97)(464, -2.91)(465, -3.30)(466, -2.66)(467, -3.49)(468, -3.03)(469, -2.09)(470, -3.23)(471, -2.52)(472, -2.21)(473, -2.68)(474, -2.81)(475, -2.78)(476, -2.96)(477, -3.08)(478, -2.69)(479, -2.10)(480, -2.46)(481, -2.89)(482, -2.77)(483, -2.03)(484, -2.10)(485, -1.94)(486, -1.85)(487, -2.95)(488, -2.18)(489, -2.24)(490, -1.60)(491, -1.34)(492, -2.29)(493, -2.05)(494, -0.88)(495, -2.18)(496, -1.70)(497, -1.94)(498, -2.20)(499, -1.19)(500, -0.78)(501, -0.52)(502, -0.75)(503, -1.27)(504, -0.93)(505, -1.10)(506, -0.61)(507, -1.35)(508, -0.49)(509, -0.69)(510, 0.01)(511, -0.72)(512, 0.04)}

\def\recbaa{(1, 2.78)(2, 2.78)(3, 2.79)(4, 2.79)(5, 2.78)(6, 2.78)(7, 2.78)(8, 2.79)(9, 2.78)(10, 2.79)(11, 2.78)(12, 2.79)(13, 2.78)(14, 2.78)(15, 2.79)(16, 2.78)(17, 2.78)(18, 2.78)(19, 2.78)(20, 2.78)(21, 2.78)(22, 2.79)(23, 2.79)(24, 2.78)(25, 2.78)(26, 2.79)(27, 2.79)(28, 2.78)(29, 2.78)(30, 2.79)(31, 2.79)(32, 2.79)(33, 2.75)(34, 2.75)(35, 2.76)(36, 2.75)(37, 2.76)(38, 2.76)(39, 2.76)(40, 2.77)(41, 2.76)(42, 2.77)(43, 2.77)(44, 2.76)(45, 2.76)(46, 2.77)(47, 2.77)(48, 2.76)(49, 2.82)(50, 2.81)(51, 2.81)(52, 2.82)(53, 2.82)(54, 2.82)(55, 2.81)(56, 2.81)(57, 2.82)(58, 2.81)(59, 2.81)(60, 2.83)(61, 2.81)(62, 2.82)(63, 2.80)(64, 2.83)(65, 3.37)(66, 3.37)(67, 3.37)(68, 3.37)(69, 3.37)(70, 3.36)(71, 3.36)(72, 3.38)(73, 3.38)(74, 3.36)(75, 3.37)(76, 3.38)(77, 3.38)(78, 3.37)(79, 3.37)(80, 3.37)(81, 3.37)(82, 3.38)(83, 3.37)(84, 3.38)(85, 3.38)(86, 3.39)(87, 3.37)(88, 3.38)(89, 3.39)(90, 3.37)(91, 3.37)(92, 3.38)(93, 3.38)(94, 3.37)(95, 3.38)(96, 3.37)(97, 3.37)(98, 3.38)(99, 3.38)(100, 3.37)(101, 3.38)(102, 3.37)(103, 3.38)(104, 3.37)(105, 3.37)(106, 3.37)(107, 3.37)(108, 3.37)(109, 3.38)(110, 3.37)(111, 3.37)(112, 3.37)(113, 3.37)(114, 3.38)(115, 3.37)(116, 3.38)(117, 3.37)(118, 3.37)(119, 3.38)(120, 3.37)(121, 3.37)(122, 3.38)(123, 3.37)(124, 3.37)(125, 3.38)(126, 3.37)(127, 3.37)(128, 3.37)(129, 2.90)(130, 2.90)(131, 2.90)(132, 2.90)(133, 2.90)(134, 2.90)(135, 2.90)(136, 2.90)(137, 2.90)(138, 2.91)(139, 2.89)(140, 2.90)(141, 2.90)(142, 2.90)(143, 2.89)(144, 2.90)(145, 2.89)(146, 2.90)(147, 2.89)(148, 2.90)(149, 2.90)(150, 2.89)(151, 2.90)(152, 2.89)(153, 2.89)(154, 2.89)(155, 2.90)(156, 2.89)(157, 2.89)(158, 2.89)(159, 2.88)(160, 2.89)(161, 2.88)(162, 2.88)(163, 2.88)(164, 2.89)(165, 2.88)(166, 2.88)(167, 2.88)(168, 2.89)(169, 2.89)(170, 2.89)(171, 2.89)(172, 2.89)(173, 2.88)(174, 2.88)(175, 2.88)(176, 2.89)(177, 2.88)(178, 2.88)(179, 2.88)(180, 2.88)(181, 2.88)(182, 2.89)(183, 2.88)(184, 2.89)(185, 2.89)(186, 2.89)(187, 2.88)(188, 2.88)(189, 2.88)(190, 2.89)(191, 2.88)(192, 2.90)(193, 2.89)(194, 2.88)(195, 2.88)(196, 2.90)(197, 2.88)(198, 2.89)(199, 2.89)(200, 2.90)(201, 2.88)(202, 2.89)(203, 2.88)(204, 2.89)(205, 2.89)(206, 2.89)(207, 2.89)(208, 2.89)(209, 2.89)(210, 2.89)(211, 2.90)(212, 2.89)(213, 2.89)(214, 2.90)(215, 2.89)(216, 2.89)(217, 2.89)(218, 2.89)(219, 2.89)(220, 2.89)(221, 2.89)(222, 2.89)(223, 2.89)(224, 2.90)(225, 2.90)(226, 2.90)(227, 2.90)(228, 2.89)(229, 2.89)(230, 2.89)(231, 2.89)(232, 2.90)(233, 2.90)(234, 2.90)(235, 2.89)(236, 2.89)(237, 2.89)(238, 2.89)(239, 2.89)(240, 2.89)(241, 2.89)(242, 2.90)(243, 2.89)(244, 2.89)(245, 2.89)(246, 2.89)(247, 2.88)(248, 2.88)(249, 2.88)(250, 2.89)(251, 2.88)(252, 2.89)(253, 2.88)(254, 2.89)(255, 2.88)(256, 2.89)(257, -3.98)(258, -3.98)(259, -3.98)(260, -3.98)(261, -3.99)(262, -3.99)(263, -3.99)(264, -3.99)(265, -3.99)(266, -3.98)(267, -3.99)(268, -3.98)(269, -3.98)(270, -3.98)(271, -3.99)(272, -3.99)(273, -3.99)(274, -3.99)(275, -3.99)(276, -3.98)(277, -3.98)(278, -3.99)(279, -3.97)(280, -3.99)(281, -3.98)(282, -3.98)(283, -3.98)(284, -3.98)(285, -3.99)(286, -3.98)(287, -3.99)(288, -3.99)(289, -4.02)(290, -4.03)(291, -4.02)(292, -4.03)(293, -4.03)(294, -4.02)(295, -4.03)(296, -4.03)(297, -4.02)(298, -4.02)(299, -4.02)(300, -4.02)(301, -4.03)(302, -4.03)(303, -4.03)(304, -4.03)(305, -4.03)(306, -4.02)(307, -4.03)(308, -4.03)(309, -4.03)(310, -4.02)(311, -4.02)(312, -4.03)(313, -4.03)(314, -4.03)(315, -4.03)(316, -4.02)(317, -4.02)(318, -4.02)(319, -4.02)(320, -4.03)(321, -4.19)(322, -4.19)(323, -4.21)(324, -4.20)(325, -4.19)(326, -4.20)(327, -4.20)(328, -4.20)(329, -4.20)(330, -4.20)(331, -4.20)(332, -4.20)(333, -4.21)(334, -4.20)(335, -4.20)(336, -4.20)(337, -4.20)(338, -4.20)(339, -4.20)(340, -4.21)(341, -4.19)(342, -4.21)(343, -4.20)(344, -4.20)(345, -4.20)(346, -4.20)(347, -4.20)(348, -4.19)(349, -4.20)(350, -4.20)(351, -4.21)(352, -4.20)(353, -4.20)(354, -4.20)(355, -4.20)(356, -4.20)(357, -4.20)(358, -4.21)(359, -4.20)(360, -4.19)(361, -4.19)(362, -4.20)(363, -4.19)(364, -4.20)(365, -4.20)(366, -4.20)(367, -4.20)(368, -4.20)(369, -4.20)(370, -4.20)(371, -4.20)(372, -4.21)(373, -4.21)(374, -4.20)(375, -4.20)(376, -4.20)(377, -4.20)(378, -4.20)(379, -4.20)(380, -4.19)(381, -4.20)(382, -4.20)(383, -4.19)(384, -4.19)(385, -1.97)(386, -1.97)(387, -1.97)(388, -1.97)(389, -1.96)(390, -1.96)(391, -1.96)(392, -1.97)(393, -1.97)(394, -1.97)(395, -1.96)(396, -1.97)(397, -1.97)(398, -1.96)(399, -1.97)(400, -1.96)(401, -1.94)(402, -1.94)(403, -1.94)(404, -1.95)(405, -1.94)(406, -1.95)(407, -1.95)(408, -1.94)(409, -1.94)(410, -1.94)(411, -1.93)(412, -1.94)(413, -1.94)(414, -1.95)(415, -1.94)(416, -1.95)(417, -1.93)(418, -1.93)(419, -1.93)(420, -1.93)(421, -1.93)(422, -1.94)(423, -1.93)(424, -1.93)(425, -1.94)(426, -1.93)(427, -1.94)(428, -1.92)(429, -1.93)(430, -1.93)(431, -1.93)(432, -1.93)(433, -1.93)(434, -1.93)(435, -1.92)(436, -1.93)(437, -1.91)(438, -1.93)(439, -1.92)(440, -1.93)(441, -1.93)(442, -1.92)(443, -1.93)(444, -1.92)(445, -1.93)(446, -1.93)(447, -1.93)(448, -1.93)(449, -1.94)(450, -1.94)(451, -1.94)(452, -1.94)(453, -1.94)(454, -1.94)(455, -1.95)(456, -1.94)(457, -1.94)(458, -1.94)(459, -1.94)(460, -1.95)(461, -1.95)(462, -1.94)(463, -1.94)(464, -1.94)(465, -1.95)(466, -1.94)(467, -1.94)(468, -1.94)(469, -1.95)(470, -1.95)(471, -1.93)(472, -1.94)(473, -1.93)(474, -1.95)(475, -1.93)(476, -1.94)(477, -1.93)(478, -1.95)(479, -1.93)(480, -1.95)(481, -1.94)(482, -1.94)(483, -1.94)(484, -1.94)(485, -1.93)(486, -1.95)(487, -1.94)(488, -1.93)(489, -1.94)(490, -1.94)(491, -1.95)(492, -1.94)(493, -1.95)(494, -1.94)(495, -1.94)(496, -1.94)(497, -1.94)(498, -1.95)(499, -1.94)(500, -1.94)(501, -1.94)(502, -1.94)(503, -1.94)(504, -1.93)(505, -1.94)(506, -1.94)(507, -1.95)(508, -1.93)(509, -1.94)(510, -1.94)(511, -1.94)(512, -1.95)}

\def\recbba{(1, 2.63)(2, 2.57)(3, 2.76)(4, 2.70)(5, 2.66)(6, 2.66)(7, 2.52)(8, 2.80)(9, 2.65)(10, 2.79)(11, 2.58)(12, 2.87)(13, 2.73)(14, 2.58)(15, 2.77)(16, 2.74)(17, 2.70)(18, 2.79)(19, 2.72)(20, 2.72)(21, 2.67)(22, 2.95)(23, 2.87)(24, 2.76)(25, 2.78)(26, 2.91)(27, 2.84)(28, 2.71)(29, 2.72)(30, 2.82)(31, 2.86)(32, 2.83)(33, 2.75)(34, 2.77)(35, 2.82)(36, 2.63)(37, 2.78)(38, 2.92)(39, 2.89)(40, 2.96)(41, 2.84)(42, 2.91)(43, 2.91)(44, 2.83)(45, 2.83)(46, 2.88)(47, 2.89)(48, 2.82)(49, 2.95)(50, 2.92)(51, 2.90)(52, 3.02)(53, 2.99)(54, 2.97)(55, 2.87)(56, 2.91)(57, 2.97)(58, 2.84)(59, 2.91)(60, 3.27)(61, 2.94)(62, 3.04)(63, 2.74)(64, 3.17)(65, 3.32)(66, 3.28)(67, 3.27)(68, 3.45)(69, 3.31)(70, 3.17)(71, 3.24)(72, 3.51)(73, 3.47)(74, 3.25)(75, 3.30)(76, 3.47)(77, 3.64)(78, 3.38)(79, 3.37)(80, 3.30)(81, 3.36)(82, 3.57)(83, 3.36)(84, 3.55)(85, 3.53)(86, 3.61)(87, 3.36)(88, 3.55)(89, 3.60)(90, 3.37)(91, 3.27)(92, 3.51)(93, 3.40)(94, 3.30)(95, 3.51)(96, 3.37)(97, 3.27)(98, 3.59)(99, 3.56)(100, 3.38)(101, 3.47)(102, 3.38)(103, 3.50)(104, 3.34)(105, 3.30)(106, 3.40)(107, 3.40)(108, 3.34)(109, 3.52)(110, 3.39)(111, 3.32)(112, 3.38)(113, 3.27)(114, 3.44)(115, 3.34)(116, 3.47)(117, 3.32)(118, 3.38)(119, 3.44)(120, 3.31)(121, 3.23)(122, 3.39)(123, 3.35)(124, 3.37)(125, 3.39)(126, 3.37)(127, 3.31)(128, 3.29)(129, 2.89)(130, 2.86)(131, 2.89)(132, 2.79)(133, 2.86)(134, 2.92)(135, 2.89)(136, 2.81)(137, 2.81)(138, 2.96)(139, 2.71)(140, 2.73)(141, 2.79)(142, 2.80)(143, 2.64)(144, 2.76)(145, 2.71)(146, 2.81)(147, 2.71)(148, 2.76)(149, 2.80)(150, 2.68)(151, 2.81)(152, 2.75)(153, 2.75)(154, 2.66)(155, 2.83)(156, 2.78)(157, 2.63)(158, 2.77)(159, 2.59)(160, 2.73)(161, 2.52)(162, 2.65)(163, 2.68)(164, 2.85)(165, 2.61)(166, 2.70)(167, 2.55)(168, 2.85)(169, 2.77)(170, 2.89)(171, 2.86)(172, 2.82)(173, 2.75)(174, 2.62)(175, 2.64)(176, 2.90)(177, 2.73)(178, 2.78)(179, 2.80)(180, 2.80)(181, 2.77)(182, 2.94)(183, 2.74)(184, 2.89)(185, 2.93)(186, 2.93)(187, 2.75)(188, 2.82)(189, 2.71)(190, 2.89)(191, 2.70)(192, 3.03)(193, 2.83)(194, 2.75)(195, 2.79)(196, 3.14)(197, 2.82)(198, 2.94)(199, 2.89)(200, 3.08)(201, 2.78)(202, 2.91)(203, 2.72)(204, 2.93)(205, 3.00)(206, 3.01)(207, 2.84)(208, 2.92)(209, 2.86)(210, 2.89)(211, 3.13)(212, 2.90)(213, 2.96)(214, 3.12)(215, 2.86)(216, 2.89)(217, 2.99)(218, 2.86)(219, 2.93)(220, 2.90)(221, 2.86)(222, 2.89)(223, 2.91)(224, 2.99)(225, 3.07)(226, 3.04)(227, 3.10)(228, 2.92)(229, 2.93)(230, 2.93)(231, 2.93)(232, 3.14)(233, 3.05)(234, 3.04)(235, 2.91)(236, 2.96)(237, 2.93)(238, 2.96)(239, 2.95)(240, 2.86)(241, 2.82)(242, 2.84)(243, 2.64)(244, 2.66)(245, 2.61)(246, 2.62)(247, 2.46)(248, 2.43)(249, 2.21)(250, 2.34)(251, 2.20)(252, 2.35)(253, 2.28)(254, 2.30)(255, 2.21)(256, 2.28)(257, -3.24)(258, -3.30)(259, -3.36)(260, -3.33)(261, -3.52)(262, -3.54)(263, -3.57)(264, -3.56)(265, -3.63)(266, -3.51)(267, -3.63)(268, -3.51)(269, -3.52)(270, -3.55)(271, -3.65)(272, -3.74)(273, -3.71)(274, -3.69)(275, -3.67)(276, -3.56)(277, -3.61)(278, -3.76)(279, -3.43)(280, -3.73)(281, -3.59)(282, -3.50)(283, -3.55)(284, -3.55)(285, -3.77)(286, -3.56)(287, -3.67)(288, -3.73)(289, -3.90)(290, -4.03)(291, -3.89)(292, -3.99)(293, -4.08)(294, -3.93)(295, -4.00)(296, -4.12)(297, -3.97)(298, -3.85)(299, -3.82)(300, -3.99)(301, -4.07)(302, -4.06)(303, -4.04)(304, -4.14)(305, -4.13)(306, -3.92)(307, -4.11)(308, -4.10)(309, -4.10)(310, -4.00)(311, -3.91)(312, -4.16)(313, -4.10)(314, -4.15)(315, -4.05)(316, -4.03)(317, -3.88)(318, -3.96)(319, -3.92)(320, -4.07)(321, -4.08)(322, -4.12)(323, -4.35)(324, -4.22)(325, -4.07)(326, -4.23)(327, -4.24)(328, -4.24)(329, -4.14)(330, -4.21)(331, -4.17)(332, -4.31)(333, -4.37)(334, -4.29)(335, -4.21)(336, -4.17)(337, -4.19)(338, -4.23)(339, -4.13)(340, -4.36)(341, -4.11)(342, -4.33)(343, -4.23)(344, -4.28)(345, -4.29)(346, -4.21)(347, -4.27)(348, -4.07)(349, -4.19)(350, -4.18)(351, -4.31)(352, -4.21)(353, -4.18)(354, -4.26)(355, -4.14)(356, -4.23)(357, -4.21)(358, -4.27)(359, -4.12)(360, -4.01)(361, -4.08)(362, -4.25)(363, -4.07)(364, -4.25)(365, -4.26)(366, -4.26)(367, -4.11)(368, -4.12)(369, -3.99)(370, -4.06)(371, -4.08)(372, -4.16)(373, -4.16)(374, -4.01)(375, -4.15)(376, -4.04)(377, -4.11)(378, -4.07)(379, -4.09)(380, -3.87)(381, -4.04)(382, -4.06)(383, -3.86)(384, -3.85)(385, -2.42)(386, -2.37)(387, -2.43)(388, -2.42)(389, -2.27)(390, -2.27)(391, -2.32)(392, -2.37)(393, -2.41)(394, -2.47)(395, -2.30)(396, -2.51)(397, -2.35)(398, -2.17)(399, -2.40)(400, -2.25)(401, -2.06)(402, -2.11)(403, -2.08)(404, -2.24)(405, -2.05)(406, -2.26)(407, -2.24)(408, -2.03)(409, -2.03)(410, -2.06)(411, -1.86)(412, -2.10)(413, -1.97)(414, -2.15)(415, -2.15)(416, -2.20)(417, -2.11)(418, -2.05)(419, -2.08)(420, -2.06)(421, -2.00)(422, -2.18)(423, -2.15)(424, -2.06)(425, -2.28)(426, -2.02)(427, -2.18)(428, -1.96)(429, -2.13)(430, -2.12)(431, -2.03)(432, -2.00)(433, -2.08)(434, -2.03)(435, -1.98)(436, -2.19)(437, -1.82)(438, -2.11)(439, -1.90)(440, -2.04)(441, -2.08)(442, -1.93)(443, -2.02)(444, -1.87)(445, -2.10)(446, -2.11)(447, -2.17)(448, -2.06)(449, -2.08)(450, -2.09)(451, -2.15)(452, -2.03)(453, -2.11)(454, -2.11)(455, -2.17)(456, -2.13)(457, -2.01)(458, -2.02)(459, -2.15)(460, -2.22)(461, -2.17)(462, -2.00)(463, -2.06)(464, -2.01)(465, -2.20)(466, -2.05)(467, -2.04)(468, -2.11)(469, -2.21)(470, -2.19)(471, -1.90)(472, -2.13)(473, -1.90)(474, -2.23)(475, -1.94)(476, -2.07)(477, -1.90)(478, -2.29)(479, -1.96)(480, -2.17)(481, -2.11)(482, -2.07)(483, -2.03)(484, -2.03)(485, -1.97)(486, -2.20)(487, -2.07)(488, -1.95)(489, -2.11)(490, -2.03)(491, -2.22)(492, -2.02)(493, -2.29)(494, -1.98)(495, -2.04)(496, -2.12)(497, -2.09)(498, -2.26)(499, -2.07)(500, -2.02)(501, -2.01)(502, -2.03)(503, -2.11)(504, -1.89)(505, -2.09)(506, -2.10)(507, -2.14)(508, -1.91)(509, -2.06)(510, -2.07)(511, -2.03)(512, -2.21)}

\def\recbca{(1, 1.52)(2, 1.31)(3, 1.99)(4, 1.95)(5, 1.78)(6, 2.01)(7, 1.72)(8, 2.55)(9, 2.26)(10, 2.86)(11, 2.33)(12, 3.35)(13, 2.96)(14, 2.64)(15, 3.32)(16, 3.19)(17, 3.03)(18, 3.30)(19, 3.16)(20, 3.26)(21, 3.01)(22, 3.98)(23, 3.59)(24, 3.27)(25, 3.32)(26, 3.83)(27, 3.53)(28, 2.99)(29, 3.16)(30, 3.39)(31, 3.61)(32, 3.50)(33, 2.99)(34, 3.01)(35, 3.12)(36, 2.35)(37, 2.90)(38, 3.33)(39, 3.25)(40, 3.43)(41, 3.15)(42, 3.42)(43, 3.45)(44, 3.03)(45, 3.00)(46, 3.17)(47, 3.17)(48, 2.93)(49, 3.21)(50, 3.10)(51, 3.12)(52, 3.66)(53, 3.38)(54, 3.24)(55, 2.77)(56, 2.95)(57, 3.20)(58, 2.70)(59, 2.92)(60, 4.11)(61, 2.91)(62, 3.21)(63, 2.12)(64, 3.42)(65, 2.85)(66, 2.71)(67, 2.51)(68, 3.09)(69, 2.62)(70, 2.24)(71, 2.40)(72, 3.42)(73, 3.22)(74, 2.41)(75, 2.72)(76, 3.32)(77, 4.06)(78, 3.18)(79, 3.09)(80, 2.94)(81, 3.14)(82, 3.91)(83, 3.15)(84, 3.78)(85, 3.68)(86, 4.07)(87, 3.10)(88, 3.79)(89, 3.81)(90, 2.97)(91, 2.56)(92, 3.23)(93, 2.95)(94, 2.59)(95, 3.17)(96, 2.68)(97, 2.40)(98, 3.46)(99, 3.23)(100, 2.59)(101, 2.99)(102, 2.72)(103, 3.32)(104, 2.84)(105, 2.72)(106, 3.16)(107, 3.10)(108, 2.93)(109, 3.60)(110, 3.13)(111, 2.98)(112, 3.34)(113, 2.97)(114, 3.58)(115, 3.29)(116, 3.76)(117, 3.19)(118, 3.42)(119, 3.71)(120, 3.19)(121, 2.85)(122, 3.34)(123, 3.20)(124, 3.32)(125, 3.36)(126, 3.26)(127, 3.04)(128, 3.00)(129, 3.03)(130, 3.09)(131, 3.18)(132, 2.92)(133, 3.12)(134, 3.39)(135, 3.35)(136, 2.97)(137, 2.97)(138, 3.52)(139, 2.75)(140, 2.81)(141, 2.88)(142, 2.83)(143, 2.34)(144, 2.78)(145, 2.64)(146, 3.08)(147, 2.72)(148, 2.83)(149, 2.86)(150, 2.45)(151, 2.90)(152, 2.70)(153, 2.67)(154, 2.45)(155, 3.05)(156, 2.84)(157, 2.25)(158, 2.70)(159, 2.15)(160, 2.52)(161, 1.85)(162, 2.38)(163, 2.49)(164, 3.06)(165, 2.38)(166, 2.73)(167, 2.34)(168, 3.40)(169, 3.10)(170, 3.55)(171, 3.37)(172, 3.27)(173, 3.02)(174, 2.47)(175, 2.58)(176, 3.49)(177, 2.82)(178, 2.98)(179, 3.08)(180, 3.02)(181, 2.80)(182, 3.31)(183, 2.51)(184, 3.13)(185, 3.23)(186, 3.17)(187, 2.51)(188, 2.71)(189, 2.41)(190, 2.97)(191, 2.22)(192, 3.42)(193, 2.70)(194, 2.34)(195, 2.62)(196, 3.83)(197, 2.91)(198, 3.26)(199, 3.12)(200, 3.72)(201, 2.63)(202, 3.05)(203, 2.42)(204, 3.21)(205, 3.44)(206, 3.46)(207, 2.83)(208, 3.10)(209, 2.77)(210, 2.85)(211, 3.66)(212, 2.74)(213, 3.01)(214, 3.49)(215, 2.71)(216, 2.85)(217, 3.20)(218, 2.62)(219, 3.01)(220, 2.83)(221, 2.61)(222, 2.73)(223, 2.88)(224, 3.11)(225, 3.21)(226, 3.11)(227, 3.27)(228, 2.67)(229, 2.76)(230, 2.73)(231, 2.80)(232, 3.67)(233, 3.36)(234, 3.21)(235, 2.86)(236, 3.00)(237, 2.84)(238, 2.92)(239, 2.74)(240, 2.30)(241, 2.80)(242, 2.65)(243, 1.90)(244, 1.84)(245, 1.71)(246, 1.57)(247, 1.00)(248, 0.74)(249, 0.48)(250, 0.78)(251, 0.01)(252, 0.38)(253, 0.01)(254, -0.08)(255, -0.58)(256, -0.29)(257, -0.23)(258, -0.62)(259, -0.93)(260, -0.99)(261, -1.62)(262, -1.75)(263, -1.94)(264, -2.08)(265, -2.34)(266, -2.10)(267, -2.71)(268, -2.35)(269, -2.45)(270, -2.69)(271, -3.02)(272, -3.32)(273, -3.21)(274, -3.22)(275, -3.11)(276, -2.79)(277, -3.02)(278, -3.49)(279, -2.36)(280, -3.40)(281, -2.82)(282, -2.44)(283, -2.69)(284, -2.70)(285, -3.38)(286, -2.74)(287, -3.14)(288, -3.28)(289, -2.90)(290, -3.36)(291, -2.98)(292, -3.28)(293, -3.64)(294, -3.21)(295, -3.46)(296, -3.89)(297, -3.47)(298, -3.15)(299, -3.09)(300, -3.63)(301, -3.88)(302, -3.88)(303, -3.74)(304, -4.19)(305, -4.21)(306, -3.53)(307, -4.31)(308, -4.25)(309, -4.19)(310, -3.84)(311, -3.70)(312, -4.53)(313, -4.36)(314, -4.48)(315, -4.17)(316, -4.09)(317, -3.58)(318, -3.75)(319, -3.64)(320, -4.32)(321, -4.11)(322, -4.13)(323, -5.00)(324, -4.63)(325, -4.29)(326, -4.90)(327, -4.99)(328, -4.99)(329, -4.53)(330, -4.78)(331, -4.70)(332, -5.20)(333, -5.45)(334, -5.16)(335, -4.84)(336, -4.76)(337, -4.89)(338, -5.02)(339, -4.69)(340, -5.60)(341, -4.66)(342, -5.43)(343, -5.00)(344, -5.06)(345, -5.03)(346, -4.69)(347, -4.97)(348, -4.21)(349, -4.67)(350, -4.61)(351, -5.02)(352, -4.75)(353, -4.82)(354, -5.11)(355, -4.81)(356, -5.11)(357, -5.05)(358, -5.24)(359, -4.69)(360, -4.34)(361, -4.45)(362, -5.03)(363, -4.38)(364, -5.01)(365, -4.93)(366, -4.89)(367, -4.40)(368, -4.36)(369, -3.94)(370, -4.17)(371, -4.32)(372, -4.56)(373, -4.50)(374, -4.02)(375, -4.39)(376, -3.93)(377, -4.18)(378, -3.96)(379, -3.94)(380, -3.10)(381, -3.63)(382, -3.65)(383, -2.96)(384, -2.83)(385, -3.12)(386, -2.95)(387, -3.11)(388, -2.97)(389, -2.44)(390, -2.43)(391, -2.55)(392, -2.69)(393, -2.73)(394, -2.98)(395, -2.26)(396, -2.94)(397, -2.43)(398, -1.81)(399, -2.55)(400, -1.94)(401, -1.73)(402, -1.94)(403, -1.78)(404, -2.33)(405, -1.63)(406, -2.36)(407, -2.34)(408, -1.56)(409, -1.63)(410, -1.74)(411, -1.01)(412, -1.88)(413, -1.47)(414, -2.14)(415, -2.10)(416, -2.19)(417, -2.01)(418, -1.85)(419, -1.97)(420, -1.99)(421, -1.79)(422, -2.39)(423, -2.44)(424, -2.15)(425, -2.81)(426, -1.85)(427, -2.45)(428, -1.58)(429, -2.05)(430, -2.01)(431, -1.70)(432, -1.52)(433, -1.83)(434, -1.66)(435, -1.56)(436, -2.27)(437, -0.81)(438, -1.76)(439, -1.07)(440, -1.68)(441, -1.73)(442, -1.17)(443, -1.58)(444, -1.15)(445, -1.98)(446, -1.99)(447, -2.26)(448, -1.90)(449, -1.96)(450, -1.98)(451, -2.30)(452, -1.87)(453, -2.22)(454, -2.04)(455, -2.30)(456, -2.19)(457, -1.82)(458, -1.88)(459, -2.38)(460, -2.61)(461, -2.41)(462, -1.79)(463, -2.03)(464, -1.86)(465, -2.53)(466, -2.03)(467, -1.99)(468, -2.28)(469, -2.66)(470, -2.62)(471, -1.61)(472, -2.40)(473, -1.60)(474, -2.73)(475, -1.70)(476, -2.14)(477, -1.60)(478, -3.09)(479, -1.96)(480, -2.80)(481, -2.63)(482, -2.44)(483, -2.32)(484, -2.28)(485, -2.11)(486, -2.98)(487, -2.50)(488, -2.03)(489, -2.59)(490, -2.34)(491, -2.93)(492, -2.12)(493, -3.09)(494, -1.96)(495, -2.17)(496, -2.37)(497, -2.14)(498, -2.65)(499, -1.92)(500, -1.73)(501, -1.61)(502, -1.53)(503, -1.73)(504, -0.83)(505, -1.47)(506, -1.35)(507, -1.43)(508, -0.54)(509, -0.97)(510, -0.81)(511, -0.51)(512, -1.01)}

\def\reccaa{(1, 2.86)(2, 2.86)(3, 2.86)(4, 2.87)(5, 2.87)(6, 2.86)(7, 2.86)(8, 2.86)(9, 2.87)(10, 2.86)(11, 2.86)(12, 2.86)(13, 2.86)(14, 2.87)(15, 2.86)(16, 2.86)(17, 2.86)(18, 2.87)(19, 2.86)(20, 2.86)(21, 2.87)(22, 2.86)(23, 2.87)(24, 2.87)(25, 2.87)(26, 2.86)(27, 2.86)(28, 2.87)(29, 2.87)(30, 2.86)(31, 2.86)(32, 2.86)(33, 2.86)(34, 2.86)(35, 2.86)(36, 2.86)(37, 2.87)(38, 2.86)(39, 2.86)(40, 2.88)(41, 2.86)(42, 2.86)(43, 2.85)(44, 2.87)(45, 2.87)(46, 2.86)(47, 2.86)(48, 2.86)(49, 2.86)(50, 2.86)(51, 2.86)(52, 2.86)(53, 2.87)(54, 2.85)(55, 2.85)(56, 2.86)(57, 2.87)(58, 2.87)(59, 2.87)(60, 2.86)(61, 2.85)(62, 2.87)(63, 2.87)(64, 2.88)(65, 3.27)(66, 3.27)(67, 3.26)(68, 3.26)(69, 3.26)(70, 3.26)(71, 3.26)(72, 3.26)(73, 3.27)(74, 3.27)(75, 3.26)(76, 3.26)(77, 3.26)(78, 3.26)(79, 3.25)(80, 3.26)(81, 3.27)(82, 3.26)(83, 3.27)(84, 3.27)(85, 3.26)(86, 3.27)(87, 3.26)(88, 3.27)(89, 3.26)(90, 3.25)(91, 3.27)(92, 3.27)(93, 3.27)(94, 3.26)(95, 3.27)(96, 3.27)(97, 3.25)(98, 3.25)(99, 3.25)(100, 3.24)(101, 3.25)(102, 3.24)(103, 3.24)(104, 3.24)(105, 3.24)(106, 3.25)(107, 3.25)(108, 3.24)(109, 3.24)(110, 3.25)(111, 3.25)(112, 3.24)(113, 3.26)(114, 3.26)(115, 3.25)(116, 3.25)(117, 3.25)(118, 3.25)(119, 3.25)(120, 3.26)(121, 3.24)(122, 3.24)(123, 3.25)(124, 3.24)(125, 3.23)(126, 3.24)(127, 3.23)(128, 3.23)(129, 2.86)(130, 2.86)(131, 2.86)(132, 2.86)(133, 2.86)(134, 2.86)(135, 2.85)(136, 2.86)(137, 2.85)(138, 2.84)(139, 2.85)(140, 2.85)(141, 2.86)(142, 2.85)(143, 2.85)(144, 2.85)(145, 2.83)(146, 2.84)(147, 2.84)(148, 2.82)(149, 2.83)(150, 2.82)(151, 2.83)(152, 2.83)(153, 2.83)(154, 2.84)(155, 2.83)(156, 2.83)(157, 2.83)(158, 2.82)(159, 2.83)(160, 2.83)(161, 2.83)(162, 2.83)(163, 2.83)(164, 2.84)(165, 2.83)(166, 2.83)(167, 2.83)(168, 2.84)(169, 2.83)(170, 2.84)(171, 2.83)(172, 2.84)(173, 2.84)(174, 2.84)(175, 2.83)(176, 2.83)(177, 2.84)(178, 2.83)(179, 2.83)(180, 2.83)(181, 2.83)(182, 2.84)(183, 2.83)(184, 2.84)(185, 2.83)(186, 2.84)(187, 2.84)(188, 2.84)(189, 2.84)(190, 2.83)(191, 2.83)(192, 2.83)(193, 3.06)(194, 3.06)(195, 3.05)(196, 3.06)(197, 3.05)(198, 3.06)(199, 3.06)(200, 3.06)(201, 3.06)(202, 3.06)(203, 3.07)(204, 3.07)(205, 3.06)(206, 3.06)(207, 3.06)(208, 3.07)(209, 3.06)(210, 3.06)(211, 3.07)(212, 3.06)(213, 3.07)(214, 3.06)(215, 3.07)(216, 3.06)(217, 3.06)(218, 3.06)(219, 3.07)(220, 3.07)(221, 3.06)(222, 3.06)(223, 3.07)(224, 3.07)(225, 3.06)(226, 3.06)(227, 3.06)(228, 3.06)(229, 3.06)(230, 3.06)(231, 3.07)(232, 3.05)(233, 3.07)(234, 3.06)(235, 3.06)(236, 3.07)(237, 3.06)(238, 3.06)(239, 3.06)(240, 3.07)(241, 3.07)(242, 3.06)(243, 3.07)(244, 3.06)(245, 3.06)(246, 3.07)(247, 3.06)(248, 3.05)(249, 3.06)(250, 3.05)(251, 3.07)(252, 3.05)(253, 3.06)(254, 3.06)(255, 3.05)(256, 3.06)(257, -3.76)(258, -3.77)(259, -3.77)(260, -3.77)(261, -3.77)(262, -3.77)(263, -3.77)(264, -3.77)(265, -3.77)(266, -3.77)(267, -3.77)(268, -3.77)(269, -3.77)(270, -3.76)(271, -3.76)(272, -3.77)(273, -3.78)(274, -3.77)(275, -3.78)(276, -3.78)(277, -3.77)(278, -3.78)(279, -3.77)(280, -3.78)(281, -3.77)(282, -3.78)(283, -3.78)(284, -3.78)(285, -3.78)(286, -3.78)(287, -3.78)(288, -3.78)(289, -3.88)(290, -3.87)(291, -3.87)(292, -3.87)(293, -3.87)(294, -3.86)(295, -3.87)(296, -3.87)(297, -3.87)(298, -3.88)(299, -3.86)(300, -3.87)(301, -3.87)(302, -3.87)(303, -3.87)(304, -3.87)(305, -3.88)(306, -3.88)(307, -3.88)(308, -3.88)(309, -3.88)(310, -3.88)(311, -3.88)(312, -3.88)(313, -3.87)(314, -3.88)(315, -3.88)(316, -3.88)(317, -3.88)(318, -3.88)(319, -3.89)(320, -3.88)(321, -4.73)(322, -4.72)(323, -4.73)(324, -4.73)(325, -4.72)(326, -4.72)(327, -4.73)(328, -4.72)(329, -4.73)(330, -4.73)(331, -4.73)(332, -4.73)(333, -4.72)(334, -4.73)(335, -4.73)(336, -4.73)(337, -4.74)(338, -4.74)(339, -4.74)(340, -4.73)(341, -4.73)(342, -4.73)(343, -4.73)(344, -4.73)(345, -4.74)(346, -4.73)(347, -4.74)(348, -4.74)(349, -4.74)(350, -4.74)(351, -4.74)(352, -4.73)(353, -4.73)(354, -4.74)(355, -4.73)(356, -4.72)(357, -4.74)(358, -4.73)(359, -4.74)(360, -4.74)(361, -4.73)(362, -4.73)(363, -4.73)(364, -4.74)(365, -4.73)(366, -4.74)(367, -4.74)(368, -4.74)(369, -4.73)(370, -4.73)(371, -4.73)(372, -4.73)(373, -4.72)(374, -4.72)(375, -4.74)(376, -4.74)(377, -4.73)(378, -4.72)(379, -4.73)(380, -4.73)(381, -4.74)(382, -4.73)(383, -4.72)(384, -4.72)(385, -1.80)(386, -1.80)(387, -1.81)(388, -1.79)(389, -1.80)(390, -1.79)(391, -1.79)(392, -1.80)(393, -1.80)(394, -1.80)(395, -1.80)(396, -1.80)(397, -1.79)(398, -1.80)(399, -1.79)(400, -1.79)(401, -1.79)(402, -1.78)(403, -1.78)(404, -1.78)(405, -1.78)(406, -1.78)(407, -1.79)(408, -1.77)(409, -1.80)(410, -1.78)(411, -1.79)(412, -1.79)(413, -1.78)(414, -1.77)(415, -1.78)(416, -1.79)(417, -1.79)(418, -1.80)(419, -1.78)(420, -1.79)(421, -1.79)(422, -1.79)(423, -1.79)(424, -1.79)(425, -1.79)(426, -1.79)(427, -1.80)(428, -1.78)(429, -1.79)(430, -1.79)(431, -1.80)(432, -1.79)(433, -1.79)(434, -1.79)(435, -1.79)(436, -1.79)(437, -1.80)(438, -1.79)(439, -1.79)(440, -1.79)(441, -1.79)(442, -1.79)(443, -1.79)(444, -1.79)(445, -1.79)(446, -1.78)(447, -1.79)(448, -1.80)(449, -1.76)(450, -1.75)(451, -1.76)(452, -1.75)(453, -1.75)(454, -1.74)(455, -1.75)(456, -1.75)(457, -1.74)(458, -1.76)(459, -1.75)(460, -1.75)(461, -1.75)(462, -1.76)(463, -1.76)(464, -1.75)(465, -1.75)(466, -1.75)(467, -1.76)(468, -1.74)(469, -1.76)(470, -1.75)(471, -1.75)(472, -1.75)(473, -1.75)(474, -1.74)(475, -1.76)(476, -1.75)(477, -1.75)(478, -1.76)(479, -1.75)(480, -1.75)(481, -1.75)(482, -1.75)(483, -1.75)(484, -1.75)(485, -1.75)(486, -1.75)(487, -1.75)(488, -1.74)(489, -1.76)(490, -1.74)(491, -1.75)(492, -1.75)(493, -1.75)(494, -1.75)(495, -1.75)(496, -1.75)(497, -1.75)(498, -1.76)(499, -1.74)(500, -1.76)(501, -1.75)(502, -1.75)(503, -1.75)(504, -1.75)(505, -1.76)(506, -1.74)(507, -1.76)(508, -1.75)(509, -1.75)(510, -1.75)(511, -1.75)(512, -1.76)}

\def\reccba{(1, 2.80)(2, 2.78)(3, 2.89)(4, 2.99)(5, 2.94)(6, 2.84)(7, 2.82)(8, 2.93)(9, 2.99)(10, 2.91)(11, 2.92)(12, 2.98)(13, 2.86)(14, 3.06)(15, 2.94)(16, 2.87)(17, 2.91)(18, 3.04)(19, 2.93)(20, 2.91)(21, 3.04)(22, 2.97)(23, 3.04)(24, 3.06)(25, 3.01)(26, 2.86)(27, 2.84)(28, 2.97)(29, 2.94)(30, 2.81)(31, 2.82)(32, 2.86)(33, 2.83)(34, 2.87)(35, 2.80)(36, 2.81)(37, 2.88)(38, 2.80)(39, 2.77)(40, 3.02)(41, 2.81)(42, 2.74)(43, 2.68)(44, 2.83)(45, 2.87)(46, 2.81)(47, 2.70)(48, 2.72)(49, 2.85)(50, 2.77)(51, 2.79)(52, 2.75)(53, 2.89)(54, 2.70)(55, 2.64)(56, 2.74)(57, 2.92)(58, 2.85)(59, 2.92)(60, 2.78)(61, 2.63)(62, 2.86)(63, 2.86)(64, 3.00)(65, 3.34)(66, 3.27)(67, 3.20)(68, 3.21)(69, 3.25)(70, 3.23)(71, 3.20)(72, 3.27)(73, 3.31)(74, 3.34)(75, 3.22)(76, 3.26)(77, 3.29)(78, 3.29)(79, 3.06)(80, 3.24)(81, 3.32)(82, 3.29)(83, 3.35)(84, 3.32)(85, 3.25)(86, 3.45)(87, 3.23)(88, 3.35)(89, 3.26)(90, 3.16)(91, 3.31)(92, 3.34)(93, 3.41)(94, 3.30)(95, 3.34)(96, 3.35)(97, 3.25)(98, 3.28)(99, 3.25)(100, 3.22)(101, 3.26)(102, 3.23)(103, 3.22)(104, 3.23)(105, 3.23)(106, 3.31)(107, 3.32)(108, 3.20)(109, 3.16)(110, 3.25)(111, 3.36)(112, 3.22)(113, 3.35)(114, 3.38)(115, 3.25)(116, 3.30)(117, 3.23)(118, 3.26)(119, 3.29)(120, 3.34)(121, 3.27)(122, 3.13)(123, 3.29)(124, 3.22)(125, 3.11)(126, 3.23)(127, 3.13)(128, 3.13)(129, 2.95)(130, 2.87)(131, 2.94)(132, 2.88)(133, 2.89)(134, 2.90)(135, 2.76)(136, 2.86)(137, 2.83)(138, 2.67)(139, 2.78)(140, 2.75)(141, 2.92)(142, 2.74)(143, 2.76)(144, 2.83)(145, 2.77)(146, 2.84)(147, 2.87)(148, 2.64)(149, 2.79)(150, 2.67)(151, 2.82)(152, 2.81)(153, 2.82)(154, 2.86)(155, 2.77)(156, 2.75)(157, 2.82)(158, 2.64)(159, 2.82)(160, 2.84)(161, 2.73)(162, 2.74)(163, 2.72)(164, 2.84)(165, 2.78)(166, 2.76)(167, 2.68)(168, 2.87)(169, 2.76)(170, 2.87)(171, 2.75)(172, 2.83)(173, 2.89)(174, 2.81)(175, 2.75)(176, 2.76)(177, 2.85)(178, 2.74)(179, 2.82)(180, 2.82)(181, 2.77)(182, 2.83)(183, 2.78)(184, 2.96)(185, 2.78)(186, 2.91)(187, 2.90)(188, 2.88)(189, 2.91)(190, 2.77)(191, 2.78)(192, 2.84)(193, 3.06)(194, 3.08)(195, 2.99)(196, 3.06)(197, 2.95)(198, 3.07)(199, 3.07)(200, 3.03)(201, 3.02)(202, 3.08)(203, 3.19)(204, 3.18)(205, 3.01)(206, 2.98)(207, 3.07)(208, 3.14)(209, 3.02)(210, 2.99)(211, 3.10)(212, 3.09)(213, 3.11)(214, 3.02)(215, 3.20)(216, 3.00)(217, 3.00)(218, 2.97)(219, 3.16)(220, 3.12)(221, 3.07)(222, 3.08)(223, 3.12)(224, 3.14)(225, 3.14)(226, 3.13)(227, 3.04)(228, 3.11)(229, 3.06)(230, 3.04)(231, 3.18)(232, 3.01)(233, 3.16)(234, 3.06)(235, 3.13)(236, 3.20)(237, 3.04)(238, 3.04)(239, 3.11)(240, 3.17)(241, 3.23)(242, 3.10)(243, 3.16)(244, 3.03)(245, 3.01)(246, 3.20)(247, 2.99)(248, 2.89)(249, 2.89)(250, 2.73)(251, 2.97)(252, 2.68)(253, 2.71)(254, 2.62)(255, 2.52)(256, 2.53)(257, -3.19)(258, -3.27)(259, -3.31)(260, -3.37)(261, -3.44)(262, -3.43)(263, -3.35)(264, -3.38)(265, -3.49)(266, -3.39)(267, -3.45)(268, -3.49)(269, -3.51)(270, -3.37)(271, -3.39)(272, -3.44)(273, -3.70)(274, -3.62)(275, -3.70)(276, -3.70)(277, -3.61)(278, -3.65)(279, -3.64)(280, -3.78)(281, -3.58)(282, -3.71)(283, -3.70)(284, -3.67)(285, -3.68)(286, -3.74)(287, -3.77)(288, -3.78)(289, -4.05)(290, -3.97)(291, -3.94)(292, -3.93)(293, -3.95)(294, -3.89)(295, -4.02)(296, -3.98)(297, -3.91)(298, -4.08)(299, -3.89)(300, -3.92)(301, -3.95)(302, -3.92)(303, -4.00)(304, -3.92)(305, -4.07)(306, -4.02)(307, -4.11)(308, -4.03)(309, -4.03)(310, -4.05)(311, -4.15)(312, -4.08)(313, -3.97)(314, -4.04)(315, -4.15)(316, -4.12)(317, -4.15)(318, -4.11)(319, -4.21)(320, -4.18)(321, -4.58)(322, -4.47)(323, -4.68)(324, -4.60)(325, -4.51)(326, -4.59)(327, -4.64)(328, -4.56)(329, -4.72)(330, -4.62)(331, -4.66)(332, -4.71)(333, -4.58)(334, -4.70)(335, -4.68)(336, -4.69)(337, -4.73)(338, -4.81)(339, -4.75)(340, -4.70)(341, -4.63)(342, -4.70)(343, -4.70)(344, -4.67)(345, -4.83)(346, -4.63)(347, -4.73)(348, -4.74)(349, -4.73)(350, -4.74)(351, -4.71)(352, -4.69)(353, -4.72)(354, -4.83)(355, -4.68)(356, -4.59)(357, -4.78)(358, -4.67)(359, -4.84)(360, -4.78)(361, -4.74)(362, -4.71)(363, -4.69)(364, -4.80)(365, -4.68)(366, -4.82)(367, -4.73)(368, -4.77)(369, -4.62)(370, -4.51)(371, -4.49)(372, -4.51)(373, -4.41)(374, -4.47)(375, -4.63)(376, -4.63)(377, -4.46)(378, -4.44)(379, -4.51)(380, -4.51)(381, -4.67)(382, -4.51)(383, -4.40)(384, -4.41)(385, -2.19)(386, -2.15)(387, -2.27)(388, -2.11)(389, -2.11)(390, -2.03)(391, -2.02)(392, -2.12)(393, -2.23)(394, -2.13)(395, -2.18)(396, -2.20)(397, -1.97)(398, -2.14)(399, -1.99)(400, -2.06)(401, -1.97)(402, -1.91)(403, -1.87)(404, -1.89)(405, -1.87)(406, -1.87)(407, -2.00)(408, -1.76)(409, -2.10)(410, -1.89)(411, -1.91)(412, -1.92)(413, -1.78)(414, -1.75)(415, -1.80)(416, -1.98)(417, -1.82)(418, -1.93)(419, -1.71)(420, -1.86)(421, -1.79)(422, -1.92)(423, -1.87)(424, -1.88)(425, -1.80)(426, -1.85)(427, -1.97)(428, -1.77)(429, -1.82)(430, -1.82)(431, -1.95)(432, -1.89)(433, -1.89)(434, -1.90)(435, -1.81)(436, -1.87)(437, -1.97)(438, -1.87)(439, -1.82)(440, -1.87)(441, -1.78)(442, -1.89)(443, -1.85)(444, -1.80)(445, -1.80)(446, -1.72)(447, -1.84)(448, -1.96)(449, -1.79)(450, -1.65)(451, -1.78)(452, -1.73)(453, -1.71)(454, -1.62)(455, -1.69)(456, -1.65)(457, -1.61)(458, -1.78)(459, -1.66)(460, -1.64)(461, -1.68)(462, -1.78)(463, -1.82)(464, -1.68)(465, -1.65)(466, -1.64)(467, -1.78)(468, -1.57)(469, -1.77)(470, -1.66)(471, -1.76)(472, -1.73)(473, -1.70)(474, -1.61)(475, -1.86)(476, -1.71)(477, -1.75)(478, -1.79)(479, -1.78)(480, -1.75)(481, -1.74)(482, -1.73)(483, -1.73)(484, -1.73)(485, -1.73)(486, -1.77)(487, -1.69)(488, -1.65)(489, -1.84)(490, -1.58)(491, -1.71)(492, -1.71)(493, -1.74)(494, -1.78)(495, -1.71)(496, -1.68)(497, -1.75)(498, -1.85)(499, -1.59)(500, -1.82)(501, -1.74)(502, -1.68)(503, -1.70)(504, -1.72)(505, -1.80)(506, -1.65)(507, -1.82)(508, -1.75)(509, -1.67)(510, -1.70)(511, -1.64)(512, -1.82)}

\def\reccca{(1, 1.78)(2, 1.75)(3, 2.12)(4, 2.61)(5, 2.83)(6, 2.55)(7, 2.41)(8, 3.01)(9, 3.11)(10, 2.89)(11, 3.09)(12, 3.32)(13, 2.99)(14, 3.51)(15, 3.34)(16, 3.04)(17, 3.18)(18, 3.65)(19, 3.22)(20, 3.20)(21, 3.70)(22, 3.47)(23, 3.58)(24, 3.78)(25, 3.45)(26, 3.19)(27, 2.97)(28, 3.39)(29, 3.32)(30, 2.87)(31, 2.93)(32, 3.06)(33, 3.05)(34, 3.20)(35, 2.87)(36, 2.93)(37, 3.21)(38, 2.92)(39, 2.70)(40, 3.71)(41, 2.93)(42, 2.76)(43, 2.35)(44, 2.95)(45, 2.97)(46, 2.99)(47, 2.44)(48, 2.60)(49, 3.14)(50, 2.84)(51, 2.71)(52, 2.42)(53, 2.87)(54, 2.18)(55, 2.11)(56, 2.31)(57, 2.98)(58, 2.68)(59, 3.03)(60, 2.61)(61, 2.06)(62, 2.96)(63, 2.86)(64, 3.49)(65, 3.69)(66, 3.40)(67, 2.92)(68, 3.13)(69, 3.28)(70, 3.23)(71, 3.10)(72, 3.27)(73, 3.57)(74, 3.48)(75, 3.14)(76, 3.27)(77, 3.25)(78, 3.23)(79, 2.46)(80, 3.07)(81, 3.20)(82, 3.17)(83, 3.19)(84, 3.08)(85, 2.84)(86, 3.66)(87, 2.84)(88, 3.31)(89, 3.00)(90, 2.58)(91, 3.05)(92, 3.08)(93, 3.47)(94, 3.03)(95, 3.33)(96, 3.18)(97, 3.00)(98, 2.98)(99, 2.95)(100, 2.94)(101, 3.05)(102, 2.90)(103, 2.99)(104, 3.01)(105, 2.93)(106, 3.27)(107, 3.26)(108, 2.91)(109, 2.79)(110, 3.20)(111, 3.65)(112, 3.20)(113, 3.53)(114, 3.60)(115, 3.29)(116, 3.35)(117, 3.23)(118, 3.27)(119, 3.25)(120, 3.46)(121, 3.12)(122, 2.70)(123, 3.28)(124, 3.08)(125, 2.73)(126, 3.35)(127, 2.83)(128, 2.87)(129, 3.23)(130, 2.98)(131, 3.31)(132, 3.05)(133, 3.08)(134, 2.92)(135, 2.51)(136, 2.87)(137, 3.03)(138, 2.38)(139, 2.68)(140, 2.64)(141, 3.39)(142, 2.80)(143, 2.99)(144, 3.27)(145, 2.99)(146, 3.24)(147, 3.29)(148, 2.41)(149, 2.88)(150, 2.49)(151, 2.90)(152, 2.89)(153, 2.95)(154, 3.00)(155, 2.50)(156, 2.60)(157, 2.77)(158, 2.34)(159, 3.00)(160, 3.07)(161, 2.68)(162, 2.70)(163, 2.63)(164, 3.09)(165, 2.79)(166, 2.70)(167, 2.38)(168, 3.17)(169, 2.67)(170, 3.09)(171, 2.71)(172, 2.96)(173, 3.31)(174, 2.98)(175, 2.79)(176, 2.73)(177, 3.00)(178, 2.49)(179, 2.89)(180, 2.86)(181, 2.69)(182, 2.91)(183, 2.68)(184, 3.47)(185, 2.80)(186, 3.42)(187, 3.26)(188, 3.29)(189, 3.36)(190, 2.85)(191, 2.79)(192, 3.11)(193, 3.29)(194, 3.25)(195, 2.80)(196, 2.97)(197, 2.41)(198, 2.80)(199, 2.91)(200, 2.74)(201, 2.74)(202, 2.93)(203, 3.26)(204, 3.39)(205, 2.72)(206, 2.68)(207, 2.96)(208, 3.29)(209, 2.79)(210, 2.72)(211, 3.11)(212, 3.10)(213, 3.15)(214, 2.87)(215, 3.69)(216, 2.78)(217, 2.91)(218, 2.70)(219, 3.41)(220, 3.31)(221, 3.16)(222, 3.20)(223, 3.35)(224, 3.47)(225, 3.27)(226, 3.16)(227, 2.92)(228, 3.10)(229, 2.98)(230, 2.93)(231, 3.35)(232, 2.77)(233, 3.33)(234, 2.94)(235, 3.22)(236, 3.41)(237, 2.88)(238, 2.70)(239, 2.89)(240, 3.22)(241, 3.31)(242, 2.81)(243, 2.94)(244, 2.50)(245, 2.39)(246, 3.04)(247, 2.28)(248, 1.88)(249, 2.10)(250, 1.41)(251, 2.03)(252, 0.86)(253, 0.97)(254, 0.30)(255, -0.01)(256, -0.36)(257, -0.24)(258, -0.80)(259, -1.23)(260, -1.76)(261, -2.10)(262, -2.32)(263, -1.92)(264, -2.35)(265, -2.82)(266, -2.64)(267, -3.02)(268, -3.09)(269, -3.23)(270, -2.82)(271, -2.87)(272, -3.03)(273, -3.36)(274, -2.96)(275, -3.22)(276, -3.30)(277, -3.07)(278, -3.11)(279, -3.08)(280, -3.60)(281, -2.87)(282, -3.40)(283, -3.24)(284, -3.25)(285, -3.33)(286, -3.55)(287, -3.73)(288, -3.73)(289, -3.96)(290, -3.64)(291, -3.57)(292, -3.52)(293, -3.58)(294, -3.36)(295, -3.73)(296, -3.66)(297, -3.47)(298, -4.16)(299, -3.34)(300, -3.51)(301, -3.67)(302, -3.52)(303, -3.74)(304, -3.56)(305, -3.99)(306, -3.68)(307, -4.07)(308, -3.90)(309, -3.95)(310, -3.98)(311, -4.26)(312, -4.01)(313, -3.70)(314, -4.00)(315, -4.40)(316, -4.25)(317, -4.52)(318, -4.38)(319, -4.78)(320, -4.79)(321, -4.60)(322, -4.25)(323, -5.06)(324, -4.86)(325, -4.51)(326, -4.73)(327, -4.90)(328, -4.56)(329, -5.03)(330, -4.58)(331, -4.80)(332, -4.97)(333, -4.58)(334, -4.99)(335, -4.94)(336, -5.12)(337, -5.16)(338, -5.46)(339, -5.28)(340, -5.15)(341, -4.85)(342, -5.01)(343, -5.12)(344, -4.95)(345, -5.60)(346, -4.75)(347, -5.21)(348, -5.31)(349, -5.27)(350, -5.24)(351, -5.14)(352, -5.16)(353, -5.14)(354, -5.41)(355, -4.84)(356, -4.66)(357, -5.28)(358, -4.95)(359, -5.61)(360, -5.44)(361, -5.27)(362, -5.11)(363, -5.01)(364, -5.49)(365, -5.04)(366, -5.62)(367, -5.09)(368, -5.22)(369, -5.11)(370, -4.57)(371, -4.46)(372, -4.58)(373, -3.99)(374, -4.14)(375, -4.63)(376, -4.75)(377, -3.95)(378, -3.79)(379, -4.08)(380, -3.96)(381, -4.41)(382, -3.74)(383, -3.28)(384, -3.16)(385, -2.90)(386, -2.68)(387, -3.10)(388, -2.49)(389, -2.63)(390, -2.30)(391, -2.17)(392, -2.61)(393, -2.85)(394, -2.39)(395, -2.59)(396, -2.70)(397, -1.99)(398, -2.58)(399, -2.02)(400, -2.16)(401, -2.20)(402, -1.99)(403, -1.73)(404, -1.83)(405, -1.79)(406, -1.84)(407, -2.26)(408, -1.33)(409, -2.53)(410, -1.77)(411, -1.83)(412, -1.90)(413, -1.41)(414, -1.39)(415, -1.54)(416, -2.22)(417, -1.68)(418, -2.05)(419, -1.18)(420, -1.77)(421, -1.38)(422, -1.81)(423, -1.62)(424, -1.81)(425, -1.50)(426, -1.67)(427, -2.16)(428, -1.33)(429, -1.56)(430, -1.55)(431, -2.03)(432, -1.87)(433, -1.80)(434, -1.83)(435, -1.70)(436, -1.94)(437, -2.35)(438, -1.96)(439, -1.86)(440, -1.99)(441, -1.62)(442, -1.96)(443, -1.91)(444, -1.77)(445, -1.71)(446, -1.39)(447, -1.86)(448, -2.24)(449, -1.81)(450, -1.20)(451, -1.67)(452, -1.44)(453, -1.50)(454, -1.19)(455, -1.33)(456, -1.21)(457, -0.97)(458, -1.57)(459, -1.25)(460, -1.33)(461, -1.38)(462, -1.78)(463, -1.90)(464, -1.42)(465, -1.23)(466, -1.21)(467, -1.80)(468, -1.05)(469, -1.72)(470, -1.28)(471, -1.79)(472, -1.75)(473, -1.65)(474, -1.16)(475, -2.17)(476, -1.65)(477, -1.79)(478, -1.98)(479, -2.01)(480, -1.90)(481, -1.92)(482, -1.94)(483, -2.03)(484, -2.15)(485, -2.18)(486, -2.31)(487, -2.09)(488, -2.04)(489, -2.73)(490, -1.71)(491, -2.18)(492, -2.19)(493, -2.22)(494, -2.41)(495, -2.20)(496, -2.02)(497, -2.30)(498, -2.74)(499, -1.74)(500, -2.47)(501, -2.23)(502, -1.90)(503, -1.88)(504, -1.89)(505, -2.03)(506, -1.32)(507, -1.87)(508, -1.49)(509, -0.97)(510, -0.76)(511, -0.35)(512, -0.84)}

\newcommand{\graphT}[1]{%
    \begin{tikzpicture}[xscale=4.4/512, yscale=1/6]
        \draw[ultra thin, black!20] (1,  0)--(512, 0);
        \draw[ultra thin, black!20] (1, -6)--(1, 8);
	\foreach \y in {-5, 0, 5}
	    \draw[ultra thin, black] (-2, \y) -- (4, \y);
        \draw[very thin, green!60!black] plot coordinates {\sol};
        \draw[very thin, blue]   plot coordinates {#1};
    \end{tikzpicture}%
}
\newcommand{\graphTl}[1]{%
    \begin{tikzpicture}[xscale=4.4/512, yscale=1/6]
        \draw[ultra thin, black!20] (1,  0)--(512, 0);
        \draw[ultra thin, black!20] (1, -6)--(1, 8);
	\foreach \y in {-5, 0, 5} {%
	  \node[left, inner sep=0pt] at (-5, \y) {{\tiny \y}};
	  \draw[ultra thin, black] (-2, \y) -- (4, \y);
	}
        \draw[very thin, green!60!black] plot coordinates {\sol};
        \draw[very thin, blue]   plot coordinates {#1};
    \end{tikzpicture}%
}

  \begin{figure}[tbp]
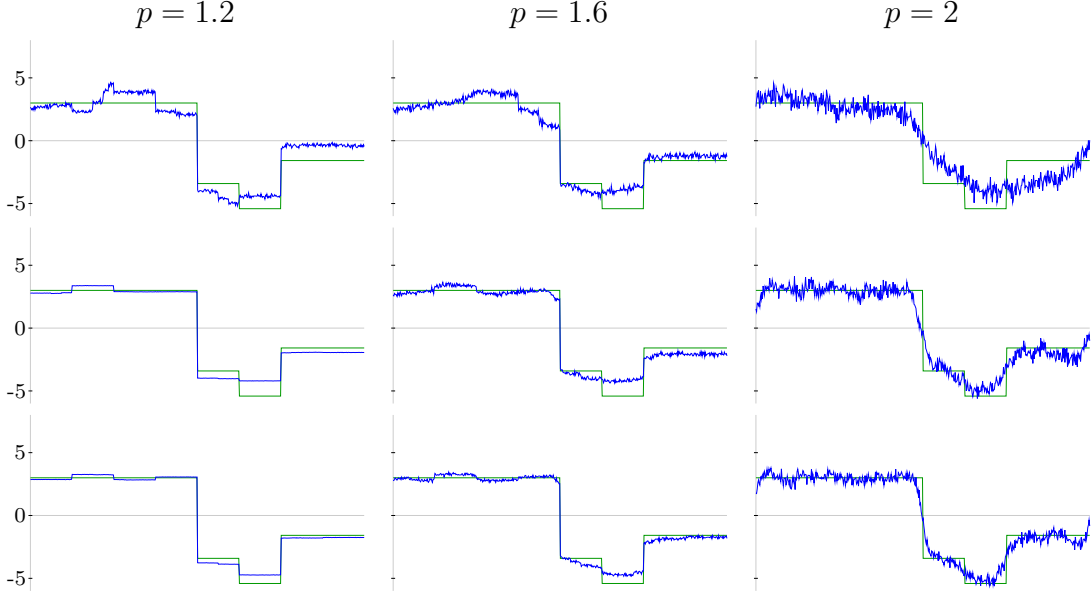

    \begin {center}
    \begin{tabular}{*{3}{c}}
	 $p=1.2$ & $p=1.6$ & $p=2$\\[1mm]
	 \graphTl{\recaaa} & \graphT{\recaba} & \graphT{\recaca}\\
	 \graphTl{\recbaa} & \graphT{\recbba} & \graphT{\recbca}\\
	 \graphTl{\reccaa} & \graphT{\reccba} & \graphT{\reccca}
    \end{tabular}
    \caption{Reconstructions obtained by the dual TIGRA method with $\|x_0\| = \start$ and $5\%$ (top), $1\%$ (middle), $0.5\%$ (bottom) noise.} \label{fig:dtigra}
    \end{center}
  \end{figure}

  \begin{table}[b!]
    \begin {center}
    \begin{tabular}{rr*{2}{|crrc}}
      \hline \noalign{\smallskip}
           &           &          \multicolumn{4}{c}{p = 1.2}                 &          \multicolumn{4}{c}{p = 1.6}                \\
  $\delta$ & $\|x_0\|$ &       $\alpha_{j^*}$ & $j^*$ & $k^*$ &     $e_{k^*}$ &       $\alpha_{j^*}$ & $j^*$ & $k^*$ &     $e_{k^*}$\\
  \hline \noalign{\smallskip}
       5\% &         1 &  $7.5 \cdot 10^{-2}$ &    47 &   290 &          0.58 &  $1.8 \cdot 10^{-2}$ &    51 &   349 &          0.33\\
       5\% &       500 &  $1.8 \cdot 10^{-2}$ &    51 &   199 &           0.6 &  $2.6 \cdot 10^{-2}$ &    50 &   214 &          0.42\\
       5\% &      1000 &  $1.5 \cdot 10^{-3}$ &    58 &  3906 &          0.43 &  $5.2 \cdot 10^{-2}$ &    48 &   311 &          0.52\\
       5\% &     10000 &  $4.3 \cdot 10^{-3}$ &    55 &   825 &          0.45 &  $1.1 \cdot 10^{-1}$ &    46 &   543 &           0.5\\[2mm]
 
       1\% &         1 &  $8.8 \cdot 10^{-3}$ &    53 &  1697 &           0.2 &  $4.3 \cdot 10^{-3}$ &    55 &  1503 &          0.22\\
       1\% &       500 &  $6.2 \cdot 10^{-3}$ &    54 &  2135 &          0.21 &  $4.3 \cdot 10^{-3}$ &    55 &  1618 &          0.23\\
       1\% &      1000 &  $8.8 \cdot 10^{-3}$ &    53 &   980 &          0.21 &  $6.2 \cdot 10^{-3}$ &    54 &  1571 &          0.23\\
       1\% &     10000 &  $2.1 \cdot 10^{-3}$ &    57 &  2841 &          0.22 &  $4.3 \cdot 10^{-3}$ &    55 &  2114 &          0.24\\[2mm]
 
     0.5\% &         1 &  $1.5 \cdot 10^{-3}$ &    58 &  5682 &          0.13 &  $1.5 \cdot 10^{-3}$ &    58 &  5006 &          0.15\\
     0.5\% &       500 &  $8.8 \cdot 10^{-3}$ &    53 &  2268 &         0.097 &  $1.5 \cdot 10^{-3}$ &    58 &  4577 &          0.18\\
     0.5\% &      1000 &  $7.3 \cdot 10^{-4}$ &    60 & 10694 &          0.14 &  $1.5 \cdot 10^{-3}$ &    58 &  6281 &          0.17\\
     0.5\% &     10000 &  $2.5 \cdot 10^{-4}$ &    63 & 11016 &         0.079 &  $1.5 \cdot 10^{-3}$ &    58 &  5155 &          0.17\\
  \hline
    \end{tabular}\\[5mm]
    \caption{Results of the dual TIGRA method with $p=1.2, 1.6$.} \label{tab:dtigra}
    \end{center}
  \end{table}

  Thus, the theoretically justified choices from Assumption \ref{ass:tigra} are not directly applicable to the problem at hand. Nevertheless, the dual TIGRA  method performed well in the numerical experiments and the results were obtained with the following parameters: The initial regularization parameter $\alpha_{0} = 10^{6}$ was updated in the outer iteration with a factor $\qa = 0.7$.
  Recall that $\alpha_0$ should be chosen large enough to ensure that the initial guess $x_0$ belongs to the region of convergence and $\qa$ should be sufficiently close to $1$ so that $\xjks$ can be used as initial guess for $\alpha_{j+1}$. In the inner iterations we used as stepsize selection
    \begin{equation} \label{eq:numbeta}
      \beta_{j, k} = \min \Big( \frac{1}{\| \nabla\Phi_{\alpha_j}(\xjk) \|}, 0.02 \Big)
    \end{equation}
  and for the stopping rule $C_{\alpha_j} = 1.5 \cdot \alpha_j$. In addition we also stopped each inner loop after at most 3000 iterations. Finally, for the discprepancy principle we chose $\tau = 2$.

 To verify the algorithm's global convergence behaviour, we have considered various randomly chosen initial guesses $x_0$ of different size $\| x_0 \|$. Some reconstructions for different values of $p$ and $\delta$ are shown in Figure \ref{fig:dtigra}. Further results are summarized in Table \ref{tab:dtigra}, where $j^*$ denotes the number of outer iterations and $k^*$ the total number of inner iterations combined for all values of $j$.
 Moreover, $e_{k^*}=\| \xjsks - \xd \| / \| \xd \|$ denotes the relative error corresponding to the final approximation $\xjsks$.

\bigskip

\subsection{Results for the dual modified Landweber method} \label{ssec:numdlwb}

\def\sol{(1, 3.00)(2, 3.00)(3, 3.00)(4, 3.00)(5, 3.00)(6, 3.00)(7, 3.00)(8, 3.00)(9, 3.00)(10, 3.00)(11, 3.00)(12, 3.00)(13, 3.00)(14, 3.00)(15, 3.00)(16, 3.00)(17, 3.00)(18, 3.00)(19, 3.00)(20, 3.00)(21, 3.00)(22, 3.00)(23, 3.00)(24, 3.00)(25, 3.00)(26, 3.00)(27, 3.00)(28, 3.00)(29, 3.00)(30, 3.00)(31, 3.00)(32, 3.00)(33, 3.00)(34, 3.00)(35, 3.00)(36, 3.00)(37, 3.00)(38, 3.00)(39, 3.00)(40, 3.00)(41, 3.00)(42, 3.00)(43, 3.00)(44, 3.00)(45, 3.00)(46, 3.00)(47, 3.00)(48, 3.00)(49, 3.00)(50, 3.00)(51, 3.00)(52, 3.00)(53, 3.00)(54, 3.00)(55, 3.00)(56, 3.00)(57, 3.00)(58, 3.00)(59, 3.00)(60, 3.00)(61, 3.00)(62, 3.00)(63, 3.00)(64, 3.00)(65, 3.00)(66, 3.00)(67, 3.00)(68, 3.00)(69, 3.00)(70, 3.00)(71, 3.00)(72, 3.00)(73, 3.00)(74, 3.00)(75, 3.00)(76, 3.00)(77, 3.00)(78, 3.00)(79, 3.00)(80, 3.00)(81, 3.00)(82, 3.00)(83, 3.00)(84, 3.00)(85, 3.00)(86, 3.00)(87, 3.00)(88, 3.00)(89, 3.00)(90, 3.00)(91, 3.00)(92, 3.00)(93, 3.00)(94, 3.00)(95, 3.00)(96, 3.00)(97, 3.00)(98, 3.00)(99, 3.00)(100, 3.00)(101, 3.00)(102, 3.00)(103, 3.00)(104, 3.00)(105, 3.00)(106, 3.00)(107, 3.00)(108, 3.00)(109, 3.00)(110, 3.00)(111, 3.00)(112, 3.00)(113, 3.00)(114, 3.00)(115, 3.00)(116, 3.00)(117, 3.00)(118, 3.00)(119, 3.00)(120, 3.00)(121, 3.00)(122, 3.00)(123, 3.00)(124, 3.00)(125, 3.00)(126, 3.00)(127, 3.00)(128, 3.00)(129, 3.00)(130, 3.00)(131, 3.00)(132, 3.00)(133, 3.00)(134, 3.00)(135, 3.00)(136, 3.00)(137, 3.00)(138, 3.00)(139, 3.00)(140, 3.00)(141, 3.00)(142, 3.00)(143, 3.00)(144, 3.00)(145, 3.00)(146, 3.00)(147, 3.00)(148, 3.00)(149, 3.00)(150, 3.00)(151, 3.00)(152, 3.00)(153, 3.00)(154, 3.00)(155, 3.00)(156, 3.00)(157, 3.00)(158, 3.00)(159, 3.00)(160, 3.00)(161, 3.00)(162, 3.00)(163, 3.00)(164, 3.00)(165, 3.00)(166, 3.00)(167, 3.00)(168, 3.00)(169, 3.00)(170, 3.00)(171, 3.00)(172, 3.00)(173, 3.00)(174, 3.00)(175, 3.00)(176, 3.00)(177, 3.00)(178, 3.00)(179, 3.00)(180, 3.00)(181, 3.00)(182, 3.00)(183, 3.00)(184, 3.00)(185, 3.00)(186, 3.00)(187, 3.00)(188, 3.00)(189, 3.00)(190, 3.00)(191, 3.00)(192, 3.00)(193, 3.00)(194, 3.00)(195, 3.00)(196, 3.00)(197, 3.00)(198, 3.00)(199, 3.00)(200, 3.00)(201, 3.00)(202, 3.00)(203, 3.00)(204, 3.00)(205, 3.00)(206, 3.00)(207, 3.00)(208, 3.00)(209, 3.00)(210, 3.00)(211, 3.00)(212, 3.00)(213, 3.00)(214, 3.00)(215, 3.00)(216, 3.00)(217, 3.00)(218, 3.00)(219, 3.00)(220, 3.00)(221, 3.00)(222, 3.00)(223, 3.00)(224, 3.00)(225, 3.00)(226, 3.00)(227, 3.00)(228, 3.00)(229, 3.00)(230, 3.00)(231, 3.00)(232, 3.00)(233, 3.00)(234, 3.00)(235, 3.00)(236, 3.00)(237, 3.00)(238, 3.00)(239, 3.00)(240, 3.00)(241, 3.00)(242, 3.00)(243, 3.00)(244, 3.00)(245, 3.00)(246, 3.00)(247, 3.00)(248, 3.00)(249, 3.00)(250, 3.00)(251, 3.00)(252, 3.00)(253, 3.00)(254, 3.00)(255, 3.00)(256, 3.00)(257, -3.41)(258, -3.41)(259, -3.41)(260, -3.41)(261, -3.41)(262, -3.41)(263, -3.41)(264, -3.41)(265, -3.41)(266, -3.41)(267, -3.41)(268, -3.41)(269, -3.41)(270, -3.41)(271, -3.41)(272, -3.41)(273, -3.41)(274, -3.41)(275, -3.41)(276, -3.41)(277, -3.41)(278, -3.41)(279, -3.41)(280, -3.41)(281, -3.41)(282, -3.41)(283, -3.41)(284, -3.41)(285, -3.41)(286, -3.41)(287, -3.41)(288, -3.41)(289, -3.41)(290, -3.41)(291, -3.41)(292, -3.41)(293, -3.41)(294, -3.41)(295, -3.41)(296, -3.41)(297, -3.41)(298, -3.41)(299, -3.41)(300, -3.41)(301, -3.41)(302, -3.41)(303, -3.41)(304, -3.41)(305, -3.41)(306, -3.41)(307, -3.41)(308, -3.41)(309, -3.41)(310, -3.41)(311, -3.41)(312, -3.41)(313, -3.41)(314, -3.41)(315, -3.41)(316, -3.41)(317, -3.41)(318, -3.41)(319, -3.41)(320, -3.41)(321, -5.41)(322, -5.41)(323, -5.41)(324, -5.41)(325, -5.41)(326, -5.41)(327, -5.41)(328, -5.41)(329, -5.41)(330, -5.41)(331, -5.41)(332, -5.41)(333, -5.41)(334, -5.41)(335, -5.41)(336, -5.41)(337, -5.41)(338, -5.41)(339, -5.41)(340, -5.41)(341, -5.41)(342, -5.41)(343, -5.41)(344, -5.41)(345, -5.41)(346, -5.41)(347, -5.41)(348, -5.41)(349, -5.41)(350, -5.41)(351, -5.41)(352, -5.41)(353, -5.41)(354, -5.41)(355, -5.41)(356, -5.41)(357, -5.41)(358, -5.41)(359, -5.41)(360, -5.41)(361, -5.41)(362, -5.41)(363, -5.41)(364, -5.41)(365, -5.41)(366, -5.41)(367, -5.41)(368, -5.41)(369, -5.41)(370, -5.41)(371, -5.41)(372, -5.41)(373, -5.41)(374, -5.41)(375, -5.41)(376, -5.41)(377, -5.41)(378, -5.41)(379, -5.41)(380, -5.41)(381, -5.41)(382, -5.41)(383, -5.41)(384, -5.41)(385, -1.58)(386, -1.58)(387, -1.58)(388, -1.58)(389, -1.58)(390, -1.58)(391, -1.58)(392, -1.58)(393, -1.58)(394, -1.58)(395, -1.58)(396, -1.58)(397, -1.58)(398, -1.58)(399, -1.58)(400, -1.58)(401, -1.58)(402, -1.58)(403, -1.58)(404, -1.58)(405, -1.58)(406, -1.58)(407, -1.58)(408, -1.58)(409, -1.58)(410, -1.58)(411, -1.58)(412, -1.58)(413, -1.58)(414, -1.58)(415, -1.58)(416, -1.58)(417, -1.58)(418, -1.58)(419, -1.58)(420, -1.58)(421, -1.58)(422, -1.58)(423, -1.58)(424, -1.58)(425, -1.58)(426, -1.58)(427, -1.58)(428, -1.58)(429, -1.58)(430, -1.58)(431, -1.58)(432, -1.58)(433, -1.58)(434, -1.58)(435, -1.58)(436, -1.58)(437, -1.58)(438, -1.58)(439, -1.58)(440, -1.58)(441, -1.58)(442, -1.58)(443, -1.58)(444, -1.58)(445, -1.58)(446, -1.58)(447, -1.58)(448, -1.58)(449, -1.58)(450, -1.58)(451, -1.58)(452, -1.58)(453, -1.58)(454, -1.58)(455, -1.58)(456, -1.58)(457, -1.58)(458, -1.58)(459, -1.58)(460, -1.58)(461, -1.58)(462, -1.58)(463, -1.58)(464, -1.58)(465, -1.58)(466, -1.58)(467, -1.58)(468, -1.58)(469, -1.58)(470, -1.58)(471, -1.58)(472, -1.58)(473, -1.58)(474, -1.58)(475, -1.58)(476, -1.58)(477, -1.58)(478, -1.58)(479, -1.58)(480, -1.58)(481, -1.58)(482, -1.58)(483, -1.58)(484, -1.58)(485, -1.58)(486, -1.58)(487, -1.58)(488, -1.58)(489, -1.58)(490, -1.58)(491, -1.58)(492, -1.58)(493, -1.58)(494, -1.58)(495, -1.58)(496, -1.58)(497, -1.58)(498, -1.58)(499, -1.58)(500, -1.58)(501, -1.58)(502, -1.58)(503, -1.58)(504, -1.58)(505, -1.58)(506, -1.58)(507, -1.58)(508, -1.58)(509, -1.58)(510, -1.58)(511, -1.58)(512, -1.58)}

\def\recaaa{(1, 3.17)(2, 2.10)(3, 1.07)(4, 1.95)(5, 0.64)(6, 2.82)(7, 3.88)(8, 0.91)(9, 3.14)(10, 2.22)(11, 2.13)(12, 3.85)(13, 3.58)(14, 2.52)(15, 2.17)(16, 3.81)(17, 3.70)(18, 2.59)(19, 2.67)(20, 1.90)(21, 2.95)(22, 2.93)(23, 2.04)(24, 4.42)(25, 4.37)(26, 1.97)(27, 3.93)(28, 0.78)(29, 1.34)(30, 3.91)(31, 5.06)(32, 1.83)(33, 2.41)(34, 2.32)(35, 2.71)(36, 3.92)(37, 3.23)(38, 4.38)(39, 3.48)(40, 3.35)(41, 4.28)(42, 2.82)(43, 1.28)(44, 1.25)(45, 2.23)(46, 3.04)(47, 2.76)(48, 3.27)(49, 2.20)(50, 3.45)(51, 2.48)(52, 2.64)(53, 2.27)(54, 1.67)(55, 5.03)(56, 2.00)(57, 4.30)(58, 3.53)(59, 3.06)(60, 4.40)(61, 3.16)(62, 1.20)(63, 3.64)(64, 4.16)(65, 2.35)(66, 1.51)(67, 3.63)(68, 2.47)(69, 2.73)(70, 1.93)(71, 0.82)(72, 0.43)(73, 1.88)(74, 2.82)(75, 1.31)(76, 1.13)(77, 1.76)(78, 1.77)(79, 2.15)(80, 2.92)(81, 1.56)(82, 2.66)(83, 1.81)(84, 0.82)(85, 0.86)(86, 2.27)(87, 0.93)(88, 0.88)(89, 0.75)(90, 2.04)(91, 1.48)(92, 2.99)(93, 1.15)(94, 1.80)(95, 2.32)(96, 3.61)(97, 2.26)(98, 0.63)(99, 1.54)(100, 0.62)(101, -0.24)(102, 1.93)(103, 1.29)(104, 0.42)(105, 1.65)(106, 2.06)(107, 1.44)(108, 3.67)(109, -0.67)(110, 0.04)(111, 0.90)(112, 3.09)(113, 2.91)(114, 2.18)(115, 3.52)(116, 5.18)(117, 3.81)(118, 4.38)(119, 3.84)(120, 3.57)(121, 8.90)(122, 7.30)(123, 9.48)(124, 10.82)(125, 7.90)(126, 10.04)(127, 9.64)(128, 10.48)(129, 2.53)(130, 4.42)(131, 3.72)(132, 4.32)(133, 3.24)(134, 4.48)(135, 3.81)(136, 2.52)(137, 2.23)(138, 3.55)(139, 4.88)(140, 3.83)(141, 4.03)(142, 4.21)(143, 3.42)(144, 4.31)(145, 4.70)(146, 4.12)(147, 2.93)(148, 3.53)(149, 3.80)(150, 1.99)(151, 4.45)(152, 5.38)(153, 3.48)(154, 3.36)(155, 4.71)(156, 2.64)(157, 4.20)(158, 2.81)(159, 3.97)(160, 3.47)(161, 3.21)(162, 3.31)(163, 4.11)(164, 4.53)(165, 3.63)(166, 1.51)(167, 3.93)(168, 1.82)(169, 3.85)(170, 2.49)(171, 4.04)(172, 2.26)(173, 1.10)(174, 4.87)(175, 5.16)(176, 3.60)(177, 5.71)(178, 5.08)(179, 2.71)(180, 3.73)(181, 5.66)(182, 3.18)(183, 3.77)(184, 4.01)(185, 3.19)(186, 3.13)(187, 5.40)(188, 2.30)(189, 5.11)(190, 3.28)(191, 5.09)(192, 4.38)(193, 2.59)(194, 3.48)(195, 3.40)(196, 4.65)(197, 2.02)(198, 4.28)(199, 3.39)(200, 3.28)(201, 3.19)(202, 4.26)(203, 2.61)(204, 3.25)(205, 3.48)(206, 3.99)(207, 4.22)(208, 2.65)(209, 2.75)(210, 3.55)(211, 2.32)(212, 2.48)(213, 2.70)(214, 1.58)(215, 3.16)(216, 1.96)(217, 3.94)(218, 3.22)(219, 2.74)(220, 3.68)(221, 2.33)(222, 1.71)(223, 4.52)(224, 1.45)(225, 0.48)(226, 1.60)(227, 3.79)(228, 2.18)(229, -0.68)(230, 3.46)(231, 1.88)(232, 1.11)(233, 2.22)(234, 1.67)(235, 1.35)(236, 1.29)(237, 1.05)(238, 0.07)(239, 1.17)(240, 1.73)(241, 2.24)(242, 1.24)(243, 2.72)(244, 0.24)(245, 2.38)(246, 1.04)(247, 1.27)(248, 1.41)(249, 2.86)(250, 0.46)(251, 1.31)(252, 0.93)(253, 1.79)(254, 3.01)(255, 2.10)(256, 0.32)(257, -2.03)(258, -3.62)(259, -3.61)(260, -2.66)(261, -2.54)(262, -2.51)(263, -3.00)(264, -3.11)(265, -3.77)(266, -2.97)(267, -2.11)(268, -1.27)(269, -1.44)(270, -3.05)(271, -2.25)(272, -1.97)(273, -1.53)(274, -2.74)(275, -3.26)(276, -4.90)(277, -4.37)(278, -2.18)(279, -1.39)(280, -3.30)(281, -2.53)(282, -2.58)(283, -2.87)(284, -4.39)(285, -2.38)(286, -3.76)(287, -3.37)(288, -4.43)(289, -6.76)(290, -4.66)(291, -5.49)(292, -4.28)(293, -4.92)(294, -3.96)(295, -3.86)(296, -5.42)(297, -5.95)(298, -4.36)(299, -5.24)(300, -5.02)(301, -5.02)(302, -5.71)(303, -5.23)(304, -5.33)(305, -9.12)(306, -8.08)(307, -8.97)(308, -8.15)(309, -7.20)(310, -6.93)(311, -8.34)(312, -8.09)(313, -7.93)(314, -7.41)(315, -10.05)(316, -7.60)(317, -7.74)(318, -10.56)(319, -8.58)(320, -7.52)(321, -5.52)(322, -4.24)(323, -6.15)(324, -6.08)(325, -5.15)(326, -5.60)(327, -6.96)(328, -5.78)(329, -4.52)(330, -4.73)(331, -5.30)(332, -4.12)(333, -6.54)(334, -6.49)(335, -2.61)(336, -5.20)(337, -6.66)(338, -5.31)(339, -7.04)(340, -5.20)(341, -5.24)(342, -6.00)(343, -5.00)(344, -5.05)(345, -5.20)(346, -5.14)(347, -5.41)(348, -5.29)(349, -3.86)(350, -5.23)(351, -4.75)(352, -5.26)(353, -5.67)(354, -3.18)(355, -5.54)(356, -6.27)(357, -5.42)(358, -5.73)(359, -5.82)(360, -7.89)(361, -4.55)(362, -5.12)(363, -3.44)(364, -6.48)(365, -5.38)(366, -6.59)(367, -5.88)(368, -4.95)(369, -6.47)(370, -3.14)(371, -3.80)(372, -5.47)(373, -5.38)(374, -5.21)(375, -4.01)(376, -5.08)(377, -4.67)(378, -5.48)(379, -6.05)(380, -5.96)(381, -5.57)(382, -6.34)(383, -5.04)(384, -5.05)(385, -1.87)(386, -2.13)(387, -0.13)(388, -0.92)(389, -0.45)(390, 0.79)(391, 0.30)(392, 2.39)(393, 0.90)(394, 4.64)(395, 4.04)(396, 2.72)(397, 2.60)(398, 1.93)(399, 3.51)(400, 2.36)(401, -0.37)(402, 0.75)(403, 1.67)(404, 2.05)(405, 2.89)(406, 2.37)(407, 1.37)(408, 1.31)(409, 0.23)(410, 0.34)(411, 0.81)(412, -0.43)(413, 2.01)(414, 1.69)(415, 1.70)(416, 2.03)(417, 2.15)(418, 2.88)(419, 1.09)(420, 1.37)(421, 2.42)(422, 1.93)(423, 0.64)(424, 1.94)(425, 3.09)(426, -0.11)(427, 1.50)(428, 1.95)(429, 1.02)(430, 1.11)(431, 1.15)(432, 1.21)(433, 2.36)(434, 0.75)(435, 2.66)(436, 0.89)(437, -0.23)(438, 1.17)(439, 1.34)(440, 3.00)(441, 1.52)(442, 0.69)(443, 0.46)(444, 1.28)(445, 0.41)(446, 2.24)(447, 0.86)(448, -0.60)(449, 1.30)(450, 0.64)(451, 1.59)(452, 1.48)(453, -0.39)(454, 0.49)(455, 1.33)(456, 1.55)(457, 2.16)(458, -0.01)(459, 2.19)(460, 0.38)(461, 2.05)(462, 2.38)(463, 3.46)(464, 1.14)(465, 2.01)(466, 0.67)(467, 2.57)(468, 1.56)(469, -0.56)(470, 2.06)(471, 0.69)(472, 0.10)(473, 1.18)(474, 1.48)(475, 1.44)(476, 1.85)(477, 2.23)(478, 1.36)(479, 0.10)(480, 1.08)(481, 2.04)(482, 1.93)(483, 0.36)(484, 0.52)(485, 0.32)(486, 0.51)(487, 2.94)(488, 1.36)(489, 1.55)(490, 0.17)(491, -0.28)(492, 1.93)(493, 1.63)(494, -0.92)(495, 2.15)(496, 1.21)(497, 2.01)(498, 2.72)(499, 0.62)(500, -0.10)(501, -0.46)(502, 0.21)(503, 1.59)(504, 1.05)(505, 1.49)(506, 0.66)(507, 2.44)(508, 0.71)(509, 1.58)(510, 0.30)(511, 2.29)(512, 0.81)}

\def\recaba{(1, 3.58)(2, 2.48)(3, 1.39)(4, 2.22)(5, 1.01)(6, 3.18)(7, 4.26)(8, 1.21)(9, 3.29)(10, 2.35)(11, 2.24)(12, 3.93)(13, 3.68)(14, 2.57)(15, 2.21)(16, 3.84)(17, 3.73)(18, 2.60)(19, 2.66)(20, 1.84)(21, 2.88)(22, 2.87)(23, 1.95)(24, 4.32)(25, 4.36)(26, 1.92)(27, 3.90)(28, 0.71)(29, 1.26)(30, 3.86)(31, 4.99)(32, 1.72)(33, 2.25)(34, 2.16)(35, 2.55)(36, 3.75)(37, 3.06)(38, 4.21)(39, 3.32)(40, 3.17)(41, 4.12)(42, 2.65)(43, 1.11)(44, 1.10)(45, 2.05)(46, 2.88)(47, 2.61)(48, 3.12)(49, 2.04)(50, 3.30)(51, 2.37)(52, 2.54)(53, 2.14)(54, 1.53)(55, 4.93)(56, 1.92)(57, 4.25)(58, 3.46)(59, 2.99)(60, 4.33)(61, 3.07)(62, 1.08)(63, 3.50)(64, 4.02)(65, 2.71)(66, 1.81)(67, 3.95)(68, 2.78)(69, 2.99)(70, 2.19)(71, 1.03)(72, 0.63)(73, 2.18)(74, 3.15)(75, 1.59)(76, 1.41)(77, 2.01)(78, 2.02)(79, 2.39)(80, 3.17)(81, 1.74)(82, 2.85)(83, 1.98)(84, 0.99)(85, 1.04)(86, 2.44)(87, 1.12)(88, 1.05)(89, 0.95)(90, 2.26)(91, 1.69)(92, 3.24)(93, 1.40)(94, 2.04)(95, 2.59)(96, 3.93)(97, 3.24)(98, 1.61)(99, 2.56)(100, 1.63)(101, 0.75)(102, 2.96)(103, 2.29)(104, 1.43)(105, 2.74)(106, 3.17)(107, 2.56)(108, 4.82)(109, 0.39)(110, 1.12)(111, 1.98)(112, 4.21)(113, 2.73)(114, 2.03)(115, 3.31)(116, 4.96)(117, 3.69)(118, 4.25)(119, 3.70)(120, 3.51)(121, 5.77)(122, 4.11)(123, 6.14)(124, 7.43)(125, 4.71)(126, 6.71)(127, 6.23)(128, 7.01)(129, 2.85)(130, 4.76)(131, 4.06)(132, 4.63)(133, 3.52)(134, 4.75)(135, 4.08)(136, 2.77)(137, 2.47)(138, 3.81)(139, 5.17)(140, 4.13)(141, 4.33)(142, 4.51)(143, 3.75)(144, 4.66)(145, 4.96)(146, 4.39)(147, 3.21)(148, 3.81)(149, 4.01)(150, 2.22)(151, 4.71)(152, 5.65)(153, 3.75)(154, 3.58)(155, 4.93)(156, 2.86)(157, 4.40)(158, 3.02)(159, 4.17)(160, 3.70)(161, 3.52)(162, 3.62)(163, 4.46)(164, 4.87)(165, 3.98)(166, 1.84)(167, 4.29)(168, 2.15)(169, 4.21)(170, 2.87)(171, 4.42)(172, 2.61)(173, 1.45)(174, 5.23)(175, 5.52)(176, 3.96)(177, 5.97)(178, 5.31)(179, 2.84)(180, 3.85)(181, 5.88)(182, 3.34)(183, 3.94)(184, 4.21)(185, 3.38)(186, 3.32)(187, 5.58)(188, 2.39)(189, 5.16)(190, 3.26)(191, 5.04)(192, 4.34)(193, 2.42)(194, 3.32)(195, 3.22)(196, 4.46)(197, 1.79)(198, 4.06)(199, 3.14)(200, 3.03)(201, 2.92)(202, 4.01)(203, 2.38)(204, 3.00)(205, 3.26)(206, 3.78)(207, 3.95)(208, 2.38)(209, 2.60)(210, 3.39)(211, 2.16)(212, 2.30)(213, 2.52)(214, 1.40)(215, 2.98)(216, 1.78)(217, 3.75)(218, 3.02)(219, 2.58)(220, 3.51)(221, 2.12)(222, 1.50)(223, 4.32)(224, 1.22)(225, 0.29)(226, 1.38)(227, 3.60)(228, 1.96)(229, -0.99)(230, 3.21)(231, 1.58)(232, 0.81)(233, 1.86)(234, 1.28)(235, 0.97)(236, 0.91)(237, 0.63)(238, -0.39)(239, 0.69)(240, 1.22)(241, 1.86)(242, 0.85)(243, 2.32)(244, -0.21)(245, 1.99)(246, 0.61)(247, 0.85)(248, 0.96)(249, 2.50)(250, 0.05)(251, 0.86)(252, 0.48)(253, 1.33)(254, 2.52)(255, 1.58)(256, -0.26)(257, -2.41)(258, -4.02)(259, -4.01)(260, -3.05)(261, -2.91)(262, -2.90)(263, -3.41)(264, -3.54)(265, -4.21)(266, -3.44)(267, -2.56)(268, -1.72)(269, -1.92)(270, -3.56)(271, -2.74)(272, -2.47)(273, -2.11)(274, -3.34)(275, -3.86)(276, -5.56)(277, -5.05)(278, -2.85)(279, -2.09)(280, -4.03)(281, -3.26)(282, -3.35)(283, -3.65)(284, -5.19)(285, -3.18)(286, -4.58)(287, -4.20)(288, -5.27)(289, -6.40)(290, -4.27)(291, -5.10)(292, -3.90)(293, -4.52)(294, -3.54)(295, -3.46)(296, -5.02)(297, -5.55)(298, -3.92)(299, -4.84)(300, -4.62)(301, -4.66)(302, -5.36)(303, -4.88)(304, -4.99)(305, -6.81)(306, -5.74)(307, -6.64)(308, -5.81)(309, -4.69)(310, -4.38)(311, -5.74)(312, -5.50)(313, -5.49)(314, -4.95)(315, -7.54)(316, -4.96)(317, -5.30)(318, -8.00)(319, -5.92)(320, -4.75)(321, -5.30)(322, -3.98)(323, -5.87)(324, -5.81)(325, -4.83)(326, -5.30)(327, -6.63)(328, -5.41)(329, -4.02)(330, -4.21)(331, -4.80)(332, -3.60)(333, -6.02)(334, -5.96)(335, -2.00)(336, -4.63)(337, -6.26)(338, -4.90)(339, -6.60)(340, -4.73)(341, -4.79)(342, -5.55)(343, -4.54)(344, -4.60)(345, -4.70)(346, -4.65)(347, -4.92)(348, -4.83)(349, -3.36)(350, -4.74)(351, -4.23)(352, -4.74)(353, -5.18)(354, -2.66)(355, -5.03)(356, -5.77)(357, -4.88)(358, -5.21)(359, -5.28)(360, -7.36)(361, -4.00)(362, -4.55)(363, -2.84)(364, -5.88)(365, -4.78)(366, -5.99)(367, -5.28)(368, -4.33)(369, -5.80)(370, -2.45)(371, -3.10)(372, -4.76)(373, -4.64)(374, -4.47)(375, -3.23)(376, -4.29)(377, -3.80)(378, -4.62)(379, -5.15)(380, -5.05)(381, -4.71)(382, -5.43)(383, -4.07)(384, -4.07)(385, -2.12)(386, -2.36)(387, -0.43)(388, -1.19)(389, -1.08)(390, 0.14)(391, -0.38)(392, 1.65)(393, -0.92)(394, 2.83)(395, 2.26)(396, 0.91)(397, 0.76)(398, 0.09)(399, 1.68)(400, 0.49)(401, -1.53)(402, -0.39)(403, 0.54)(404, 0.90)(405, 1.74)(406, 1.26)(407, 0.23)(408, 0.18)(409, -0.90)(410, -0.80)(411, -0.33)(412, -1.55)(413, 0.92)(414, 0.58)(415, 0.61)(416, 0.91)(417, 1.11)(418, 1.88)(419, 0.06)(420, 0.35)(421, 1.41)(422, 0.92)(423, -0.38)(424, 0.95)(425, 2.09)(426, -1.09)(427, 0.51)(428, 0.96)(429, 0.03)(430, 0.09)(431, 0.14)(432, 0.21)(433, 1.38)(434, -0.23)(435, 1.70)(436, -0.07)(437, -1.23)(438, 0.20)(439, 0.40)(440, 2.09)(441, 0.68)(442, -0.17)(443, -0.42)(444, 0.42)(445, -0.45)(446, 1.43)(447, 0.06)(448, -1.38)(449, 0.48)(450, -0.17)(451, 0.81)(452, 0.66)(453, -1.21)(454, -0.29)(455, 0.51)(456, 0.75)(457, 1.33)(458, -0.87)(459, 1.36)(460, -0.47)(461, 1.19)(462, 1.51)(463, 2.61)(464, 0.25)(465, 1.16)(466, -0.20)(467, 1.71)(468, 0.69)(469, -1.46)(470, 1.18)(471, -0.25)(472, -0.86)(473, 0.24)(474, 0.55)(475, 0.51)(476, 0.94)(477, 1.32)(478, 0.44)(479, -0.85)(480, 0.12)(481, 0.99)(482, 0.86)(483, -0.75)(484, -0.57)(485, -0.78)(486, -0.63)(487, 1.84)(488, 0.24)(489, 0.44)(490, -0.95)(491, -1.40)(492, 0.81)(493, 0.49)(494, -2.07)(495, 1.02)(496, 0.06)(497, 0.84)(498, 1.57)(499, -0.53)(500, -1.26)(501, -1.63)(502, -0.94)(503, 0.46)(504, -0.09)(505, 0.39)(506, -0.47)(507, 1.34)(508, -0.38)(509, 0.43)(510, -0.85)(511, 1.14)(512, -0.34)}

\def\recaca{(1, 3.79)(2, 4.86)(3, 5.92)(4, 5.30)(5, 6.25)(6, 4.01)(7, 2.93)(8, 5.93)(9, 3.81)(10, 4.70)(11, 4.96)(12, 3.31)(13, 3.38)(14, 4.47)(15, 5.07)(16, 3.49)(17, 3.28)(18, 4.43)(19, 4.32)(20, 5.16)(21, 4.21)(22, 4.11)(23, 5.01)(24, 2.66)(25, 2.63)(26, 5.00)(27, 3.02)(28, 6.16)(29, 5.59)(30, 3.01)(31, 1.83)(32, 5.11)(33, 4.49)(34, 4.45)(35, 4.03)(36, 2.92)(37, 3.50)(38, 2.46)(39, 3.17)(40, 3.42)(41, 2.34)(42, 3.77)(43, 5.45)(44, 5.17)(45, 4.41)(46, 3.46)(47, 3.65)(48, 3.23)(49, 4.18)(50, 2.99)(51, 3.69)(52, 3.39)(53, 3.81)(54, 4.33)(55, 0.89)(56, 3.93)(57, 1.47)(58, 2.34)(59, 2.71)(60, 1.32)(61, 2.59)(62, 4.51)(63, 2.01)(64, 1.42)(65, 2.41)(66, 3.42)(67, 1.17)(68, 2.32)(69, 2.14)(70, 2.83)(71, 4.05)(72, 4.42)(73, 2.81)(74, 1.65)(75, 3.28)(76, 3.43)(77, 2.92)(78, 2.96)(79, 2.54)(80, 1.73)(81, 3.28)(82, 2.11)(83, 3.03)(84, 3.96)(85, 3.86)(86, 2.61)(87, 3.71)(88, 3.90)(89, 3.93)(90, 2.67)(91, 3.28)(92, 1.60)(93, 3.44)(94, 2.93)(95, 2.35)(96, 0.88)(97, 1.68)(98, 3.30)(99, 2.29)(100, 3.30)(101, 4.12)(102, 1.99)(103, 2.97)(104, 3.82)(105, 2.44)(106, 2.05)(107, 2.73)(108, 0.47)(109, 4.91)(110, 4.19)(111, 3.37)(112, 1.21)(113, 3.25)(114, 3.86)(115, 2.62)(116, 1.10)(117, 2.22)(118, 1.80)(119, 2.47)(120, 2.23)(121, 0.89)(122, 2.51)(123, 0.57)(124, -0.66)(125, 2.01)(126, 0.20)(127, 0.62)(128, -0.17)(129, 3.37)(130, 1.41)(131, 2.01)(132, 1.63)(133, 2.65)(134, 1.59)(135, 2.15)(136, 3.49)(137, 3.76)(138, 2.35)(139, 0.98)(140, 1.93)(141, 1.68)(142, 1.54)(143, 2.16)(144, 1.28)(145, 1.25)(146, 1.75)(147, 2.91)(148, 2.28)(149, 2.30)(150, 4.12)(151, 1.64)(152, 0.61)(153, 2.39)(154, 2.77)(155, 1.37)(156, 3.32)(157, 1.88)(158, 3.19)(159, 2.13)(160, 2.52)(161, 2.77)(162, 2.66)(163, 1.78)(164, 1.43)(165, 2.28)(166, 4.50)(167, 2.03)(168, 4.22)(169, 2.11)(170, 3.31)(171, 1.84)(172, 3.77)(173, 4.90)(174, 1.10)(175, 0.89)(176, 2.48)(177, 0.43)(178, 1.23)(179, 3.60)(180, 2.65)(181, 0.62)(182, 3.10)(183, 2.57)(184, 2.18)(185, 3.09)(186, 3.16)(187, 0.93)(188, 4.13)(189, 1.36)(190, 3.29)(191, 1.58)(192, 2.21)(193, 3.52)(194, 2.46)(195, 2.56)(196, 1.40)(197, 4.03)(198, 1.80)(199, 2.82)(200, 2.94)(201, 3.15)(202, 1.99)(203, 3.44)(204, 3.02)(205, 2.66)(206, 2.02)(207, 2.06)(208, 3.70)(209, 3.21)(210, 2.51)(211, 3.60)(212, 3.55)(213, 3.25)(214, 4.32)(215, 2.73)(216, 3.86)(217, 1.98)(218, 2.65)(219, 2.87)(220, 1.93)(221, 3.27)(222, 3.80)(223, 1.00)(224, 4.01)(225, 4.16)(226, 3.12)(227, 0.80)(228, 2.41)(229, 5.33)(230, 1.03)(231, 2.74)(232, 3.36)(233, 2.36)(234, 2.88)(235, 2.95)(236, 3.04)(237, 3.18)(238, 4.13)(239, 3.01)(240, 2.54)(241, 1.80)(242, 2.65)(243, 1.11)(244, 3.53)(245, 1.08)(246, 2.51)(247, 2.03)(248, 2.06)(249, 0.50)(250, 2.79)(251, 1.96)(252, 2.10)(253, 1.10)(254, -0.23)(255, 0.54)(256, 2.27)(257, 0.15)(258, 1.66)(259, 1.52)(260, 0.45)(261, 0.12)(262, -0.01)(263, 0.46)(264, 0.56)(265, 1.03)(266, 0.25)(267, -0.78)(268, -1.72)(269, -1.60)(270, -0.05)(271, -1.08)(272, -1.44)(273, -2.04)(274, -0.94)(275, -0.55)(276, 1.15)(277, 0.55)(278, -1.80)(279, -2.57)(280, -0.76)(281, -1.80)(282, -1.60)(283, -1.43)(284, -0.07)(285, -2.22)(286, -0.97)(287, -1.37)(288, -0.47)(289, 0.17)(290, -1.96)(291, -1.30)(292, -2.54)(293, -2.00)(294, -3.11)(295, -3.22)(296, -1.74)(297, -1.40)(298, -3.18)(299, -2.12)(300, -2.47)(301, -2.41)(302, -1.79)(303, -2.30)(304, -2.25)(305, -0.83)(306, -1.98)(307, -1.25)(308, -2.06)(309, -3.25)(310, -3.84)(311, -2.46)(312, -2.67)(313, -2.76)(314, -3.33)(315, -0.78)(316, -3.50)(317, -3.17)(318, -0.59)(319, -2.66)(320, -4.08)(321, -2.81)(322, -4.27)(323, -2.44)(324, -2.53)(325, -3.62)(326, -3.17)(327, -1.98)(328, -3.19)(329, -4.56)(330, -4.42)(331, -3.84)(332, -5.04)(333, -2.69)(334, -2.80)(335, -6.77)(336, -4.19)(337, -2.80)(338, -4.17)(339, -2.62)(340, -4.60)(341, -4.45)(342, -3.76)(343, -4.76)(344, -4.68)(345, -4.65)(346, -4.67)(347, -4.47)(348, -4.38)(349, -6.00)(350, -4.62)(351, -5.21)(352, -4.74)(353, -4.33)(354, -6.81)(355, -4.46)(356, -3.63)(357, -4.64)(358, -4.21)(359, -4.20)(360, -2.08)(361, -5.34)(362, -4.83)(363, -6.63)(364, -3.60)(365, -4.60)(366, -3.45)(367, -4.12)(368, -5.07)(369, -3.47)(370, -6.80)(371, -6.15)(372, -4.46)(373, -4.69)(374, -4.77)(375, -5.98)(376, -4.83)(377, -5.13)(378, -4.18)(379, -3.64)(380, -3.69)(381, -3.92)(382, -3.27)(383, -4.54)(384, -4.56)(385, -3.37)(386, -3.05)(387, -4.83)(388, -4.14)(389, -3.96)(390, -5.11)(391, -4.56)(392, -6.42)(393, -3.59)(394, -7.19)(395, -6.66)(396, -5.20)(397, -4.98)(398, -4.30)(399, -5.87)(400, -4.61)(401, -2.92)(402, -3.99)(403, -4.87)(404, -5.11)(405, -5.87)(406, -5.47)(407, -4.33)(408, -4.25)(409, -3.16)(410, -3.20)(411, -3.57)(412, -2.49)(413, -4.81)(414, -4.45)(415, -4.46)(416, -4.63)(417, -4.81)(418, -5.67)(419, -3.77)(420, -3.99)(421, -5.14)(422, -4.63)(423, -3.31)(424, -4.76)(425, -5.80)(426, -2.62)(427, -4.21)(428, -4.63)(429, -3.64)(430, -3.63)(431, -3.65)(432, -3.73)(433, -4.87)(434, -3.26)(435, -5.24)(436, -3.41)(437, -2.24)(438, -3.65)(439, -3.90)(440, -5.58)(441, -4.21)(442, -3.31)(443, -3.03)(444, -3.79)(445, -3.01)(446, -4.89)(447, -3.41)(448, -2.06)(449, -3.83)(450, -3.24)(451, -4.29)(452, -4.06)(453, -2.17)(454, -3.18)(455, -3.90)(456, -4.16)(457, -4.72)(458, -2.47)(459, -4.72)(460, -2.90)(461, -4.49)(462, -4.79)(463, -5.86)(464, -3.46)(465, -4.40)(466, -2.99)(467, -4.91)(468, -3.91)(469, -1.78)(470, -4.43)(471, -2.86)(472, -2.18)(473, -3.26)(474, -3.60)(475, -3.54)(476, -4.00)(477, -4.31)(478, -3.45)(479, -2.14)(480, -3.04)(481, -4.05)(482, -3.84)(483, -2.20)(484, -2.41)(485, -2.10)(486, -2.01)(487, -4.55)(488, -2.87)(489, -3.06)(490, -1.65)(491, -1.13)(492, -3.34)(493, -2.89)(494, -0.28)(495, -3.31)(496, -2.28)(497, -2.91)(498, -3.57)(499, -1.35)(500, -0.51)(501, 0.00)(502, -0.62)(503, -1.90)(504, -1.21)(505, -1.68)(506, -0.66)(507, -2.44)(508, -0.58)(509, -1.13)(510, 0.33)(511, -1.45)(512, 0.19)}

\def\recbaa{(1, 3.63)(2, 2.84)(3, 4.40)(4, 3.65)(5, 2.97)(6, 2.85)(7, 1.74)(8, 4.03)(9, 2.62)(10, 3.84)(11, 1.86)(12, 4.41)(13, 3.11)(14, 1.44)(15, 3.27)(16, 2.82)(17, 2.19)(18, 3.04)(19, 2.27)(20, 2.30)(21, 1.77)(22, 4.47)(23, 3.70)(24, 2.68)(25, 2.89)(26, 4.09)(27, 3.30)(28, 2.05)(29, 2.22)(30, 3.20)(31, 3.56)(32, 3.32)(33, 1.79)(34, 1.96)(35, 2.42)(36, 0.47)(37, 1.91)(38, 3.31)(39, 2.99)(40, 3.77)(41, 2.54)(42, 3.25)(43, 3.19)(44, 2.37)(45, 2.46)(46, 2.97)(47, 3.05)(48, 2.32)(49, 3.59)(50, 3.25)(51, 3.04)(52, 4.17)(53, 3.92)(54, 3.70)(55, 2.65)(56, 3.01)(57, 3.51)(58, 2.16)(59, 2.99)(60, 6.39)(61, 3.12)(62, 4.02)(63, 1.11)(64, 4.98)(65, 2.73)(66, 2.15)(67, 2.06)(68, 3.83)(69, 2.43)(70, 0.94)(71, 1.61)(72, 4.13)(73, 3.72)(74, 1.43)(75, 1.95)(76, 3.64)(77, 5.35)(78, 2.79)(79, 2.65)(80, 1.83)(81, 2.45)(82, 4.52)(83, 2.50)(84, 4.25)(85, 4.10)(86, 4.99)(87, 2.49)(88, 4.33)(89, 4.90)(90, 2.50)(91, 1.66)(92, 3.91)(93, 2.93)(94, 1.86)(95, 3.96)(96, 2.37)(97, 1.52)(98, 4.75)(99, 4.48)(100, 2.61)(101, 3.49)(102, 2.58)(103, 3.78)(104, 2.24)(105, 1.86)(106, 2.83)(107, 2.94)(108, 2.31)(109, 4.19)(110, 2.97)(111, 2.29)(112, 2.91)(113, 1.86)(114, 3.53)(115, 2.69)(116, 4.03)(117, 2.47)(118, 3.18)(119, 3.83)(120, 2.61)(121, 1.86)(122, 3.46)(123, 3.07)(124, 3.32)(125, 3.59)(126, 3.34)(127, 2.87)(128, 2.79)(129, 3.12)(130, 2.92)(131, 3.33)(132, 2.36)(133, 3.18)(134, 3.79)(135, 3.53)(136, 2.84)(137, 3.10)(138, 4.71)(139, 2.43)(140, 2.66)(141, 3.29)(142, 3.47)(143, 1.89)(144, 3.21)(145, 2.76)(146, 3.87)(147, 2.83)(148, 3.37)(149, 3.75)(150, 2.52)(151, 3.86)(152, 3.28)(153, 3.27)(154, 2.37)(155, 4.04)(156, 3.51)(157, 2.00)(158, 3.38)(159, 1.48)(160, 2.90)(161, 1.00)(162, 2.26)(163, 2.51)(164, 4.16)(165, 1.74)(166, 2.61)(167, 1.18)(168, 4.10)(169, 3.24)(170, 4.49)(171, 4.11)(172, 3.68)(173, 2.96)(174, 1.57)(175, 1.79)(176, 4.30)(177, 2.27)(178, 2.66)(179, 2.90)(180, 2.86)(181, 2.60)(182, 4.20)(183, 2.09)(184, 3.63)(185, 3.99)(186, 4.04)(187, 2.14)(188, 2.77)(189, 1.67)(190, 3.45)(191, 1.54)(192, 4.79)(193, 2.43)(194, 1.55)(195, 1.88)(196, 5.32)(197, 2.13)(198, 3.30)(199, 2.69)(200, 4.64)(201, 1.64)(202, 2.91)(203, 0.98)(204, 3.09)(205, 3.79)(206, 3.84)(207, 2.15)(208, 2.88)(209, 2.27)(210, 2.58)(211, 5.06)(212, 2.58)(213, 3.32)(214, 4.84)(215, 2.20)(216, 2.50)(217, 3.55)(218, 2.11)(219, 2.91)(220, 2.62)(221, 2.18)(222, 2.46)(223, 2.78)(224, 3.71)(225, 3.80)(226, 3.55)(227, 4.13)(228, 2.39)(229, 2.52)(230, 2.51)(231, 2.63)(232, 4.84)(233, 3.79)(234, 3.77)(235, 2.51)(236, 3.11)(237, 2.88)(238, 3.27)(239, 3.23)(240, 2.38)(241, 3.97)(242, 4.27)(243, 2.55)(244, 2.77)(245, 2.78)(246, 2.93)(247, 1.79)(248, 1.59)(249, 1.29)(250, 2.90)(251, 1.43)(252, 3.22)(253, 2.69)(254, 3.00)(255, 2.24)(256, 3.37)(257, -2.13)(258, -2.67)(259, -3.16)(260, -2.71)(261, -4.04)(262, -4.13)(263, -4.32)(264, -4.17)(265, -4.29)(266, -2.91)(267, -4.15)(268, -2.85)(269, -2.87)(270, -3.13)(271, -3.95)(272, -4.82)(273, -4.61)(274, -4.35)(275, -4.02)(276, -2.87)(277, -3.37)(278, -4.86)(279, -1.39)(280, -4.49)(281, -2.93)(282, -2.00)(283, -2.49)(284, -2.55)(285, -4.51)(286, -2.33)(287, -3.53)(288, -4.03)(289, -3.17)(290, -4.46)(291, -3.03)(292, -4.01)(293, -4.74)(294, -3.21)(295, -3.93)(296, -5.06)(297, -3.72)(298, -2.45)(299, -2.09)(300, -3.87)(301, -4.49)(302, -4.45)(303, -4.24)(304, -5.20)(305, -5.09)(306, -2.91)(307, -4.81)(308, -4.71)(309, -4.73)(310, -3.66)(311, -2.66)(312, -5.11)(313, -4.49)(314, -5.02)(315, -3.93)(316, -3.67)(317, -2.02)(318, -2.89)(319, -2.38)(320, -3.85)(321, -3.45)(322, -3.87)(323, -6.07)(324, -4.81)(325, -3.29)(326, -4.96)(327, -4.98)(328, -4.93)(329, -3.82)(330, -4.56)(331, -4.13)(332, -5.66)(333, -6.37)(334, -5.50)(335, -4.63)(336, -4.22)(337, -4.48)(338, -4.93)(339, -3.84)(340, -6.21)(341, -3.60)(342, -5.95)(343, -4.88)(344, -5.45)(345, -5.57)(346, -4.72)(347, -5.35)(348, -3.33)(349, -4.59)(350, -4.47)(351, -5.84)(352, -4.82)(353, -4.66)(354, -5.53)(355, -4.36)(356, -5.37)(357, -5.06)(358, -5.77)(359, -4.37)(360, -3.29)(361, -3.78)(362, -5.56)(363, -3.71)(364, -5.59)(365, -5.82)(366, -5.82)(367, -4.35)(368, -4.54)(369, -3.64)(370, -4.45)(371, -4.78)(372, -5.77)(373, -5.82)(374, -4.35)(375, -5.76)(376, -4.74)(377, -5.71)(378, -5.31)(379, -5.67)(380, -3.52)(381, -5.31)(382, -5.55)(383, -3.76)(384, -3.62)(385, -2.42)(386, -2.04)(387, -2.79)(388, -2.69)(389, -1.43)(390, -1.42)(391, -2.01)(392, -2.67)(393, -2.71)(394, -3.40)(395, -1.74)(396, -3.96)(397, -2.48)(398, -0.77)(399, -3.17)(400, -1.68)(401, -0.79)(402, -1.39)(403, -1.05)(404, -2.72)(405, -0.87)(406, -3.00)(407, -2.83)(408, -0.64)(409, -0.76)(410, -1.08)(411, 1.00)(412, -1.55)(413, -0.22)(414, -2.14)(415, -2.15)(416, -2.73)(417, -2.06)(418, -1.36)(419, -1.69)(420, -1.45)(421, -0.87)(422, -2.81)(423, -2.45)(424, -1.49)(425, -3.84)(426, -1.15)(427, -2.81)(428, -0.52)(429, -2.30)(430, -2.24)(431, -1.28)(432, -0.98)(433, -1.89)(434, -1.28)(435, -0.84)(436, -2.97)(437, 0.89)(438, -2.08)(439, 0.01)(440, -1.47)(441, -1.84)(442, -0.34)(443, -1.29)(444, 0.35)(445, -2.09)(446, -2.18)(447, -2.82)(448, -1.64)(449, -1.76)(450, -1.88)(451, -2.49)(452, -1.17)(453, -2.03)(454, -2.01)(455, -2.67)(456, -2.18)(457, -0.94)(458, -1.05)(459, -2.35)(460, -3.05)(461, -2.51)(462, -0.79)(463, -1.47)(464, -0.89)(465, -2.88)(466, -1.31)(467, -1.18)(468, -1.95)(469, -2.96)(470, -2.71)(471, 0.25)(472, -2.12)(473, 0.28)(474, -3.12)(475, -0.13)(476, -1.52)(477, 0.31)(478, -3.80)(479, -0.29)(480, -2.63)(481, -1.94)(482, -1.45)(483, -1.14)(484, -1.07)(485, -0.48)(486, -2.91)(487, -1.49)(488, -0.30)(489, -1.94)(490, -1.19)(491, -3.21)(492, -1.13)(493, -3.98)(494, -0.68)(495, -1.36)(496, -2.24)(497, -1.64)(498, -3.40)(499, -1.51)(500, -0.98)(501, -0.98)(502, -1.28)(503, -2.12)(504, 0.11)(505, -2.09)(506, -2.25)(507, -2.70)(508, -0.43)(509, -2.09)(510, -2.26)(511, -2.00)(512, -4.02)}

\def\recbba{(1, 3.22)(2, 2.40)(3, 3.98)(4, 3.30)(5, 2.71)(6, 2.62)(7, 1.54)(8, 3.76)(9, 2.43)(10, 3.67)(11, 1.71)(12, 4.24)(13, 3.00)(14, 1.42)(15, 3.23)(16, 2.80)(17, 2.09)(18, 2.91)(19, 2.21)(20, 2.25)(21, 1.73)(22, 4.34)(23, 3.58)(24, 2.60)(25, 2.82)(26, 4.04)(27, 3.24)(28, 2.02)(29, 2.23)(30, 3.10)(31, 3.51)(32, 3.28)(33, 1.98)(34, 2.16)(35, 2.59)(36, 0.65)(37, 2.08)(38, 3.41)(39, 3.07)(40, 3.83)(41, 2.65)(42, 3.34)(43, 3.31)(44, 2.41)(45, 2.50)(46, 3.00)(47, 3.08)(48, 2.36)(49, 3.30)(50, 2.95)(51, 2.76)(52, 3.94)(53, 3.60)(54, 3.37)(55, 2.30)(56, 2.69)(57, 3.25)(58, 1.90)(59, 2.74)(60, 6.07)(61, 2.84)(62, 3.71)(63, 0.86)(64, 4.61)(65, 3.12)(66, 2.56)(67, 2.45)(68, 4.13)(69, 2.78)(70, 1.35)(71, 1.98)(72, 4.46)(73, 4.07)(74, 1.74)(75, 2.31)(76, 3.97)(77, 5.71)(78, 3.17)(79, 2.97)(80, 2.19)(81, 2.74)(82, 4.78)(83, 2.80)(84, 4.46)(85, 4.31)(86, 5.19)(87, 2.69)(88, 4.50)(89, 5.04)(90, 2.64)(91, 1.81)(92, 3.95)(93, 3.08)(94, 1.98)(95, 4.00)(96, 2.43)(97, 1.77)(98, 4.96)(99, 4.62)(100, 2.79)(101, 3.74)(102, 2.81)(103, 4.03)(104, 2.51)(105, 2.16)(106, 3.10)(107, 3.21)(108, 2.60)(109, 4.47)(110, 3.21)(111, 2.59)(112, 3.24)(113, 2.16)(114, 3.75)(115, 2.97)(116, 4.25)(117, 2.71)(118, 3.40)(119, 4.06)(120, 2.82)(121, 2.12)(122, 3.64)(123, 3.27)(124, 3.52)(125, 3.78)(126, 3.47)(127, 3.01)(128, 2.93)(129, 2.93)(130, 2.73)(131, 3.13)(132, 2.21)(133, 3.01)(134, 3.58)(135, 3.42)(136, 2.60)(137, 2.84)(138, 4.39)(139, 2.21)(140, 2.41)(141, 2.96)(142, 3.11)(143, 1.60)(144, 2.88)(145, 2.44)(146, 3.56)(147, 2.55)(148, 3.04)(149, 3.42)(150, 2.20)(151, 3.55)(152, 2.98)(153, 3.00)(154, 2.08)(155, 3.80)(156, 3.24)(157, 1.75)(158, 3.09)(159, 1.27)(160, 2.65)(161, 0.73)(162, 2.02)(163, 2.33)(164, 3.94)(165, 1.57)(166, 2.41)(167, 1.04)(168, 3.87)(169, 3.09)(170, 4.35)(171, 3.96)(172, 3.53)(173, 2.85)(174, 1.48)(175, 1.68)(176, 4.20)(177, 2.28)(178, 2.65)(179, 2.89)(180, 2.83)(181, 2.58)(182, 4.14)(183, 2.00)(184, 3.57)(185, 3.95)(186, 3.99)(187, 2.12)(188, 2.73)(189, 1.70)(190, 3.43)(191, 1.53)(192, 4.77)(193, 2.35)(194, 1.46)(195, 1.84)(196, 5.28)(197, 2.17)(198, 3.28)(199, 2.73)(200, 4.64)(201, 1.67)(202, 2.92)(203, 0.99)(204, 3.14)(205, 3.85)(206, 3.87)(207, 2.19)(208, 2.92)(209, 2.26)(210, 2.54)(211, 5.04)(212, 2.51)(213, 3.27)(214, 4.76)(215, 2.20)(216, 2.50)(217, 3.55)(218, 2.05)(219, 2.92)(220, 2.57)(221, 2.14)(222, 2.40)(223, 2.76)(224, 3.62)(225, 3.91)(226, 3.67)(227, 4.23)(228, 2.44)(229, 2.60)(230, 2.58)(231, 2.70)(232, 4.94)(233, 3.88)(234, 3.85)(235, 2.60)(236, 3.22)(237, 2.96)(238, 3.32)(239, 3.28)(240, 2.37)(241, 3.84)(242, 4.11)(243, 2.38)(244, 2.57)(245, 2.55)(246, 2.68)(247, 1.50)(248, 1.24)(249, 0.86)(250, 2.42)(251, 0.93)(252, 2.59)(253, 1.98)(254, 2.24)(255, 1.46)(256, 2.54)(257, -1.64)(258, -2.19)(259, -2.71)(260, -2.31)(261, -3.74)(262, -3.84)(263, -4.03)(264, -3.91)(265, -4.15)(266, -2.84)(267, -4.07)(268, -2.81)(269, -2.85)(270, -3.16)(271, -3.96)(272, -4.88)(273, -4.45)(274, -4.23)(275, -3.88)(276, -2.81)(277, -3.29)(278, -4.78)(279, -1.34)(280, -4.44)(281, -2.86)(282, -1.94)(283, -2.47)(284, -2.50)(285, -4.44)(286, -2.28)(287, -3.51)(288, -3.97)(289, -3.26)(290, -4.57)(291, -3.17)(292, -4.12)(293, -4.87)(294, -3.38)(295, -4.09)(296, -5.23)(297, -3.88)(298, -2.67)(299, -2.29)(300, -4.03)(301, -4.65)(302, -4.60)(303, -4.37)(304, -5.36)(305, -5.24)(306, -3.07)(307, -5.01)(308, -4.89)(309, -4.89)(310, -3.82)(311, -2.89)(312, -5.29)(313, -4.67)(314, -5.21)(315, -4.11)(316, -3.83)(317, -2.16)(318, -3.00)(319, -2.48)(320, -4.00)(321, -3.15)(322, -3.53)(323, -5.71)(324, -4.50)(325, -2.98)(326, -4.65)(327, -4.65)(328, -4.62)(329, -3.48)(330, -4.21)(331, -3.76)(332, -5.30)(333, -6.03)(334, -5.11)(335, -4.23)(336, -3.84)(337, -4.12)(338, -4.58)(339, -3.48)(340, -5.88)(341, -3.25)(342, -5.59)(343, -4.50)(344, -5.03)(345, -5.13)(346, -4.29)(347, -4.92)(348, -2.89)(349, -4.16)(350, -4.05)(351, -5.40)(352, -4.45)(353, -4.31)(354, -5.17)(355, -4.03)(356, -5.05)(357, -4.73)(358, -5.42)(359, -4.02)(360, -2.94)(361, -3.43)(362, -5.20)(363, -3.37)(364, -5.24)(365, -5.41)(366, -5.40)(367, -3.94)(368, -4.12)(369, -3.20)(370, -4.03)(371, -4.38)(372, -5.35)(373, -5.36)(374, -3.92)(375, -5.30)(376, -4.26)(377, -5.20)(378, -4.79)(379, -5.11)(380, -2.95)(381, -4.72)(382, -4.95)(383, -3.12)(384, -2.98)(385, -2.77)(386, -2.39)(387, -3.11)(388, -2.99)(389, -1.68)(390, -1.68)(391, -2.27)(392, -2.91)(393, -2.98)(394, -3.70)(395, -1.97)(396, -4.17)(397, -2.71)(398, -1.02)(399, -3.39)(400, -1.86)(401, -1.06)(402, -1.69)(403, -1.29)(404, -2.99)(405, -1.11)(406, -3.29)(407, -3.12)(408, -0.90)(409, -1.04)(410, -1.37)(411, 0.74)(412, -1.84)(413, -0.52)(414, -2.45)(415, -2.46)(416, -2.99)(417, -2.32)(418, -1.62)(419, -1.97)(420, -1.76)(421, -1.18)(422, -3.12)(423, -2.79)(424, -1.82)(425, -4.17)(426, -1.44)(427, -3.13)(428, -0.80)(429, -2.56)(430, -2.50)(431, -1.53)(432, -1.21)(433, -2.17)(434, -1.56)(435, -1.15)(436, -3.30)(437, 0.64)(438, -2.34)(439, -0.23)(440, -1.76)(441, -2.10)(442, -0.58)(443, -1.58)(444, 0.05)(445, -2.41)(446, -2.50)(447, -3.16)(448, -1.96)(449, -2.04)(450, -2.16)(451, -2.78)(452, -1.45)(453, -2.32)(454, -2.28)(455, -2.95)(456, -2.46)(457, -1.23)(458, -1.33)(459, -2.67)(460, -3.37)(461, -2.82)(462, -1.06)(463, -1.75)(464, -1.18)(465, -3.19)(466, -1.59)(467, -1.46)(468, -2.23)(469, -3.25)(470, -3.00)(471, -0.01)(472, -2.40)(473, 0.04)(474, -3.40)(475, -0.36)(476, -1.74)(477, 0.08)(478, -4.12)(479, -0.55)(480, -2.95)(481, -2.27)(482, -1.77)(483, -1.44)(484, -1.37)(485, -0.78)(486, -3.25)(487, -1.81)(488, -0.57)(489, -2.25)(490, -1.48)(491, -3.51)(492, -1.37)(493, -4.28)(494, -0.94)(495, -1.65)(496, -2.52)(497, -1.94)(498, -3.71)(499, -1.80)(500, -1.29)(501, -1.25)(502, -1.55)(503, -2.40)(504, -0.12)(505, -2.32)(506, -2.47)(507, -2.93)(508, -0.60)(509, -2.26)(510, -2.39)(511, -2.08)(512, -4.08)}

\def\recbca{(1, 2.49)(2, 1.57)(3, 3.11)(4, 2.63)(5, 2.04)(6, 2.22)(7, 1.31)(8, 3.23)(9, 2.24)(10, 3.52)(11, 1.72)(12, 4.15)(13, 3.15)(14, 1.88)(15, 3.59)(16, 3.18)(17, 2.66)(18, 3.30)(19, 2.79)(20, 2.83)(21, 2.30)(22, 4.66)(23, 3.72)(24, 2.86)(25, 3.05)(26, 4.33)(27, 3.51)(28, 2.32)(29, 2.75)(30, 3.30)(31, 3.89)(32, 3.67)(33, 2.38)(34, 2.55)(35, 2.85)(36, 0.90)(37, 2.49)(38, 3.68)(39, 3.39)(40, 3.97)(41, 3.15)(42, 3.76)(43, 3.88)(44, 2.76)(45, 2.84)(46, 3.28)(47, 3.33)(48, 2.73)(49, 3.52)(50, 3.19)(51, 3.15)(52, 4.50)(53, 3.82)(54, 3.56)(55, 2.39)(56, 2.89)(57, 3.47)(58, 2.12)(59, 2.94)(60, 6.05)(61, 2.89)(62, 3.65)(63, 0.89)(64, 4.33)(65, 2.89)(66, 2.39)(67, 2.10)(68, 3.61)(69, 2.33)(70, 1.20)(71, 1.71)(72, 4.21)(73, 3.74)(74, 1.42)(75, 2.17)(76, 3.76)(77, 5.65)(78, 3.22)(79, 2.91)(80, 2.30)(81, 2.82)(82, 4.80)(83, 2.90)(84, 4.32)(85, 4.15)(86, 5.10)(87, 2.66)(88, 4.39)(89, 4.75)(90, 2.34)(91, 1.47)(92, 3.23)(93, 2.61)(94, 1.53)(95, 3.28)(96, 1.85)(97, 1.27)(98, 4.26)(99, 3.67)(100, 1.97)(101, 3.02)(102, 2.11)(103, 3.69)(104, 2.28)(105, 1.99)(106, 2.91)(107, 2.94)(108, 2.43)(109, 4.26)(110, 2.95)(111, 2.50)(112, 3.32)(113, 2.32)(114, 3.76)(115, 3.20)(116, 4.39)(117, 2.91)(118, 3.51)(119, 4.23)(120, 2.98)(121, 2.11)(122, 3.42)(123, 3.11)(124, 3.41)(125, 3.64)(126, 3.17)(127, 2.77)(128, 2.74)(129, 2.92)(130, 2.89)(131, 3.36)(132, 2.60)(133, 3.23)(134, 3.82)(135, 3.89)(136, 2.81)(137, 3.04)(138, 4.44)(139, 2.54)(140, 2.64)(141, 3.04)(142, 2.94)(143, 1.65)(144, 2.75)(145, 2.48)(146, 3.54)(147, 2.72)(148, 3.02)(149, 3.23)(150, 2.03)(151, 3.34)(152, 2.82)(153, 2.87)(154, 2.02)(155, 3.77)(156, 3.17)(157, 1.65)(158, 2.85)(159, 1.24)(160, 2.36)(161, 0.60)(162, 1.88)(163, 2.18)(164, 3.68)(165, 1.70)(166, 2.50)(167, 1.37)(168, 4.08)(169, 3.35)(170, 4.54)(171, 4.02)(172, 3.63)(173, 2.99)(174, 1.53)(175, 1.71)(176, 4.16)(177, 2.54)(178, 2.87)(179, 3.12)(180, 2.99)(181, 2.64)(182, 4.09)(183, 1.90)(184, 3.50)(185, 3.92)(186, 3.80)(187, 2.01)(188, 2.52)(189, 1.75)(190, 3.29)(191, 1.40)(192, 4.61)(193, 2.51)(194, 1.61)(195, 2.14)(196, 5.55)(197, 2.83)(198, 3.75)(199, 3.38)(200, 5.03)(201, 2.10)(202, 3.20)(203, 1.35)(204, 3.55)(205, 4.25)(206, 4.09)(207, 2.47)(208, 3.19)(209, 2.37)(210, 2.54)(211, 4.96)(212, 2.36)(213, 3.16)(214, 4.53)(215, 2.20)(216, 2.54)(217, 3.59)(218, 1.96)(219, 3.06)(220, 2.49)(221, 2.00)(222, 2.19)(223, 2.68)(224, 3.36)(225, 3.49)(226, 3.26)(227, 3.71)(228, 1.99)(229, 2.27)(230, 2.11)(231, 2.38)(232, 4.76)(233, 3.82)(234, 3.61)(235, 2.54)(236, 3.11)(237, 2.80)(238, 3.09)(239, 2.98)(240, 1.92)(241, 3.49)(242, 3.40)(243, 1.70)(244, 1.66)(245, 1.66)(246, 1.64)(247, 0.49)(248, 0.02)(249, -0.35)(250, 1.03)(251, -0.74)(252, 0.60)(253, 0.06)(254, 0.18)(255, -0.80)(256, 0.33)(257, 0.89)(258, 0.26)(259, -0.34)(260, -0.17)(261, -1.66)(262, -1.81)(263, -2.02)(264, -2.08)(265, -2.58)(266, -1.51)(267, -2.92)(268, -1.79)(269, -1.93)(270, -2.44)(271, -3.14)(272, -4.10)(273, -3.69)(274, -3.64)(275, -3.16)(276, -2.33)(277, -2.80)(278, -4.22)(279, -1.00)(280, -3.98)(281, -2.23)(282, -1.29)(283, -2.04)(284, -1.95)(285, -3.85)(286, -1.85)(287, -3.17)(288, -3.40)(289, -2.41)(290, -3.69)(291, -2.45)(292, -3.31)(293, -4.15)(294, -2.83)(295, -3.53)(296, -4.69)(297, -3.41)(298, -2.40)(299, -2.08)(300, -3.58)(301, -4.26)(302, -4.19)(303, -3.88)(304, -4.99)(305, -4.94)(306, -2.90)(307, -5.01)(308, -4.78)(309, -4.66)(310, -3.60)(311, -3.06)(312, -5.29)(313, -4.78)(314, -5.24)(315, -4.21)(316, -3.86)(317, -2.33)(318, -2.91)(319, -2.48)(320, -4.14)(321, -3.37)(322, -3.48)(323, -5.75)(324, -4.70)(325, -3.47)(326, -5.19)(327, -5.20)(328, -5.31)(329, -3.93)(330, -4.61)(331, -4.25)(332, -5.72)(333, -6.53)(334, -5.53)(335, -4.66)(336, -4.37)(337, -4.75)(338, -5.16)(339, -4.18)(340, -6.67)(341, -4.02)(342, -6.27)(343, -5.11)(344, -5.44)(345, -5.43)(346, -4.55)(347, -5.22)(348, -3.22)(349, -4.50)(350, -4.38)(351, -5.66)(352, -4.94)(353, -4.92)(354, -5.78)(355, -4.87)(356, -5.87)(357, -5.61)(358, -6.18)(359, -4.75)(360, -3.75)(361, -4.11)(362, -5.80)(363, -4.10)(364, -5.96)(365, -5.92)(366, -5.85)(367, -4.46)(368, -4.55)(369, -3.33)(370, -4.18)(371, -4.63)(372, -5.49)(373, -5.46)(374, -4.09)(375, -5.31)(376, -4.20)(377, -4.99)(378, -4.54)(379, -4.70)(380, -2.49)(381, -4.18)(382, -4.36)(383, -2.48)(384, -2.27)(385, -3.39)(386, -3.02)(387, -3.65)(388, -3.36)(389, -1.95)(390, -2.07)(391, -2.55)(392, -3.04)(393, -3.23)(394, -3.98)(395, -2.04)(396, -4.07)(397, -2.75)(398, -1.06)(399, -3.25)(400, -1.63)(401, -1.00)(402, -1.72)(403, -1.20)(404, -2.91)(405, -0.95)(406, -3.15)(407, -3.04)(408, -0.78)(409, -0.97)(410, -1.33)(411, 0.82)(412, -1.77)(413, -0.47)(414, -2.40)(415, -2.39)(416, -2.61)(417, -2.11)(418, -1.45)(419, -1.88)(420, -1.84)(421, -1.25)(422, -3.15)(423, -3.08)(424, -2.18)(425, -4.34)(426, -1.55)(427, -3.29)(428, -0.82)(429, -2.38)(430, -2.33)(431, -1.35)(432, -0.92)(433, -1.95)(434, -1.39)(435, -1.13)(436, -3.24)(437, 0.93)(438, -1.94)(439, 0.08)(440, -1.59)(441, -1.83)(442, -0.22)(443, -1.41)(444, 0.07)(445, -2.41)(446, -2.43)(447, -3.19)(448, -2.00)(449, -2.11)(450, -2.17)(451, -2.96)(452, -1.63)(453, -2.58)(454, -2.19)(455, -2.93)(456, -2.53)(457, -1.33)(458, -1.45)(459, -2.88)(460, -3.55)(461, -2.95)(462, -1.10)(463, -1.82)(464, -1.27)(465, -3.28)(466, -1.67)(467, -1.56)(468, -2.32)(469, -3.42)(470, -3.18)(471, -0.20)(472, -2.57)(473, -0.11)(474, -3.52)(475, -0.44)(476, -1.75)(477, -0.06)(478, -4.43)(479, -0.92)(480, -3.44)(481, -2.83)(482, -2.27)(483, -1.93)(484, -1.85)(485, -1.33)(486, -3.89)(487, -2.45)(488, -1.08)(489, -2.82)(490, -2.06)(491, -4.02)(492, -1.72)(493, -4.72)(494, -1.37)(495, -2.09)(496, -2.87)(497, -2.25)(498, -3.97)(499, -1.94)(500, -1.50)(501, -1.34)(502, -1.45)(503, -2.24)(504, 0.19)(505, -1.99)(506, -1.95)(507, -2.40)(508, 0.03)(509, -1.56)(510, -1.42)(511, -0.87)(512, -2.72)}

\def\reccaa{(1, 3.24)(2, 2.75)(3, 3.22)(4, 3.98)(5, 3.87)(6, 2.57)(7, 2.15)(8, 3.37)(9, 3.54)(10, 2.67)(11, 2.75)(12, 3.37)(13, 2.23)(14, 3.79)(15, 2.86)(16, 2.10)(17, 2.52)(18, 3.90)(19, 2.54)(20, 2.33)(21, 3.69)(22, 2.93)(23, 3.73)(24, 3.95)(25, 3.61)(26, 2.48)(27, 2.16)(28, 3.59)(29, 3.19)(30, 2.03)(31, 2.06)(32, 2.61)(33, 2.65)(34, 3.02)(35, 2.23)(36, 2.42)(37, 3.21)(38, 2.50)(39, 2.40)(40, 4.96)(41, 2.92)(42, 2.26)(43, 1.68)(44, 3.10)(45, 3.65)(46, 3.09)(47, 2.24)(48, 2.46)(49, 3.82)(50, 2.75)(51, 2.93)(52, 2.30)(53, 3.73)(54, 1.91)(55, 1.25)(56, 2.36)(57, 4.06)(58, 3.18)(59, 4.03)(60, 2.24)(61, 0.77)(62, 2.79)(63, 2.60)(64, 3.85)(65, 4.52)(66, 3.77)(67, 2.83)(68, 2.89)(69, 3.35)(70, 3.07)(71, 2.49)(72, 3.32)(73, 3.85)(74, 4.13)(75, 2.72)(76, 3.24)(77, 3.56)(78, 3.57)(79, 1.18)(80, 2.88)(81, 3.53)(82, 3.15)(83, 3.85)(84, 3.43)(85, 2.66)(86, 4.75)(87, 2.37)(88, 3.90)(89, 2.96)(90, 1.82)(91, 3.22)(92, 3.48)(93, 4.26)(94, 3.28)(95, 3.68)(96, 3.69)(97, 3.00)(98, 3.32)(99, 3.06)(100, 2.59)(101, 2.99)(102, 2.63)(103, 2.61)(104, 2.77)(105, 2.82)(106, 3.54)(107, 3.76)(108, 2.45)(109, 2.05)(110, 3.21)(111, 4.16)(112, 2.67)(113, 3.71)(114, 4.17)(115, 2.85)(116, 3.36)(117, 2.54)(118, 3.00)(119, 3.26)(120, 4.04)(121, 3.25)(122, 2.06)(123, 3.62)(124, 2.88)(125, 1.95)(126, 3.36)(127, 2.34)(128, 2.34)(129, 3.64)(130, 2.89)(131, 3.74)(132, 3.19)(133, 3.61)(134, 3.69)(135, 2.26)(136, 3.18)(137, 3.28)(138, 1.54)(139, 2.78)(140, 2.48)(141, 4.11)(142, 2.45)(143, 2.43)(144, 3.13)(145, 2.59)(146, 3.53)(147, 3.77)(148, 1.55)(149, 2.95)(150, 1.73)(151, 3.16)(152, 3.06)(153, 3.29)(154, 3.70)(155, 2.66)(156, 2.42)(157, 3.21)(158, 1.54)(159, 3.20)(160, 3.48)(161, 2.25)(162, 2.42)(163, 2.27)(164, 3.49)(165, 2.98)(166, 2.66)(167, 1.64)(168, 3.90)(169, 2.43)(170, 3.81)(171, 2.30)(172, 2.95)(173, 3.82)(174, 2.86)(175, 2.24)(176, 2.45)(177, 3.20)(178, 1.91)(179, 2.82)(180, 2.79)(181, 2.18)(182, 2.92)(183, 2.11)(184, 4.22)(185, 1.97)(186, 3.72)(187, 3.29)(188, 3.09)(189, 3.48)(190, 1.93)(191, 1.99)(192, 2.58)(193, 3.36)(194, 3.53)(195, 2.35)(196, 3.33)(197, 1.98)(198, 3.33)(199, 3.62)(200, 2.97)(201, 2.79)(202, 3.40)(203, 4.45)(204, 4.36)(205, 2.54)(206, 2.23)(207, 3.00)(208, 3.84)(209, 2.93)(210, 2.50)(211, 3.61)(212, 3.46)(213, 3.66)(214, 2.58)(215, 4.62)(216, 2.24)(217, 2.45)(218, 2.02)(219, 4.15)(220, 3.59)(221, 3.03)(222, 3.17)(223, 3.41)(224, 3.86)(225, 3.57)(226, 3.41)(227, 2.44)(228, 3.25)(229, 2.75)(230, 2.45)(231, 3.69)(232, 2.08)(233, 3.54)(234, 2.55)(235, 3.23)(236, 4.08)(237, 2.48)(238, 2.35)(239, 3.08)(240, 3.97)(241, 4.58)(242, 3.19)(243, 3.67)(244, 2.39)(245, 2.55)(246, 4.67)(247, 2.99)(248, 1.97)(249, 3.59)(250, 1.93)(251, 4.41)(252, 2.10)(253, 3.39)(254, 2.65)(255, 2.23)(256, 2.54)(257, -2.25)(258, -2.98)(259, -3.31)(260, -3.86)(261, -4.24)(262, -4.06)(263, -2.83)(264, -3.04)(265, -4.12)(266, -2.94)(267, -3.46)(268, -3.93)(269, -4.05)(270, -2.20)(271, -2.49)(272, -2.93)(273, -4.02)(274, -3.02)(275, -3.86)(276, -3.93)(277, -2.78)(278, -3.25)(279, -3.20)(280, -4.60)(281, -2.38)(282, -3.91)(283, -3.68)(284, -3.33)(285, -3.23)(286, -4.02)(287, -4.11)(288, -4.33)(289, -5.28)(290, -4.28)(291, -3.98)(292, -3.77)(293, -4.09)(294, -3.34)(295, -4.86)(296, -4.28)(297, -3.34)(298, -5.24)(299, -2.91)(300, -3.26)(301, -3.70)(302, -3.33)(303, -4.21)(304, -3.27)(305, -4.39)(306, -3.56)(307, -4.66)(308, -3.64)(309, -3.64)(310, -3.92)(311, -4.98)(312, -4.14)(313, -2.89)(314, -3.62)(315, -4.78)(316, -4.20)(317, -4.57)(318, -3.96)(319, -5.01)(320, -4.62)(321, -4.52)(322, -3.23)(323, -5.70)(324, -4.57)(325, -3.42)(326, -4.28)(327, -4.84)(328, -3.81)(329, -5.38)(330, -4.15)(331, -4.66)(332, -5.24)(333, -3.75)(334, -5.06)(335, -4.80)(336, -4.89)(337, -5.20)(338, -6.18)(339, -5.21)(340, -4.68)(341, -3.95)(342, -4.85)(343, -4.85)(344, -4.49)(345, -6.47)(346, -3.84)(347, -5.12)(348, -5.27)(349, -5.19)(350, -5.35)(351, -4.94)(352, -4.68)(353, -4.47)(354, -5.78)(355, -4.16)(356, -3.03)(357, -5.38)(358, -4.17)(359, -6.23)(360, -5.60)(361, -5.27)(362, -4.92)(363, -4.80)(364, -6.16)(365, -4.72)(366, -6.44)(367, -5.36)(368, -5.96)(369, -5.66)(370, -4.41)(371, -4.38)(372, -4.63)(373, -3.67)(374, -4.47)(375, -6.22)(376, -6.23)(377, -4.23)(378, -4.09)(379, -5.09)(380, -5.17)(381, -7.25)(382, -5.48)(383, -4.51)(384, -4.67)(385, -1.95)(386, -1.63)(387, -3.10)(388, -1.37)(389, -1.89)(390, -0.96)(391, -0.92)(392, -2.12)(393, -3.09)(394, -1.99)(395, -2.70)(396, -3.01)(397, -0.45)(398, -2.74)(399, -0.90)(400, -1.82)(401, -2.50)(402, -1.75)(403, -1.38)(404, -1.67)(405, -1.47)(406, -1.44)(407, -3.06)(408, -0.15)(409, -4.21)(410, -1.70)(411, -1.94)(412, -2.11)(413, -0.61)(414, -0.34)(415, -0.71)(416, -3.02)(417, -1.19)(418, -2.76)(419, -0.04)(420, -1.94)(421, -1.12)(422, -2.66)(423, -1.98)(424, -2.19)(425, -1.18)(426, -1.87)(427, -3.30)(428, -0.87)(429, -1.54)(430, -1.50)(431, -3.18)(432, -2.54)(433, -2.38)(434, -2.50)(435, -1.19)(436, -2.07)(437, -3.38)(438, -2.15)(439, -1.52)(440, -2.09)(441, -1.17)(442, -2.48)(443, -1.91)(444, -1.34)(445, -1.35)(446, -0.34)(447, -1.80)(448, -3.42)(449, -2.44)(450, -0.68)(451, -2.32)(452, -1.73)(453, -1.45)(454, -0.26)(455, -1.23)(456, -0.77)(457, -0.37)(458, -2.37)(459, -0.82)(460, -0.59)(461, -1.15)(462, -2.41)(463, -2.86)(464, -1.12)(465, -0.81)(466, -0.61)(467, -2.41)(468, 0.23)(469, -2.34)(470, -0.78)(471, -2.10)(472, -1.60)(473, -1.18)(474, -0.04)(475, -3.24)(476, -1.21)(477, -1.81)(478, -2.19)(479, -2.08)(480, -1.69)(481, -1.65)(482, -1.54)(483, -1.51)(484, -1.53)(485, -1.51)(486, -1.97)(487, -0.82)(488, -0.27)(489, -2.79)(490, 0.60)(491, -1.19)(492, -1.18)(493, -1.49)(494, -1.99)(495, -1.01)(496, -0.65)(497, -1.58)(498, -2.90)(499, 0.68)(500, -2.48)(501, -1.50)(502, -0.79)(503, -1.09)(504, -1.35)(505, -2.38)(506, -0.50)(507, -2.81)(508, -1.94)(509, -1.11)(510, -1.50)(511, -0.95)(512, -3.50)}

\def\reccba{(1, 3.05)(2, 2.66)(3, 3.11)(4, 3.83)(5, 3.91)(6, 2.58)(7, 2.12)(8, 3.47)(9, 3.45)(10, 2.69)(11, 2.83)(12, 3.39)(13, 2.36)(14, 3.82)(15, 2.95)(16, 2.24)(17, 2.68)(18, 3.99)(19, 2.59)(20, 2.43)(21, 3.89)(22, 3.06)(23, 3.76)(24, 3.98)(25, 3.49)(26, 2.56)(27, 2.25)(28, 3.60)(29, 3.12)(30, 2.06)(31, 2.08)(32, 2.59)(33, 2.62)(34, 2.93)(35, 2.15)(36, 2.32)(37, 3.13)(38, 2.49)(39, 2.24)(40, 4.88)(41, 2.84)(42, 2.23)(43, 1.65)(44, 2.97)(45, 3.53)(46, 3.07)(47, 2.26)(48, 2.48)(49, 3.80)(50, 2.79)(51, 2.85)(52, 2.22)(53, 3.58)(54, 1.82)(55, 1.23)(56, 2.27)(57, 3.92)(58, 3.10)(59, 3.96)(60, 2.17)(61, 0.70)(62, 2.75)(63, 2.67)(64, 3.82)(65, 4.46)(66, 3.82)(67, 2.73)(68, 2.91)(69, 3.45)(70, 3.14)(71, 2.48)(72, 3.40)(73, 3.96)(74, 4.23)(75, 2.84)(76, 3.33)(77, 3.63)(78, 3.64)(79, 1.22)(80, 2.89)(81, 3.51)(82, 3.20)(83, 3.80)(84, 3.37)(85, 2.65)(86, 4.76)(87, 2.36)(88, 3.93)(89, 2.98)(90, 1.86)(91, 3.14)(92, 3.40)(93, 4.18)(94, 3.23)(95, 3.65)(96, 3.65)(97, 3.05)(98, 3.31)(99, 3.15)(100, 2.72)(101, 2.98)(102, 2.66)(103, 2.68)(104, 2.78)(105, 2.91)(106, 3.56)(107, 3.77)(108, 2.51)(109, 2.02)(110, 3.33)(111, 4.20)(112, 2.74)(113, 3.75)(114, 4.11)(115, 2.92)(116, 3.40)(117, 2.74)(118, 3.15)(119, 3.39)(120, 4.19)(121, 3.26)(122, 2.15)(123, 3.69)(124, 2.96)(125, 2.04)(126, 3.57)(127, 2.56)(128, 2.56)(129, 3.55)(130, 2.88)(131, 3.67)(132, 3.09)(133, 3.45)(134, 3.45)(135, 2.07)(136, 2.96)(137, 3.26)(138, 1.46)(139, 2.62)(140, 2.40)(141, 4.01)(142, 2.38)(143, 2.50)(144, 3.15)(145, 2.52)(146, 3.46)(147, 3.87)(148, 1.46)(149, 2.93)(150, 1.74)(151, 3.14)(152, 3.07)(153, 3.40)(154, 3.68)(155, 2.58)(156, 2.41)(157, 3.15)(158, 1.59)(159, 3.27)(160, 3.63)(161, 2.31)(162, 2.42)(163, 2.28)(164, 3.46)(165, 2.97)(166, 2.58)(167, 1.58)(168, 3.93)(169, 2.36)(170, 3.83)(171, 2.24)(172, 2.80)(173, 3.76)(174, 2.82)(175, 2.14)(176, 2.31)(177, 3.11)(178, 1.74)(179, 2.84)(180, 2.76)(181, 2.14)(182, 2.82)(183, 1.93)(184, 4.17)(185, 1.92)(186, 3.80)(187, 3.32)(188, 3.14)(189, 3.55)(190, 1.97)(191, 1.98)(192, 2.66)(193, 3.38)(194, 3.60)(195, 2.42)(196, 3.40)(197, 1.93)(198, 3.31)(199, 3.73)(200, 2.98)(201, 2.86)(202, 3.38)(203, 4.49)(204, 4.47)(205, 2.57)(206, 2.27)(207, 3.04)(208, 3.88)(209, 2.96)(210, 2.52)(211, 3.62)(212, 3.45)(213, 3.54)(214, 2.55)(215, 4.70)(216, 2.16)(217, 2.52)(218, 1.87)(219, 4.23)(220, 3.60)(221, 3.03)(222, 3.14)(223, 3.33)(224, 3.85)(225, 3.50)(226, 3.31)(227, 2.42)(228, 3.10)(229, 2.72)(230, 2.46)(231, 3.71)(232, 2.06)(233, 3.57)(234, 2.55)(235, 3.30)(236, 4.09)(237, 2.61)(238, 2.32)(239, 3.01)(240, 4.04)(241, 4.68)(242, 3.31)(243, 3.68)(244, 2.44)(245, 2.63)(246, 4.82)(247, 3.04)(248, 2.04)(249, 3.57)(250, 1.82)(251, 4.35)(252, 1.76)(253, 3.17)(254, 2.08)(255, 1.61)(256, 1.76)(257, -1.48)(258, -2.31)(259, -2.93)(260, -3.52)(261, -4.08)(262, -3.87)(263, -2.59)(264, -3.01)(265, -3.86)(266, -2.80)(267, -3.31)(268, -3.70)(269, -3.93)(270, -2.07)(271, -2.38)(272, -2.76)(273, -4.08)(274, -3.03)(275, -3.85)(276, -3.95)(277, -2.78)(278, -3.17)(279, -3.12)(280, -4.58)(281, -2.29)(282, -3.95)(283, -3.64)(284, -3.31)(285, -3.27)(286, -4.11)(287, -4.14)(288, -4.37)(289, -5.43)(290, -4.42)(291, -4.16)(292, -3.86)(293, -4.16)(294, -3.46)(295, -4.97)(296, -4.35)(297, -3.42)(298, -5.40)(299, -3.00)(300, -3.35)(301, -3.84)(302, -3.46)(303, -4.30)(304, -3.44)(305, -4.53)(306, -3.65)(307, -4.85)(308, -3.84)(309, -3.87)(310, -4.13)(311, -5.13)(312, -4.32)(313, -3.04)(314, -3.79)(315, -4.97)(316, -4.31)(317, -4.79)(318, -4.12)(319, -5.16)(320, -4.82)(321, -4.34)(322, -3.06)(323, -5.68)(324, -4.58)(325, -3.40)(326, -4.23)(327, -4.78)(328, -3.76)(329, -5.35)(330, -4.02)(331, -4.60)(332, -5.21)(333, -3.73)(334, -5.02)(335, -4.76)(336, -4.89)(337, -5.25)(338, -6.21)(339, -5.18)(340, -4.65)(341, -3.94)(342, -4.83)(343, -4.90)(344, -4.39)(345, -6.55)(346, -3.74)(347, -5.06)(348, -5.30)(349, -5.17)(350, -5.33)(351, -4.91)(352, -4.67)(353, -4.52)(354, -5.83)(355, -4.04)(356, -2.86)(357, -5.26)(358, -4.08)(359, -6.21)(360, -5.54)(361, -5.21)(362, -4.81)(363, -4.61)(364, -6.07)(365, -4.51)(366, -6.38)(367, -5.13)(368, -5.78)(369, -5.45)(370, -4.09)(371, -4.14)(372, -4.37)(373, -3.30)(374, -4.06)(375, -5.95)(376, -6.02)(377, -3.88)(378, -3.78)(379, -4.82)(380, -4.86)(381, -7.00)(382, -5.14)(383, -4.09)(384, -4.21)(385, -2.19)(386, -1.89)(387, -3.37)(388, -1.51)(389, -2.19)(390, -1.19)(391, -1.17)(392, -2.40)(393, -3.36)(394, -2.18)(395, -2.93)(396, -3.31)(397, -0.63)(398, -2.98)(399, -1.05)(400, -1.92)(401, -2.56)(402, -1.80)(403, -1.30)(404, -1.70)(405, -1.49)(406, -1.46)(407, -3.14)(408, -0.01)(409, -4.39)(410, -1.72)(411, -1.94)(412, -2.12)(413, -0.55)(414, -0.35)(415, -0.71)(416, -3.17)(417, -1.27)(418, -2.91)(419, -0.07)(420, -2.08)(421, -1.15)(422, -2.72)(423, -1.97)(424, -2.25)(425, -1.20)(426, -1.93)(427, -3.42)(428, -0.79)(429, -1.53)(430, -1.48)(431, -3.22)(432, -2.62)(433, -2.45)(434, -2.54)(435, -1.20)(436, -2.13)(437, -3.56)(438, -2.21)(439, -1.57)(440, -2.17)(441, -1.17)(442, -2.51)(443, -1.96)(444, -1.37)(445, -1.41)(446, -0.33)(447, -1.89)(448, -3.56)(449, -2.49)(450, -0.59)(451, -2.33)(452, -1.76)(453, -1.52)(454, -0.25)(455, -1.21)(456, -0.76)(457, -0.23)(458, -2.37)(459, -0.73)(460, -0.57)(461, -1.15)(462, -2.50)(463, -2.97)(464, -1.13)(465, -0.77)(466, -0.52)(467, -2.46)(468, 0.35)(469, -2.40)(470, -0.70)(471, -2.15)(472, -1.65)(473, -1.23)(474, 0.08)(475, -3.34)(476, -1.21)(477, -1.80)(478, -2.23)(479, -2.08)(480, -1.69)(481, -1.61)(482, -1.51)(483, -1.49)(484, -1.52)(485, -1.53)(486, -1.99)(487, -0.83)(488, -0.30)(489, -2.94)(490, 0.72)(491, -1.21)(492, -1.19)(493, -1.51)(494, -2.04)(495, -1.01)(496, -0.60)(497, -1.60)(498, -3.06)(499, 0.75)(500, -2.58)(501, -1.61)(502, -0.82)(503, -1.12)(504, -1.43)(505, -2.48)(506, -0.46)(507, -2.94)(508, -2.03)(509, -1.11)(510, -1.46)(511, -0.81)(512, -3.46)}

\def\reccca{(1, 3.01)(2, 2.58)(3, 2.93)(4, 3.41)(5, 4.12)(6, 2.73)(7, 2.11)(8, 3.57)(9, 3.00)(10, 2.48)(11, 3.16)(12, 3.49)(13, 2.46)(14, 3.61)(15, 2.90)(16, 2.38)(17, 2.62)(18, 3.73)(19, 2.61)(20, 2.49)(21, 4.05)(22, 3.18)(23, 3.59)(24, 3.97)(25, 3.21)(26, 2.93)(27, 2.57)(28, 3.70)(29, 3.17)(30, 2.44)(31, 2.35)(32, 2.84)(33, 2.90)(34, 3.23)(35, 2.62)(36, 2.62)(37, 3.43)(38, 2.95)(39, 2.40)(40, 4.92)(41, 3.05)(42, 2.58)(43, 1.86)(44, 2.89)(45, 3.32)(46, 3.15)(47, 2.31)(48, 2.51)(49, 3.90)(50, 3.11)(51, 2.63)(52, 1.96)(53, 2.84)(54, 1.39)(55, 1.09)(56, 1.74)(57, 3.13)(58, 2.47)(59, 3.36)(60, 1.92)(61, 0.79)(62, 2.72)(63, 2.86)(64, 3.84)(65, 4.41)(66, 3.97)(67, 2.65)(68, 3.21)(69, 3.74)(70, 3.40)(71, 2.66)(72, 3.64)(73, 4.13)(74, 4.10)(75, 2.98)(76, 3.45)(77, 3.60)(78, 3.69)(79, 1.47)(80, 2.93)(81, 3.31)(82, 3.22)(83, 3.44)(84, 3.05)(85, 2.68)(86, 4.56)(87, 2.41)(88, 3.86)(89, 3.01)(90, 1.96)(91, 2.98)(92, 3.06)(93, 3.94)(94, 3.09)(95, 3.60)(96, 3.37)(97, 2.95)(98, 2.86)(99, 2.89)(100, 2.72)(101, 2.66)(102, 2.41)(103, 2.79)(104, 2.56)(105, 2.71)(106, 3.40)(107, 3.39)(108, 2.41)(109, 1.92)(110, 3.40)(111, 4.25)(112, 3.00)(113, 3.86)(114, 3.93)(115, 3.14)(116, 3.46)(117, 3.31)(118, 3.50)(119, 3.60)(120, 4.18)(121, 2.96)(122, 2.20)(123, 3.61)(124, 3.13)(125, 2.23)(126, 3.94)(127, 2.87)(128, 2.74)(129, 3.68)(130, 3.06)(131, 3.81)(132, 3.20)(133, 3.47)(134, 3.00)(135, 1.81)(136, 2.84)(137, 3.33)(138, 1.61)(139, 2.40)(140, 2.48)(141, 3.97)(142, 2.55)(143, 3.07)(144, 3.62)(145, 2.87)(146, 3.68)(147, 4.10)(148, 1.39)(149, 2.85)(150, 1.80)(151, 2.83)(152, 2.80)(153, 3.26)(154, 3.22)(155, 1.96)(156, 2.16)(157, 2.73)(158, 1.77)(159, 3.48)(160, 3.81)(161, 2.58)(162, 2.50)(163, 2.39)(164, 3.45)(165, 2.84)(166, 2.29)(167, 1.64)(168, 4.02)(169, 2.26)(170, 3.71)(171, 2.39)(172, 2.84)(173, 3.89)(174, 3.12)(175, 2.42)(176, 2.29)(177, 3.11)(178, 1.60)(179, 3.00)(180, 2.76)(181, 2.38)(182, 2.95)(183, 2.04)(184, 4.24)(185, 2.26)(186, 4.17)(187, 3.63)(188, 3.53)(189, 3.92)(190, 2.32)(191, 2.06)(192, 3.12)(193, 3.61)(194, 3.70)(195, 2.51)(196, 3.10)(197, 1.36)(198, 2.69)(199, 3.29)(200, 2.41)(201, 2.53)(202, 2.67)(203, 3.77)(204, 4.15)(205, 2.24)(206, 2.11)(207, 2.91)(208, 3.82)(209, 2.82)(210, 2.47)(211, 3.48)(212, 3.34)(213, 3.30)(214, 2.68)(215, 4.88)(216, 2.20)(217, 2.83)(218, 1.95)(219, 4.28)(220, 3.65)(221, 3.20)(222, 3.20)(223, 3.51)(224, 3.99)(225, 3.74)(226, 3.43)(227, 2.87)(228, 3.11)(229, 3.08)(230, 2.88)(231, 3.93)(232, 2.22)(233, 3.66)(234, 2.60)(235, 3.46)(236, 4.05)(237, 2.86)(238, 2.13)(239, 2.62)(240, 3.89)(241, 4.07)(242, 2.97)(243, 3.12)(244, 2.11)(245, 2.21)(246, 4.31)(247, 2.44)(248, 1.57)(249, 2.91)(250, 1.16)(251, 3.45)(252, 0.52)(253, 1.87)(254, 0.34)(255, 0.07)(256, -0.15)(257, 1.01)(258, -0.04)(259, -1.15)(260, -1.86)(261, -2.69)(262, -2.76)(263, -1.28)(264, -2.36)(265, -3.36)(266, -2.69)(267, -3.40)(268, -3.48)(269, -3.94)(270, -2.32)(271, -2.50)(272, -2.67)(273, -3.92)(274, -2.81)(275, -3.53)(276, -3.72)(277, -2.84)(278, -2.87)(279, -2.86)(280, -4.32)(281, -2.18)(282, -3.95)(283, -3.41)(284, -3.18)(285, -3.27)(286, -4.09)(287, -4.03)(288, -4.15)(289, -5.01)(290, -4.11)(291, -4.00)(292, -3.57)(293, -3.87)(294, -3.37)(295, -4.66)(296, -4.10)(297, -3.36)(298, -5.37)(299, -2.99)(300, -3.40)(301, -3.90)(302, -3.44)(303, -4.00)(304, -3.52)(305, -4.52)(306, -3.58)(307, -4.79)(308, -3.99)(309, -4.11)(310, -4.23)(311, -5.00)(312, -4.41)(313, -3.03)(314, -3.82)(315, -4.93)(316, -4.16)(317, -4.99)(318, -4.29)(319, -5.21)(320, -5.06)(321, -4.25)(322, -3.06)(323, -5.70)(324, -4.92)(325, -3.71)(326, -4.40)(327, -4.78)(328, -3.82)(329, -5.32)(330, -3.87)(331, -4.62)(332, -5.15)(333, -3.82)(334, -4.96)(335, -4.68)(336, -5.20)(337, -5.50)(338, -6.39)(339, -5.50)(340, -5.03)(341, -4.39)(342, -5.12)(343, -5.34)(344, -4.57)(345, -6.98)(346, -4.00)(347, -5.39)(348, -5.88)(349, -5.69)(350, -5.75)(351, -5.38)(352, -5.29)(353, -5.18)(354, -6.28)(355, -4.23)(356, -3.10)(357, -5.30)(358, -4.32)(359, -6.54)(360, -5.73)(361, -5.35)(362, -4.92)(363, -4.60)(364, -6.18)(365, -4.79)(366, -6.86)(367, -4.90)(368, -5.51)(369, -5.45)(370, -3.85)(371, -3.96)(372, -4.14)(373, -2.85)(374, -3.46)(375, -5.20)(376, -5.65)(377, -3.25)(378, -3.21)(379, -4.46)(380, -4.17)(381, -6.27)(382, -4.35)(383, -3.19)(384, -3.11)(385, -2.64)(386, -2.44)(387, -3.90)(388, -1.89)(389, -2.76)(390, -1.75)(391, -1.59)(392, -2.90)(393, -3.78)(394, -2.40)(395, -3.12)(396, -3.70)(397, -1.03)(398, -3.19)(399, -1.28)(400, -1.87)(401, -2.37)(402, -1.72)(403, -0.87)(404, -1.44)(405, -1.35)(406, -1.48)(407, -2.93)(408, 0.28)(409, -4.30)(410, -1.56)(411, -1.74)(412, -2.06)(413, -0.42)(414, -0.47)(415, -0.80)(416, -3.37)(417, -1.28)(418, -2.81)(419, 0.01)(420, -2.08)(421, -0.95)(422, -2.47)(423, -1.54)(424, -2.17)(425, -1.11)(426, -1.81)(427, -3.31)(428, -0.50)(429, -1.29)(430, -1.18)(431, -2.83)(432, -2.50)(433, -2.32)(434, -2.32)(435, -1.35)(436, -2.31)(437, -3.83)(438, -2.39)(439, -1.85)(440, -2.32)(441, -1.15)(442, -2.43)(443, -2.06)(444, -1.52)(445, -1.60)(446, -0.42)(447, -2.06)(448, -3.63)(449, -2.27)(450, -0.22)(451, -1.97)(452, -1.48)(453, -1.61)(454, -0.39)(455, -1.00)(456, -0.64)(457, 0.02)(458, -2.26)(459, -0.72)(460, -0.97)(461, -1.27)(462, -2.69)(463, -3.19)(464, -1.37)(465, -0.76)(466, -0.44)(467, -2.50)(468, 0.38)(469, -2.35)(470, -0.53)(471, -2.19)(472, -1.77)(473, -1.34)(474, 0.31)(475, -3.31)(476, -1.19)(477, -1.63)(478, -2.16)(479, -2.02)(480, -1.67)(481, -1.63)(482, -1.61)(483, -1.63)(484, -1.89)(485, -1.89)(486, -2.23)(487, -1.11)(488, -0.75)(489, -3.42)(490, 0.47)(491, -1.46)(492, -1.43)(493, -1.61)(494, -2.27)(495, -1.37)(496, -0.85)(497, -1.97)(498, -3.64)(499, 0.19)(500, -3.03)(501, -2.23)(502, -1.29)(503, -1.47)(504, -1.83)(505, -2.83)(506, -0.64)(507, -3.17)(508, -2.17)(509, -0.88)(510, -0.82)(511, 0.07)(512, -2.43)}

\newcommand{\graphLl}[1]{%
    \begin{tikzpicture}[xscale=4.4/512, yscale=1/8]
        \draw[ultra thin, black!20] (1,  0)--(512, 0);
        \draw[ultra thin, black!20] (1,-10)--(1,  10);
	\foreach \y in {-5, 0, 5}
	    \node[left, inner sep=0pt] at (-5, \y) {{\tiny \y}};
	\foreach \y in {-10, -5, 0, 5, 10}
	    \draw[ultra thin, black] (-2, \y) -- (4, \y);
        \draw[very thin, green!60!black] plot coordinates {\sol};
        \draw[very thin, blue]   plot coordinates {#1};
    \end{tikzpicture}%
}

\newcommand{\graphL}[1]{%
    \begin{tikzpicture}[xscale=4.4/512, yscale=1/8]
        \draw[ultra thin, black!20] (1,  0)--(512, 0);
        \draw[ultra thin, black!20] (1,-10)--(1,  10);
	\foreach \y in {-10, -5, 0, 5, 10}
	    \draw[ultra thin, black] (-2, \y) -- (4, \y);
        \draw[very thin, green!60!black] plot coordinates {\sol};
        \draw[very thin, blue]   plot coordinates {#1};
    \end{tikzpicture}%
}

  \begin{figure}[tpb]
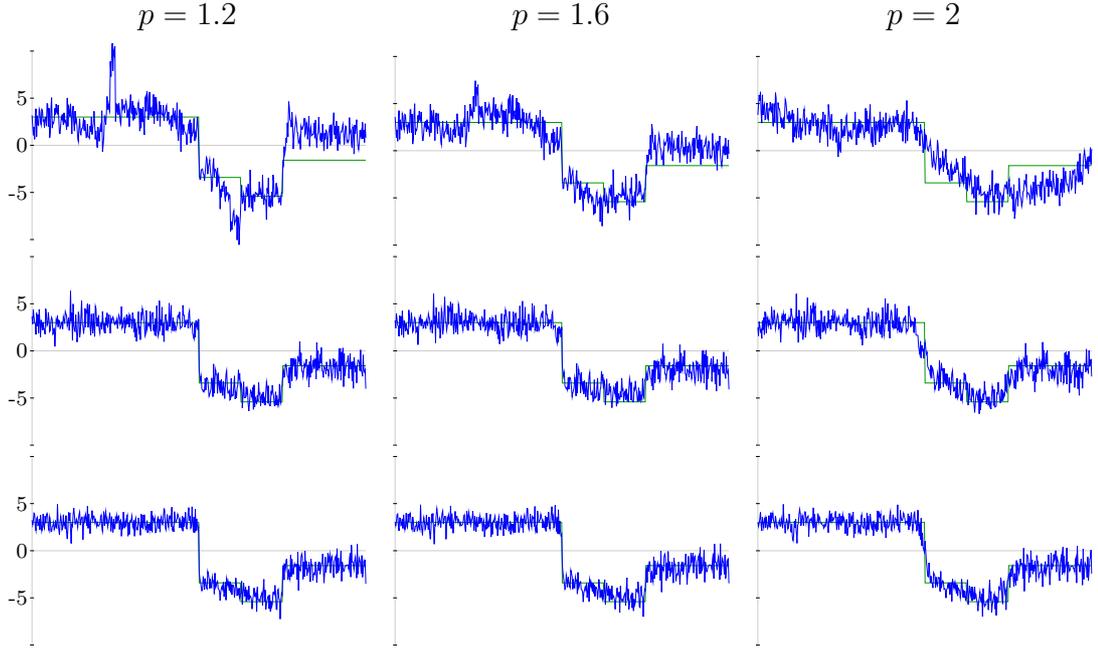

    \begin {center}
    \begin{tabular}{*{3}{c}}
	 $p=1.2$ & $p=1.6$ & $p=2$\\[1mm]
	 \graphLl{\recaaa} & \graphL{\recaba} & \graphL{\recaca}\\
	 \graphLl{\recbaa} & \graphL{\recbba} & \graphL{\recbca}\\
	 \graphLl{\reccaa} & \graphL{\reccba} & \graphL{\reccca}
    \end{tabular}
    \caption{Reconstructions obtained by the modified Landweber method with $\|x_0\| = \start$ and $5\%$ (top), $1\%$ (middle), $0.5\%$ (bottom) noise.}  \label{fig:modlwb}
    \end{center}
  \end{figure}

  In this example, we have adapted an algorithm suggested in \cite{s98} for the Hilbert space situation $p=2$ to $\ell^p$ with $1<p\leq 2$. 
  In the present notation the resulting \emph{dual modified Landweber} algorithm reads
  \begin{equation}\label{C:dual}
  \begin{aligned}
      J_p (x_{k+1}) & = J_p (x_k) - \beta_k \nabla \Phi_{\alpha_k} (x_k)\\
      x_{k+1} & = J_q (J_p (x_{k+1})).
  \end{aligned}
  \end{equation}
   In \cite{s98} no stepsize was used, i.e.~$\beta_k = 1$ for all $k$, but we found that using $\beta_k$ as in \eqref{eq:numbeta} significantly improved the performance. Also the choice of the regularization parameters $\alpha_k$ was slightly adapted introducing an additional factor $\|x_0\|$,
  \begin{equation*}
      \alpha_k = \frac{ \| x_0 \|}{2 ~(k+1000)^{0.99}}.
  \end{equation*}
  In particular for larger values of $\| x_0 \|$, omitting this factor results in underregularized solutions.
  Some reconstructions obtained with the dual modified Landweber method are shown in Figure \ref{fig:modlwb} and further results summarized in Table \ref{tab:modlwb}.
  As above $k^*$ denotes the total number of iterations until the discrepancy principle is satisfied with $\tau = 2$ and $e_{k^*}$ the relative error in the final approximation. Missing values in the table indicate that either the discrepancy principle was not fulfilled after $2 \cdot 10^5$ iterations or that the algorithm stopped in zero. In our experiments with $\| x_0 \| = 10^4$ the dual modified Landweber method never produced an approximation satisfying the discrepancy princple and the corresponding lines are thus omitted in Table \ref{tab:modlwb}.

  \begin{table}[hbpt]
    \begin {center}
    \begin{tabular}{rr*{2}{|ccc}}
      \hline \noalign{\smallskip}
           &           &          \multicolumn{3}{c}{p = 1.2}                 &          \multicolumn{3}{c}{p = 1.6}                \\
  $\delta$ & $\|x_0\|$ &       $\alpha_{k^*}$ &         $k^*$ &     $e_{k^*}$ &       $\alpha_{k^*}$ &         $k^*$ &     $e_{k^*}$\\
  \hline \noalign{\smallskip}
       5\% &         1 &  $1.2 \cdot 10^{-3}$ &           224 &           2.4 &  $1.1 \cdot 10^{-3}$ &           242 &          0.83\\
       5\% &       500 &  $3.1 \cdot 10^{-1}$ &          1380 &          0.36 &  $3.2 \cdot 10^{-1}$ &          1305 &          0.34\\
       5\% &      1000 &                   -- &            -- &            -- &  $3.2 \cdot 10^{-1}$ &          3749 &          0.34\\[2mm]
 
       1\% &         1 &  $8.6 \cdot 10^{-4}$ &           692 &           1.6 &  $6.6 \cdot 10^{-4}$ &          1234 &          0.59\\
       1\% &       500 &  $1.4 \cdot 10^{-2}$ &         59237 &          0.21 &  $1.8 \cdot 10^{-2}$ &         46437 &           0.2\\
       1\% &      1000 &                   -- &            -- &            -- &  $1.8 \cdot 10^{-2}$ &         96354 &           0.2\\[2mm]
 
     0.5\% &         1 &  $2.6 \cdot 10^{-4}$ &          4933 &           1.2 &  $1.9 \cdot 10^{-4}$ &          7092 &          0.48\\
     0.5\% &       500 &  $5.6 \cdot 10^{-3}$ &        160831 &          0.11 &  $6.4 \cdot 10^{-3}$ &        141508 &          0.11\\
     0.5\% &      1000 &                   -- &            -- &            -- &                   -- &            -- &            --\\
  \hline
    \end{tabular}\\[5mm]
    \caption{Results of the mod.~Landweber method with $p=1.2$, $1.6$.} \label{tab:modlwb}
    \end{center}
  \end{table}

  \bigskip

  Comparing the results of Subsections \ref{ssec:numdtigra} and \ref{ssec:numdlwb}, we note that in our experiments with small starting value ($\|x_0\| = 1$) the modified Landweber method tends to achieve the prescribed accuracy in the data misfit after fewer iterations than the dual TIGRA method. The reconstructions in Figure \ref{fig:dtigra} and Figure \ref{fig:modlwb} show, however, that this increase in speed comes at the cost of a lower reconstruction quality. As the initial value gets larger, the dual TIGRA method outperforms the modified Landweber method and proofs to be robust with respect to $\| x_0 \|$.

\bigskip

\section{Proofs of selected results} \label{sec:proofs}

  We collect here the proofs of several of the results presented throughout the previous sections in order to improve the flow of reading ibidem.
  To lighten our formulae, we recall the shorthand notations $R_{\Phi_\alpha} (z,x)$ from \eqref{eq:RPhi} and additionally introduce
  \begin{equation} \label{eq:RF}
    R_{F}(z,x) := F(z) - F(x) - F'(x)(z-x)
  \end{equation}
 for $x, z \in X$.

\subsection{Proof of Proposition~\ref{pro:upd}} \label{ssec:upd}

  First, observe that
    \begin{multline*}
      \frac{1}{2}\|F(\xbd)-F(\xad)\|^2\\
	= \frac{1}{2}\|F(\xbd)-\yd\|^2 - \frac{1}{2}\|F(\xad)-\yd\|^2 + \langle F(\xad)-F(\xbd), F(\xad)-\yd \rangle.
    \end{multline*}
  Then, using that $\xbd$ is a minimizer of $\Phi_{\bar \alpha} (x)$, we find
    \begin{multline*}
      \frac{1}{2}\|F(\xbd)-F(\xad)\|^2 + {\bar \alpha} D_{f_p}(\xbd, \xad)\\
	\begin{aligned}
	  & = \Phi_{\bar \alpha} (\xbd) - \Phi_{\bar \alpha} (\xad) + \langle F(\xad)-F(\xbd), F(\xad)-\yd \rangle  - {\bar \alpha} \langle J_p(\xad), \xbd-\xad \rangle\\
	  & \leq \langle F(\xad)-F(\xbd), F(\xad)-\yd \rangle - {\bar \alpha} \langle J_p(\xad), \xbd-\xad \rangle.
	\end{aligned}
    \end{multline*}
  On the other hand, $\xad$ is a minimizer of $\Phi_\alpha (x)$ and the first order optimality condition reads
  \begin{align*}
    \nabla \Phi_\alpha (\xad) = F'(\xad)^{*}(F(\xad) - \yd) + \alpha J_p (\xad) = 0.
  \end{align*}
  Using the resulting expression for $J_p (\xad)$ as well as Assumption \ref{assp_main} \ref{ass:nl} and Lem\-ma~\ref{lem:normbds} we further obtain
  \begin{multline*}
    \frac{1}{2}\|F(\xbd)-F(\xad)\|^{2}+{\bar \alpha} D_{f_p}( \xbd,\xad)\\
    \begin{aligned}
      & \leq \left\langle F(\xad)-F(\xbd), F(\xad)-\yd\right\rangle +\frac{{\bar \alpha}}\alpha\left\langle F'(\xad)^{*}(F(\xad)-\yd), \xbd-\xad \right\rangle\\
      & = - \left\langle R_F (\xbd, \xad), F(\xad)-\yd\right\rangle + (1-\frac{{\bar \alpha}}\alpha)\langle F'(\xad)(\xad-\xbd), F(\xad)-\yd\rangle\\
      & \leq cD_{f_p}( \xbd,\xad) \|F(\xad)-\yd\| + (1-\frac{{\bar \alpha}}\alpha) K\|\xad-\xbd\|\| F(\xad)-\yd\|.
    \end{aligned}
  \end{multline*}
  Moreover, for $\alpha \geq \astar$ Theorem~\ref{rate0} and Assumption \ref{assp_main} \ref{ass:src} yield
      \[ \|F(\xad)-\yd\| \leq 2 \|\omega\| \alpha + \delta \leq s \varrho \alpha, \]
  whence it follows that
  \begin{multline*}
    \frac{1}{2} \|F(\xbd)-F(\xad)\|^2 + {\bar \alpha} D_{f_p}( \xbd,\xad) \leq \nrho c \varrho \alpha D_{f_p}( \xbd,\xad) + \nrho  \varrho K (\alpha - \bar \alpha) \|\xad-\xbd\|.
  \end{multline*}
  Using $\alpha \leq \frac{{\bar \alpha}}{\bar q_0} =  \frac{1+\nrho c \varrho}{2 \nrho c \varrho} \bar \alpha$ we find
  \begin{equation} \label{eq:Fxadest}
    \frac{1}{2}\|F(\xbd)-F(\xad)\|^{2} + \frac{1 - \nrho c \varrho}{2} {\bar \alpha} D_{f_p}( \xbd,\xad)  \leq \nrho  \varrho K (\alpha - {\bar \alpha}) \|\xad-\xbd\|.
  \end{equation}
  Finally, due to Lemma~\ref{lem:normbds} we may apply \eqref{Dfl} in Lemma~\ref{lemma1} with constant $c_p = c_A = \frac{p-1}{2} (3 A)^{p-2}$ and since $\bar \alpha \geq \alpha^*$ we obtain
    \[ \|\xad-\xbd\|^2 \leq \frac{1}{c_A} D_{f_p}( \xbd,\xad) \leq \frac{2 \nrho \varrho K} {c_A (1 - \nrho c \varrho) \alpha^*} (\alpha - {\bar \alpha}) \|\xad-\xbd\|, \]
  which proves \eqref{eq:xadcont}. Thus, continuity from the right of the mapping $\alpha \mapsto \xad$ readily follows and with \eqref{eq:Fxadest} also of the mapping $\alpha \mapsto F(\xad)$.
\hfill \fbox{}

\bigskip

\subsection{Proof of Theorem~\ref{thm:convex}}  \label{ssec:convex}
  Let $\alpha \geq \astar$ and $h \in X$ with $\|h\|=1$ be fixed. For $t_1,t_2 \in \IR$, we define $x_i = \xad + t_i h$, $i=1,2$, and $\triangle t=t_2-t_1$. Note that $\nabla \Phi_\alpha (x) = F'(x)^* (F(x) - \yd) + \alpha J_p (x)$ and recall that $Y$ is a Hilbert space. One easily verifies,
  \begin{align*}
  \Phi_{\alpha}(x_{2}) = & ~ \frac{1}{2}\|F(x_1) + \triangle t \cdot F'(x_1)h+R_{F}(x_2,x_1)-\yd\|^{2} +\alpha f_p(x_2)\\
    = & ~ \Phi_\alpha (x_1) -\triangle t \cdot \langle \nabla \Phi_{\alpha}(x_1), h \rangle + \frac{1}{2} \triangle t^{2}\|F'(x_1) h\|^{2}+\frac{1}{2}\|R_{F}(x_2,x_1)\|^{2}\\
      & - \langle \yd-F(x_1),R_{F}(x_2,x_1)\rangle + \triangle t \cdot \langle F'(x_1)h,R_{F}(x_2,x_1)\rangle + \alpha  D_{f_p}(x_2,x_1).
  \end{align*}
  Thus,
  \begin{multline*}
    \begin{aligned}
    R_{\Phi_{\alpha}}(x_2,x_1) 
      & = \frac{1}{2}\triangle t^{2} \|F'(x_{1})h\|^{2} + \frac{1}{2}\|R_{F}(x_2,x_1)\|^{2} + \langle F(x_1)-\yd,R_{F}(x_2,x_1)\rangle\\
    \end{aligned}\\
	+ \triangle t \cdot \langle F'(x_1)h,R_{F}(x_2,x_1) \rangle + \alpha D_{f_p}(x_2,x_1)
  \end{multline*}
  and it remains to establish bounds for the terms which may be negative. Due to Assumption \ref{assp_main} \ref{ass:lip} and Lemma~\ref{lem:normbds}, we have
  \begin{align*}
    \langle F'(x_1)h,R_{F}(x_2,x_1)\rangle & \geq - \l( \|F'(x_1) - F'(\xad)\| + \| F'(\xad) \| \r) cD_{f_p}(x_2,x_1)\\
	& \geq - ( L |t_1| + K) cD_{f_p}(x_2,x_1).
  \end{align*}
  and, using also Assumption \ref{assp_main} \ref{ass:nl} and \eqref{eq:lipnl},
  \begin{align*}
    \langle F(x_1)-\yd &, R_F (x_2, x_1) \rangle\\
	& \geq - \l( \| R_F (x_1, \xad) \| + \| F'(\xad) \| \, |t_1| + \| F(\xad) - \yd \|  \r) \| R_F (x_2, x_1) \|\\
	& \geq - \left(\nrho  \varrho \alpha + K |t_1| + \frac{L}{2} |t_1|^2 \right) c D_{f_p}(x_2,x_1).
  \end{align*}
  Collecting these estimates we obtain
  \begin{align*}
    R_{\Phi_{\alpha}}(x_2,x_1) & \geq \left( ( 1 - \nrho c\varrho ) \alpha - \frac{c( 2K + L |t_1|) }{2}|t_{1}|  - c( K + L |t_1| )  |\triangle t|\right)D_{f_p}(x_2,x_1).
  \end{align*}
  For $|t_{2}|,|t_{1}|\leq \cra$, we have $|\triangle t|\leq 2\cra$. Thus, using $2 \gamma = 1 - \nrho c\varrho$,
  \begin{equation*}
    R_{\Phi_{\alpha}}(x_2, x_1) - \gamma \alpha D_{f_p}(x_2,x_1) \geq p(\cra) D_{f_p}(x_2, x_1),
  \end{equation*}
  where $p(r) : = - \frac{5cL}{2} r^{2} - 3 c K r + \gamma \alpha$ has the zeros
  \begin{equation*}
    r_{1,2}(\alpha) = \frac{-3 K \pm \sqrt{9 K^2+10(L/c) \gamma \alpha}}{5L}. 
  \end{equation*}
  Now, if $\bar r_\alpha : = \min\{|r_1|, |r_2|\}$, then
    \begin{align*}
      \bar r_\alpha &= \frac{ \sqrt{9 K^2+10(L/c) \gamma \alpha} -3 K}{5L}  
          = \frac{ 2 \gamma \alpha}{c ~ (\sqrt{9 K^2+10(L/c) \gamma \alpha} +3 K)}\\
          & \geq \frac{2}{1+\sqrt{2}} ~ \cdot \Bigg\{ \begin{array}{ll}
           \frac{\gamma \alpha}{3 c K}  & \quad \mbox{if } ~ 10 \, L \gamma \alpha \leq 9 c K^2, \\[2mm]
           \sqrt{\frac{\gamma  \alpha}{10 c L}} & \quad \mbox{otherwise} \end{array}\\
          & \geq \frac{2}{1+\sqrt{2}} ~ \cdot \min\left\{ \frac{\gamma \alpha}{3 c K}, \sqrt{\frac{\gamma  \alpha}{10 c L}} \right\} = \cra.
    \end{align*}
 Thus,
  \begin{align*}
    R_{\Phi_{\alpha}}(x_2,x_1) -\gamma \alpha D_{f_p}(x_2,x_1) \geq 0
  \end{align*}
  holds for $|t_{2}|,|t_{1}|\leq \cra$ and according to Proposition~\ref{pro:convex} the function $\varphi_{\alpha,h}(t)$ is then convex on $[-\cra, \cra]$, which completes the proof.
\hfill \fbox{}

\bigskip

\subsection{Proof of Lemma~\ref{lem:Dfnorm}} \label{ssec:Dfnorm}

By virtue of Lemma~\ref{lemma1} we may estimate the Bregman distance $D_{f_p} (z,x)$ in terms of the norm $\| x-z \|_{\ell^p}$ not only from above, but also from below.
However, the lower bound involves a number $c_p$ which in fact depends on the elements $x, z$ under consideration or, more precisely, on estimates of the size of $\|x\|$ and on the distance $\|x-z\|$. In this respect the following coercivity result turns out to be particularly useful.

\begin{lemma} \label{lem:Dfcoercive}
 Let $z \in \ell^p$ be arbitrary but fixed, then to every $\brd > 0$ there exists a constant $C_\brd > \| z \|_{\ell^p}$ such that, if $D_{f_p} (z, x) \leq \brd$ holds for $x \in \ell^p$, then $\| x \| \leq C_\brd$. Moreover, the numbers $C_\brd$ may be chosen such that they depend on $\brd$ continuously and are strictly monotonically increasing on $(\| z \|_{\ell^p}, \infty)$.
\end{lemma}

\begin{proof}
  Suppose that $x \in \ell^p$ satisfies $D_{f_p} (z, x) \leq \brd$. Using the first identity in \eqref{eq:Bregd} and $\| J_p (x) \| = \|x\|^{p-1}$ we obtain
    \[ 0 \leq \frac{1}{q} \|x\|^{p} + \frac{1}{p} \|z\|^{p} - \|x \|^{p-1} \| z\| \leq D_{f_p} (z,x) \leq \brd. \]
  Note that the function $f(t):= \frac{1}{q} t^p - \|z\| \cdot t^{p-1} + \frac{1}{p} \|z\|^{p}$ attains its minimum $f(t_0) = 0$ at $t_0 = \|z\|$ and that it is continuous and strictly increasing on $[t_0, \infty)$ with $f(t) \to \infty$ as $t \to \infty$. Thus, for every $\brd >0 $ there exists a number $C_\brd > t_0 = \|z\|$ such that $f(t) > \brd$ for all $t>C_\brd$. Clearly this choice of $C_\brd$ is continuous and strictly increasing and the proof is complete.
\end{proof}

\begin{remark} \label{rem:Dfcoercive}
 From the proof of Lemma~\ref{lem:Dfcoercive} it follows that the number $C_\brd$ only depends on $\| z \|$ (not $z$ itself) and that it increases with $\| z \|$. Thus, for each $R > 0$, $C_\brd = C_\brd(R)$ may be chosen uniformly for all $\| z \| \leq R$.
\end{remark}

\bigskip

{\it Proof of Lemma~\ref{lem:Dfnorm}.}
  If $\brd$ is any positive number and $x \in \ell^p$ satisfies  $D_{f_p} (\xad, x) \leq \brd$, then with $C_\brd$ from Lemma~\ref{lem:Dfcoercive} and $A$ from Lemma~\ref{lem:normbds} we obtain $\| x \| \leq C_\brd$ and $\| x - \xad \| \leq C_d + A$. Note that here the numbers $C_\brd$ are assumed to be independent of $\alpha \geq \astar$ (cf. Remark \ref{rem:Dfcoercive}). Thus, we may apply the lower bound for the Bregman distance from Lemma~\ref{lemma1},
      \[ \| x - \xad \|^2 \leq c_{p,\brd}^{-1}  D_{f_p} (\xad, x) \leq \brd c_{p,\brd}^{-1}, \]
  with constant $c_{p,\brd} := \frac{p-1}{2}(A + 2 C_\brd)^{p-2}$. Observe that, due to the monotonicity and continuity of $C_\brd$ in Lemma~\ref{lem:Dfcoercive}, the quantity $\brd c_{p,\brd}^{-1}$ is monotonically increasing and continuous, and that $\brd c_{p,\brd}^{-1} \to 0$ as $\brd \to +0$. This readily yields the existence of $\bra>0$ such that $D_{f_p} (\xad, x) \leq \bra$ implies
    \[ \| x - \xad \|^2 \leq \bra c_{p,\bra}^{-1} = \cra^2, \]
  and $\bra \to \infty$ follows from $\cra \to \infty$ as $\alpha \to \infty$.
\hfill \endproof

\bigskip

\subsection{Proof of Theorem~\ref{thm:convg_breg}}  \label{ssec:convg_breg}

  Throughout this subsection we assume the outer iteration index $j$ to be fixed. To shorten our formulae we neglect the dependence of the iterates $\xjk$, the step-sizes $\beta_{j,k}$ and the regularization parameter $\alpha_j$ on $j$, and simply write $x_k$, $\beta_k$ and $\alpha$ instead.

  The following Proposition already asserts that the Bregman distance between the minimizer $\xad$ and the iterates $x_k$ decays monotonically if each step-size $\beta_k$ remains below the threshold $\bar \beta_k$. However, these upper bounds $\bar \beta_k$ are given in terms of the searched-for minimizers $\xad$. This dependence is then removed in Theorem~\ref{thm:convg_breg} with the choice \eqref{eq:betak}.

\begin{proposition}\label{closer:0}
  Let Assumption \ref{assp_main} hold and suppose that $D_{f_p} (\xad,x_k) \leq \bra$ with $\bra$ as in Lemma~\ref{lem:Dfnorm} and that $\|\nabla\Phi_{\alpha}(x_{k})\| > 0$. Then the next iterate $x_{k+1}$ from \eqref{eq:lwb} is closer to $\xad$ than $x_{k}$ with respect to the Bregman distance, i.e.
    $$D_{f_p}(\xad,x_{k+1})\leq D_{f_p}(\xad,x_{k}),$$
  if the step-size satisfies $\beta_{k} \in I_k : = (0,\bar \beta_k]$. Here
    \begin{equation} \label{eq:beta0}
     \bar \beta_k : = \frac{\langle\nabla\Phi_{\alpha}(x_{k}),x_{k}-\xad\rangle}{\tilde c_q (\alpha) \| \nabla\Phi_{\alpha} (x_k) \|^2},
    \end{equation}
  where the constant $\tilde c_q (\alpha)$ is given by
    \begin{equation} \label{eq:cqa}
      \tilde c_q (\alpha) : = \max \left\{ 1, \frac{q-1}{2} \bigg( 2 (\cra+A)^{p-1} + \cra \bigg)^{q-2} \right\}
    \end{equation}
  with $\cra$ and $A$ from Theorem~\ref{thm:convex} and Lemma~\ref{lem:normbds}, respectively.
\end{proposition}

\begin{proof}
  By the definition of  Bregman distance $ D_{f_p}(z, x)$ and the dual iteration \eqref{eq:lwb},  we have
    \begin{align*}
      D_{f_p}( \xad, x_{k+1})-D_{f_p}(\xad,x_{k}) & = \frac{1}{p} \|x_k\|^{p}-\frac{1}{p} \|x_{k+1}\|^{p}-\langle J_{p}(x_{k+1})-J_{p}(x_{k}),\xad-x_{k}\rangle\\
	&=D_{f_p}(x_{k},x_{k+1}) - \beta_{k}\langle\nabla\Phi_{\alpha}(x_{k}),x_{k}-\xad\rangle.
    \end{align*}
  As in \eqref{eq:Bregd} we rewrite the first term in the dual space, i.e.
    \[ D_{f_p}(x_{k},x_{k+1}) = D_{f_q}(J_p (x_k), J_p (x_{k+1})). \]
  According to Lemma~\ref{lem:Dfnorm} we have $\| x_k - \xad \| \leq \cra$ due to $D_{f_p} (\xad,x_k) \leq \bra$ and, consequently,
    \[ \beta_k \leq \bar \beta_k \leq \frac{\cra}{\|\nabla\Phi_{\alpha}(x_{k})\|} \]
  holds for $\bar \beta_k$ from \eqref{eq:beta0}. Thus,
  \begin{align*}
    \| J_p (x_k) \| & = \| x_k \|^{p-1} \leq (\cra + A )^{p-1}\\
    \| J_p (x_{k+1}) \| & \leq \| J_p (x_k) \| + \beta_k \| \nabla \Phi_\alpha (x_k) \| \leq (\cra + A )^{p-1} + \cra
  \end{align*}
  and we may apply \eqref{Dfsq} with constant $\tilde c_q = \tilde c_q (\alpha)$ from \eqref{eq:cqa} to obtain
  \begin{equation}\label{Df-phi}
  \begin{aligned}
    D_{f_q}(J_p (x_k), J_p (x_{k+1})) & \leq \tilde{c}_q (\alpha) ~ \|J_{p}(x_{k+1})-J_{p}(x_{k})\|^{2}\\
	  & = \tilde c_q (\alpha) ~ \beta_{k}^{2} ~ \|\nabla\Phi_{\alpha}(x_{k})\|^{2}.
  \end{aligned}
  \end{equation}
  Altogether we have shown
    \[ D_{f_p}( \xad, x_{k+1})-D_{f_p}(\xad,x_{k}) \leq g(\beta_k), \]
  where
  \begin{equation*}
      g(\beta_{k}):=\tilde{c}_{q} (\alpha) \beta_{k}^{2}\|\nabla\Phi_{\alpha}(x_{k})\|^{2}-\beta_{k}\langle\nabla\Phi_{\alpha}(x_{k}),x_{k}-\xad\rangle
  \end{equation*}
  satisfies $g(\beta_{k})<0$ for small values of $\beta_{k}$ as according to Proposition~\ref{pro:vphi} we have $\langle \nabla \Phi_{\alpha}(x_{k}), x_{k}-\xad \rangle = -\varphi'_{1}(0) > 0$. Now, $g(\beta_{k})$ is minimized if
    $$\tilde{c}_{q} (\alpha) \bar \beta_k = \frac{\langle\nabla\Phi_{\alpha}(x_{k}),x_{k}-\xad\rangle}{\|\nabla\Phi_{\alpha}(x_{k})\|^2}$$
  and hence $x_{k+1}$ is closer to $\xad$ than $x_k$ with respect to the Bregman distance for $\beta_k\in (0, \bar \beta_k]$.
\end{proof}

\bigskip

{\it Proof of Theorem~\ref{thm:convg_breg}.}
 The idea of the proof is to establish
  \begin{equation} \label{eq:barck}
    \langle \nabla \Phi_\alpha (x_k), x_k - \xad \rangle \geq \gamma \alpha D_{f_p} (\xad, x_k) \geq \gamma \alpha \bar c_k^2,
  \end{equation}
  so that, consequently, the right hand side in \eqref{eq:beta1} is a lower bound for $\beta_0$ in \eqref{eq:beta0} that is independent of $\xad$. Once this is verified, we may appeal to Proposition~\ref{closer:0} to obtain the assertion. The left hand side inequality in \eqref{eq:barck} has already been proven in Proposition~\ref{pro:vphi} and it remains to show $D_{f_p} (\xad, x_k) \geq \bar c_k^2$. Note that, if $D_{f_p} (\xad, x_k) \geq 1$ this is evidently true. Therefore, we only consider the case $D_{f_p} (\xad, x_k) < 1$ from here on.

  Now, we use the minimizing property of $\xad$ and $\nabla \Phi_\alpha (x_k) \neq 0$ to obtain
  \begin{equation*} 
    0< \Phi_\alpha (x_k) - \phi_{j,k} \leq \Phi_\alpha (x_k) - \Phi_\alpha (\xad).
  \end{equation*}
  To bound the right hand side in terms of $D_{f_p} (x_k, \xad)$ we note that
  \begin{align*}
    \Phi_{\alpha}(x_k) & = \frac{1}{2} \langle F(x_k) - F(\xad), F(x_k) - \yd \rangle +  \frac{1}{2} \langle F(\xad) - \yd, F(x_k) - \yd \rangle + \alpha f_p (x_k)\\
    \Phi_{\alpha}(\xad) & = \frac{1}{2} \langle F(\xad) - F(x_k), F(\xad) - \yd \rangle +  \frac{1}{2} \langle F(x_k) - \yd, F(\xad) - \yd \rangle + \alpha f_p (\xad)
  \end{align*}
  and, for all $x \in \ell^p$,
    \[ \langle \nabla \Phi_{\alpha}(x), \xad - x_k \rangle = \langle F(x) - \yd, F'(x) (\xad - x_k) \rangle + \alpha \langle f_p'(x), \xad - x_k \rangle. \]
 Due to $\nabla \Phi_{\alpha}(\xad)=0$ for the minimizer $\xad$,  the latter yields
  \begin{multline*}
    \begin{aligned}
      \Phi_\alpha (x_k) - \Phi_\alpha & (\xad) = R_{\Phi_\alpha} (\xad, x_k)\\
	      = ~ & \frac{1}{2} \langle F(x_k) - F(\xad), F(x_k) + F(\xad) - 2\yd \rangle\\
	      & \qquad \qquad \qquad \quad + \langle F(\xad) - \yd, F'(\xad) (\xad - x_k) \rangle + \alpha D_{f_p} (x_k, \xad)\\
	      = ~ & \frac{1}{2} \| F(x_k) - F(\xad) \|^2 + \langle F(\xad) - \yd, R_F (x_k, \xad) \rangle + \alpha D_{f_p} (x_k, \xad)\\
        = ~ & \frac{1}{2}\|F'(\xad) (\xad- x_k)\|^{2} + \langle F(\xad)-\yd,R_{F}(x_k,\xad)\rangle +\alpha D_{f_p}(x_k,\xad)
    \end{aligned}\\
    -\langle F'(\xad) (\xad- x_k),R_{F}(x_k,\xad)\rangle + \frac{1}{2}\|R_{F}(x_k,\xad)\|^{2}.
  \end{multline*}
  The Bregman distance is not symmetric and to derive the required estimates, we will bound $D_{f_p}(x_k,\xad)$ in terms of $D_{f_p}(\xad, x_k)$. Using $\| \xad \| \leq A$ according to Lemma~\ref{lem:normbds} and $\| x_k - \xad \| \leq \cra$, this is possible as
    \[ D_{f_p} (x_k, \xad) \leq \tilde c_p \| \xad- x_k \|^p \quad \mbox{and} \quad \| \xad- x_k \|^2 \leq \frac{1}{\bar c_A} D_{f_p} (\xad, x_k)\]
  follows from Lemma~\ref{lemma1} with $\bar c_A:= \frac{p-1}{2} (2 \cra + A)^{p-2}$.
  In combination with the estimates for $\| R_F (x_k, \xad) \|$ and $\| F'(\xad)\|$ in Remark~\ref{rem:nl} and Lemma~\ref{lem:normbds}, respectively, we thus obtain by applying Cauchy's inequality
   \begin{multline*}
    \Phi_\alpha (x_k) - \phi_{j,k} \leq \frac{1}{2 \bar c_A} \left( K^2 + L \| F (\xad) - \yd \| + L K \cra + \frac{L^2 \cra^2}{4} \right) D_{f_p} (\xad, x_k) \\
      + \alpha \frac{ \tilde{c}_p}{\bar c_A^{p/2}}D_{f_p} (\xad, x_k)^{p/2}.
  \end{multline*}
  Finally, for $\alpha \geq \astar$ as in \eqref{eq:astar} we know from \eqref{eq:dpalpha} that $\| F(\xad) - \yd\| \leq \nrho  \varrho \alpha$ and as we only consider the case $D_{f_p} (\xad, x_k) <1$,
     \[ \Phi_\alpha (x_k) - \phi_{j,k} \leq \frac{1}{2 \bar c_A} \left( K^2 + L \nrho  \varrho \alpha + L K \cra + \frac{L^2 \cra^2}{4} + 2 \alpha \tilde{c}_p \bar c_A^{\frac{2-p}{2}} \right) D_{f_p}^{1/2} (\xad, x_k)\]
   holds. Collecting the above estimates,
   we have established
    \[ D_{f_p} (x_k, \xad) \geq \bar c_k^2. \]
  Hence, \eqref{eq:barck} holds true and by virtue of the inital arguments and Proposition~\ref{closer:0}, the proof is complete.
\hfill \endproof

\bigskip

\subsection{Proof of Proposition~\ref{pro:stop}} \label{ssec:stop}

  As in the previous section, we simply write $x_k$, $\beta_k$ and $\alpha$ neglecting the dependence on $j$, which is assumed to be fixed.

\begin{lemma} \label{lem:gradbd}
  If Assumption \ref{assp_main} holds, then to each $\alpha > \astar$ there exists $\kappa_\alpha > 0$ such that $\| \nabla \Phi_\alpha (x) \| \leq \kappa_\alpha$ holds
  for all $x \in X$ satisfying $D_{f_p} (\xad, x) \leq \bra$ with $\bra$ from Lemma~\ref{lem:Dfnorm}.
\end{lemma}

\begin{proof}
  Recall that
      \[ \nabla \Phi_\alpha (x) = F'(x)^* (F(x) - \yd) + \alpha J_p (x) \]
  and that $\| x - \xad \| \leq \cra$ holds with $\cra$ as defined in \eqref{eq:ralpha} due to Lemma~\ref{lem:Dfnorm}. Using Assumption \ref{assp_main}, Theorem~\ref{rate0} as well as Lemma~\ref{lem:normbds}, we obtain
  \begin{align*}
    \| F'(x) \| & \leq \| F'(x) - F'(\xad) \| + \| F'(\xad) \| \leq L \cra + K \\
    \| F(x) - \yd \| & \leq \| R_F (\xad, x) \| + \| F'(x) (\xad - x) \| + \| F(\xad) - \yd \| \\
	& \leq c D_{f_p} (\xad, x) + (L \cra + K) \cra + \nrho  \varrho \alpha\\
    \| x \| & \leq \cra + A,
  \end{align*}
  where $R_F(\xad, x)$ is as defined in \eqref{eq:RF}. Collecting the above estimates, it thus follows that
  \begin{align*}
    \| \nabla \Phi_\alpha (x) \| & \leq \| F'(x) \| \cdot \| F(x) - \yd \| + \alpha \| x \|^{p-1} \leq \kappa_\alpha,
  \end{align*}
  where
      \[ \kappa_\alpha := (L \cra + K) \big( c \bra + (L \cra + K) \cra + \nrho  \varrho \alpha \big) + \alpha \big( \cra + A \big)^{p-1} \]
  and the proof is complete.
\end{proof}

In preparation for the proof of Proposition~\ref{pro:stop} we need another result for the dual gradient descent method.

\begin{proposition} \label{pro:xkt}
  Let Assumption \ref{assp_main} hold and suppose that $D_{f_p} (\xad, x_k) \leq \bra$ holds with $\bra$ from Lemma~\ref{lem:Dfnorm}. Then there exists a constant $M_\alpha >0$, such that for all $\beta \in [0, T_\alpha]$, where
    \[ T_\alpha := \frac{\cra}{\tilde c_q (\alpha) C_\alpha} \]
  with $\cra$, $\tilde c_q (\alpha)$ and $C_\alpha$ as in \eqref{eq:ralpha}, \eqref{eq:cqa} and \eqref{eq:Cabd}, respectively,
  $x_k(\beta)$ defined by
  \begin{align*}
    J_{p}(x_k (\beta)) & =J_{p}(x_{k})-\beta ~ \nabla\Phi_{\alpha}(x_{k}),\\
    x_k (\beta) & =J_{q}(J_{p}(x_k (\beta)))
  \end{align*}
  satisfies
  \begin{equation*}
    R_{\Phi_{\alpha}}(x_k (\beta),x_k)\leq M_\alpha D_{f_p}(x_k(\beta), x_k).
  \end{equation*}
\end{proposition}

{\it Proof.}
  Arguing in a similar way as in the proof of Theorem~\ref{thm:convex}, we get
  \begin{multline*}
    R_{\Phi_{\alpha}}(x_k (\beta),x_k)= \frac{1}{2} \| R_{F} (x_k(\beta),x_k) + F'(x_k)(x_k(\beta) - x_k) \|^2\\
      - \langle R_F (x_k(\beta), x_k), F(x_k) - \yd)\rangle  + \alpha D_{f_p}(x_k(\beta), x_k),
  \end{multline*}
  with $R_F$ and $R_\Phi$ as defined in \eqref{eq:RF} and \eqref{eq:RPhi}, respectively.
  To apply \eqref{Dfl} in Lemma~\ref{lemma1}, we use $\| x_k(\beta) \| = \| J_p (x_k(\beta)) \|^{p-1}$ and $(a+b)^{p-1} \leq a^{p-1} + b^{p-1}$ for $a,b \geq 0$, $p \in (1, 2]$ as well as Lemmata \ref{lem:normbds}, \ref{lem:Dfnorm} and \ref{lem:gradbd} to estimate
  \begin{align*}
    \| x_k \| & \leq \|x_k - \xad \| + \| \xad \| \leq \cra + A\\
    \| x_k(\beta) \| & \leq \Big(\| J_p (x_k) \| + \beta \, \| \nabla \Phi_\alpha (x_k) \|\Big)^{p-1} \leq \cra + A + (T_\alpha \kappa_\alpha)^{p-1}\\
    \| x_k(\beta) -x_k \| & \leq 2 (\cra + A) + (T_\alpha \kappa_\alpha)^{p-1}.
  \end{align*}
  Hence, \eqref{Dfl} yields $c_p (\alpha) ~ \|x_k(\beta) - x_k\|^{2} \leq D_{f_p} (x_k(\beta), x_k)$, but also
  \begin{align*}
    2 c_p (\alpha) ~ \|x_k(\beta) - x_k\|^{2} & \leq D_{f_p} (x_k(\beta), x_k) + D_{f_p} (x_k, x_k(\beta))\\
      & = \langle J_p (x_k (\beta)) - J_{p}(x_k), x_k(\beta) - x_k \rangle \\
      & \leq \beta ~ \|\nabla \Phi_{\alpha} (x_k) \| ~ \| x_k(\beta) - x_k \|,
  \end{align*}
  with constant $c_p (\alpha)$ given by
      \[ c_p (\alpha) = \frac{p-1}{2} \Big( 3 (\cra + A) + 2 (T_\alpha \kappa_\alpha)^{p-1} \Big)^{p-2}. \]
  Due to Lemma~\ref{lem:gradbd}, this further implies
  \begin{align*}
    \|x_k (\beta)-x_{k}\| \leq \frac{T_\alpha \kappa_\alpha}{2 c_p (\alpha)} \qquad \mbox{and} \qquad D_{f_p} (x_k(\beta), x_k) \leq \frac{T_\alpha^2 \kappa_\alpha^2}{2 c_p (\alpha)}.
  \end{align*}
  Therefore, using Assumption \ref{assp_main} and
  \begin{align*}
    \| F'(x_k) \| & \leq  L \cra + K \\
    \| F(x_k) - \yd \| & \leq c \bra + (L \cra + K) \cra + \nrho  \varrho \alpha
  \end{align*}
   as obtained in the proof of Lemma~\ref{lem:gradbd}, we finally estimate
  \begin{align*}
    R_{\Phi_\alpha} (x_k(\beta), x_k) & \leq \| R_F (x_k(\beta), x_k) \|^2  + \| F'(x_k) \| \| x_k(\beta) - x_k \|^2\\
	  & \qquad \qquad + \| R_F (x_k(\beta), x_k) \|  \| F(x_k) - \yd \|  + \alpha D_{f_p}(x_k(\beta), x_k)\\
      & \leq M_\alpha D_{f_p} (x_k(\beta), x_k),
  \end{align*}
  where
    \[ M_\alpha := c^2 \frac{T_\alpha^2 \kappa_\alpha^2}{2 c_p (\alpha)} + \frac{L \cra + K}{c_p (\alpha)} + c (c \bra + (L \cra + K) \cra + \nrho  \varrho \alpha) + \alpha. \qquad \endproof \]

\bigskip

{\it Proof of Proposition~\ref{pro:stop}.}
  We first show that $\Phi_\alpha (x_k)$ is monotonically decreasing with limit $\Phi_0 \geq 0$.
  Let $x_k (\beta)$ be defined as in Proposition~\ref{pro:xkt}. Then $x_{k+1} = x_k (\beta_k)$ and, with $R_{\Phi_\alpha}$ from \eqref{eq:RPhi}, we have
  \begin{align*}
   \Phi_\alpha (x_k (\beta)) - \Phi_\alpha (x_k) & = R_{\Phi_\alpha} (x_k(\beta), x_k) + \langle \nabla \Phi_\alpha (x_k), x_k(\beta) - x_k \rangle\\
	& = R_{\Phi_\alpha} (x_k(\beta), x_k) - \frac{1}{\beta} \langle J_p (x_k(\beta)) - J_p (x_k), x_k(\beta) - x_k \rangle.
  \end{align*}
  Due the Proposition~\ref{pro:xkt} and
  \begin{align*}
    D_{f_p} (x_k (\beta), x_k) & \leq  D_{f_p}(x_k(\beta), x_k) + D_{f_p}(x_k, x_k(\beta))\\
       & = \langle J_p (x_k(\beta))-J_p (x_k),x_k (\beta)-x_k \rangle,
  \end{align*}
  we obtain for $\beta \leq T_\alpha$
  \begin{equation} \label{eq:dPhi}
    \Phi_\alpha (x_k (\beta)) - \Phi_\alpha (x_k) \leq \Big( M_\alpha - \frac{1}{\beta} \Big) D_{f_p}(x_{k}(\beta), x_{k}).
  \end{equation}
  Therefore, $\Phi_\alpha (x_{k+1}) < \Phi_\alpha (x_k)$ holds if $\beta_k < \min \{ T_\alpha, 1/M_\alpha \}$, which is satisfied for the choice $\beta_k$ in \eqref{eq:betak}.

  Moreover, the stepsizes $\beta_k$ are bounded from below, $\beta_k \geq \bar \beta > 0$ for all $k$. Indeed, due to Theorem~\ref{thm:convg_breg} we have
    \[ D_{f_p} (\xad, x_k) \leq D_{f_p} (\xad, x_0) \leq \bra \]
  and thus Lemma~\ref{lem:gradbd} yields
	\[ \beta_k \geq \bar \beta :=  \min \left\{ \frac{\gamma \, \bar c_k \, \alpha}{\kappa_\alpha^2}, \frac{1}{2M_\alpha} \right\} > 0. \]

   Finally, we show that $\| \nabla \Phi_\alpha (x_k) \| \to 0$. To this end, we proceed by contradiction and assume that there exists $\eps > 0$ such that $\| \nabla \Phi_\alpha (x_{k_l}) \| \geq \eps$ holds for some subsequence $\{ k_l \}_{l \in \IN} \subset \IN$.
  Now, we have shown above that the sequence $\{ \Phi_\alpha (x_k) \}_{k \in \IN_0}$ decreases and as it is also bounded from below, it is hence convergent from above to some $\Phi_0 \geq 0$. Thus, there exists $\bar l$ large enough such that $\Phi_\alpha (x_{k_l}) - \Phi_0 \leq  \frac{M_\alpha}{2} c_q \bar \beta^q \eps^q$ holds for all $l \geq \bar l$ with $c_q$ from Lemma~\ref{lemma2}. Using $\bar \beta \leq \beta_k \leq \frac{1}{2 M_\alpha}$ for the step-size selection \eqref{eq:betak}, we obtain from \eqref{eq:dPhi}, \eqref{eq:Bregd} and Lemma~\ref{lemma2}
    \begin{align*}
      \Phi_\alpha (x_{k_l +1}) - \Phi_\alpha (x_{k_l}) & \leq - M_\alpha  D_{f_p} (x_{k_l+1}, x_{k_l}) = - M_\alpha  D_{f_q}(J_p(x_{k_l}), J_p (x_{k_l+1}))\\
	& \leq - M_\alpha c_q \| J_p (x_{k_l+1} - J_p(x_{k_l}) \|^q \leq - M_\alpha c_q \beta_{k_l}^q \| \nabla \Phi_\alpha (x_{k_l}) \|^q \\
	& \leq - M_\alpha c_q \bar \beta^q \eps^q.
    \end{align*}
  This yields
    \begin{align*}
      \Phi_\alpha (x_{k_l +1}) - \Phi_0 & \leq \Phi_\alpha (x_{k_l +1}) - \Phi_\alpha (x_{k_l}) + \Phi_\alpha (x_{k_l}) - \Phi_0\\
	& \leq - \frac{M_\alpha}{2} c_q \bar \beta^q \eps^q,
    \end{align*}
  which contradicts the fact that $\Phi_\alpha (x_k)$ converges to $\Phi_0$ from above.
\hfill \endproof 

\bigskip

\section{Conclusions}
In this paper we have proposed a Banach space version of the TIGRA algorithm  to compute a global minimizer of the Tikhonov-functional with sparsity constraints for nonlinear ill-posed problem.  The new d-TIGRA method applies a dual gradient descent method at decreasing values of the regularization parameter. Using the discprepancy principle as a stopping rule, the algorithm terminates with a regularization parameter $\alpha_\js \geq \astar$, where $\astar$ results from a trade-off between optimal estimates with respect to the Bregman distance and the smallness assumption in the source condition. We have shown convergence of the algorithm under suitable step-size selection and stopping rules, and illustrated the theoretic results with numerical experiments for the autoconvolution problem.

\bigskip

\section*{Acknowledgments}

 \noindent W.~Wang was supported by Zhejiang Provincial NSFC (LQ14A010013), S.~Anzen\-gruber by the German Science Foundation DFG (HO~1454/8-1), R.~Ramlau by the Austrian Science Fund (W1214), and B.~Han by NSFC (91230119). 

\bigskip

\section*{List of constants}

For the convenience of the reader, we collect here values and references for several constants which appear repeatedly throughout the manuscript:

\subsection*{Lower case letters}
    \begin{align*}
	c & & \mbox{Asmp.~\ref{assp_main} \ref{ass:nl}}\\
	c_A & = \frac{p-1}{2} (3 A)^{p-2} & \mbox{Prop.~\ref{pro:upd}} \\
	\bar c_A & = \frac{p-1}{2} (2 \cra + A)^{p-2} & \mbox{Theorem \ref{thm:convg_breg}}\\
	\bar c_{j, k} & = \min \l\{ 1, \frac{2 \bar c_A (\Phi_{\alpha_j}(\xjk)-\phi_{min,k})}{K^{2}+2c \sqrt{\bar c_A} K+c^2 \bar c_A + 4 \bar c_A \alpha_j} \r\} & \mbox{Theorem \ref{thm:convg_breg}}\\
	c_p, \tilde c_p & & \mbox{Lemma \ref{lemma1}} \\
	c_p (\alpha) & = \frac{p-1}{2} \Big( 3 (\cra + A) + 2 (T_\alpha \kappa_\alpha)^{p-1} \Big)^{p-2} & \mbox{Prop.~\ref{pro:xkt}}\\
	c_q, \tilde c_q & & \mbox{Lemma \ref{lemma2}}\\
	\tilde c_q (\alpha) & = \max \left\{ 1, \frac{q-1}{2} \bigg( 2 (\cra+A)^{p-1} + \cra \bigg)^{q-2} \right\} & \mbox{Prop.~\ref{closer:0}}\\
	\bra & & \mbox{Lemma \ref{lem:Dfnorm}} \\
	\qa & & \mbox{Lemma \ref{lem:qa}}\\
	\qa_0 & = \frac{2 \nrho c \varrho}{1 + \nrho c \varrho} & \mbox{Prop.~\ref{pro:upd}}\\
	\cra & = \frac{2}{\nrho c(1+\sqrt{2})} \min\left\{ \frac{2 \gamma \alpha}{3K}, \sqrt{\frac{9c}{10L} \gamma \alpha} \right\} & \mbox{Theorem \ref{thm:convex}} \\
	\nrho & & \mbox{Asmp.~\ref{assp_main} \ref{ass:src}}\\
    \end{align*}

\subsection*{Upper case letters}

    \begin{align*}
	A & = \l( \frac{p}{2 \astar} \r)^{1/p} \| F(0) - \yd \|^{2/p} & \mbox{Lemma \ref{lem:normbds}}\\
	C_{\alpha_j} & \leq \gamma \alpha_j \min \bigg\{ \frac{\br{\ajj}}{3 \crr{\alpha_j}}, \frac{\delta \, \bar c_A}{K + \crr{\alpha_j} (c \, \bar c_A + L)} \bigg\} & \mbox{Prop.~\ref{pro:outk}}\\
	K & = L A + \|F'(0)\| & \mbox{Lemma \ref{lem:normbds}}\\
	L & & \mbox{Asmp.~\ref{assp_main} \ref{ass:lip}}\\
	M_\alpha & = c^2 \frac{T_\alpha^2 \kappa_\alpha^2}{2 c_p (\alpha)} + \frac{L \cra + K}{c_p (\alpha)} + c (c \bra + (L \cra + K) \cra + \nrho  \varrho \alpha) + \alpha & \mbox{Prop.~\ref{pro:xkt}}\\
	T_\alpha & = \frac{\cra}{\tilde c_q (\alpha) C_\alpha} & \mbox{Prop.~\ref{pro:xkt}}\\
    \end{align*}

\subsection*{Greek letters}

    \begin{align*}
	\alpha_0 & & \mbox{Lemma \ref{lem:Dfnorm}}\\
	\astar & = \frac{\delta}{(\nrho - 2) \varrho} & \mbox{\eqref{eq:astar}}\\
	\beta_{j,k} & = \min \left\{\frac{\gamma \; \bar{c}_{j, k} \; \alpha_j}{\|\nabla\Phi_{\alpha_j}(\xjk)\|^2}, \frac{1}{2 M_{\alpha_j}} \right\} & \mbox{Prop.~\ref{pro:stop}}\\
	\gamma & = \frac{1 - sc \varrho}{2} & \mbox{Theorem \ref{thm:convex}}\\
	\phi_{j,k} & = \min\{\Phi_{\alpha_j}(J_q ( J_p (\xjk) + t \nabla\Phi_{\alpha_j} (\xjk)) ):t \in \IR^{+}\} & \mbox{Theorem \ref{thm:convg_breg}}\\
	\kappa_\alpha & = (L \cra + K) \big( c \bra + (L \cra + K) \cra + \nrho  \varrho \alpha \big) + \alpha \big( \cra + A \big)^{p-1} & \mbox{Lemma \ref{lem:gradbd}}\\
	\varrho & & \mbox{Asmp.~\ref{assp_main} \ref{ass:src}}\\
	\rho & = \max \big\{ \tilde c_p^{1/p}, ~ 2 A^{p-1} + \crr{\alpha_0}^{p-1} \big\} & \mbox{Lemma \ref{lem:qa}}\\
	\sigma & = \frac{2 \nrho \varrho K} {c_A (1- \nrho c \varrho) \astar} & \mbox{Prop.~\ref{pro:upd}}\\
	\tau & = 2 + \frac{2}{\qa (s-2)} & \mbox{Asmp.~\ref{ass:tigra} \ref{ass:tau}}
    \end{align*}

\bibliography{refs}
\bibliographystyle{ip}

\end{document}